\documentclass[10pt,letterpaper]{article}

\usepackage[top=0.85in,left=2.75in,footskip=0.75in]{geometry}

\usepackage{mhchem}
\usepackage{colortbl}
\usepackage{amsfonts}
\usepackage{graphicx}
\usepackage{float}
\usepackage{enumerate}
\usepackage{mathrsfs}
\usepackage{color}
\usepackage{mathtools}
\usepackage{breqn}
\usepackage{subfigure}
\usepackage{indentfirst}
\usepackage{arydshln}
\usepackage{cases}
\usepackage{amsthm}
\newtheorem{example}{Example}
\newtheorem{corollary}{Corollary}
\newtheorem{question}{Question}

\newtheorem{theorem}{Theorem}
\newtheorem{lemma}{Lemma}
\newtheorem{definition}{Definition}
\newtheorem{remark}{Remark}

\newtheorem{algorithm}{Algorithm}

\usepackage{tikz}
\usetikzlibrary{fadings}
\usetikzlibrary{shapes,arrows,chains}
\tikzstyle{dec} = [trapezium, trapezium left angle = 80,trapezium right angle=100,minimum width=3cm,minimum height=0.7cm,text centered,draw=black,fill=green!30]
\tikzstyle{inp} = [rectangle,rounded corners, minimum width=3cm,minimum height=0.7cm,text centered, draw=black,fill=red!30]
\tikzstyle{pro} = [rectangle,minimum width=2cm,minimum height=0.7cm,text centered,draw=black,fill=blue!30]
\tikzstyle{arrow} = [thick,->,>=stealth]
\tikzstyle{arrow1} = [thick,-,>=stealth]
\tikzstyle{point}=[coordinate,on grid]

\usepackage{amsmath,amssymb}

\usepackage{changepage}

\usepackage{textcomp,marvosym}

\usepackage{cite}

\usepackage{nameref,hyperref}
\hypersetup{hidelinks,
	colorlinks=true,
	allcolors=black,
	pdfstartview=Fit,
	breaklinks=true}

\usepackage[right]{lineno}

\usepackage[nopatch=eqnum]{microtype}
\DisableLigatures[f]{encoding = *, family = * }

\usepackage{array}

\def\Q{{\mathbb{Q}}}
\newcommand{\defword}{\textit}

\newcolumntype{+}{!{\vrule width 2pt}}

\newlength\savedwidth

\raggedright
\usepackage[aboveskip=1pt,labelfont=bf,labelsep=period,justification=raggedright,singlelinecheck=off]{caption}

\makeatletter
\renewcommand{\@biblabel}[1]{\quad#1.}
\makeatother

\usepackage{lastpage,fancyhdr,graphicx}
\usepackage{epstopdf}
\fancyheadoffset[L]{2.25in}
\fancyfootoffset[L]{2.25in}
\usepackage{hyphenat}
\usepackage{makecell}

\begin{document}
\vspace*{0.2in}

\begin{flushleft}
{\Large
\textbf\newline{Multistability of small zero-one reaction networks} 
}
\newline
\\
Yue Jiao\textsuperscript{}, 
Xiaoxian Tang\textsuperscript{*},
Xiaowei Zeng\textsuperscript{}
\\
\bigskip School of Mathematical Sciences, Beihang University, Beijing, China
\\
\bigskip

%
%





* xiaoxian@buaa.edu.cn

\end{flushleft}
\section*{Abstract}
Zero-one biochemical reaction networks play key roles in cell signalling such as signalling pathways regulated by protein phosphorylation.
Multistability of reaction networks is a crucial dynamics feature enabling decision-making in cells. It is well known that  multistability  can be lifted from a  ``subnetwork" (a  network with less species and fewer reactions) to  large networks. So, we aim to  explore the multistability problem of small zero-one networks. In this work,
we prove the following main results: 1. 
any
zero-one network with a  one-dimensional stoichiometric subspace admits at most one positive steady state (it must be stable), and  all the one-dimensional zero-one networks can be classified  according to if they indeed admit a stable positive steady state or not; 2. any two-dimensional zero-one network with up to three species either admits only degenerate positive steady states, or admits at most one  positive steady state (it must be stable); 3. 
the smallest zero-one networks  (here, by ``smallest", we mean these networks contain species as few as possible) that admit nondegenerate multistationarity/multistability  contain three species and five/six reactions, and they are three-dimensional. 
In these proofs, we use the theorems based on the Brouwer degree theory and the theory of real algebraic geometry. 
Moreover, applying the tools of computational real algebraic geometry, we 
provide a systematical way for detecting the networks that admit nondegenerate multistationarity/multistability.

\section*{Author summary}
This work addresses the challenging problem of detecting the smallest reaction networks that admit more than one  nondegenerate/stable positive steady state,
a property termed nondegenerate multistationarity/multistability. 
In particular, we are interested in the networks arising from the
study of cell signalling, say zero-one networks.  
The main contribution of this work is to show that the smallest zero-one networks that admit nondegenerate multistationarity/multistability contain three species and five/six reactions, and they are three-dimensional.  
  This work also gives insights on the other interesting dynamical features for the small zero-one networks such as dissipativity, degeneracy and absolute concentration  robustness. Also, we provide a computational procedure to detect a multistable network, and by this method we successfully check through over sixty thousands  networks.  
  

\section{Introduction}
For the dynamical systems that arise from biochemical reaction networks, the following question has attracted widespread attention. 
\textcolor{black}{
\begin{question}\label{question}
What is the smallest reaction network that admits at least two stable positive steady states (i.e., multistability)?
\end{question}
}

The multistability problem of biochemical reaction systems \cite{FM1998, BF2001, XF2003, CTF2006}
is a key dynamical feature  linked to switch-like behavior and decision-making process in cellular signalling.
We say a reaction network admits multistability if there exist
rate constants such that the corresponding dynamical system arising under mass-action kinetics
exhibits  at least two stable positive steady states in the same stoichiometric compatibility class.

Deciding the existence of multistability is a challenging  problem in general. 
So far, a typical method is first finding multistationarity (i.e., finding rate constants and total constants such that a given network exhibits at least two positive steady states), and then numerically checking the stability of those steady states (e.g., \cite{OSTT2019}). 
Many typical methods for deciding multistationarity are to check if the determinant of a certain Jacobian matrix changes the sign such as checking injectivity (e.g., \cite{BP, CFMW, CF2005, EF2014}) and the methods based on degree theory (e.g., \cite{CFMW, CG2008, EG2013}). Besides, once multistationarity is found,  we usually need to obtain witness (e.g., \cite{MD2016, PD2012}) and  characterize open regions in the parameters' space for multistationarity (e.g., \cite{BD2020, CFMW, GB2019, HF2013}), for which checking the connectedness of the multistationarity region becomes a crucial problem (e.g., \cite{CFMW, AA2024}).  
Symbolic methods based on computational real algebraic geometry are also successfully applied to  a list of biochemical reaction networks  for detecting multistability \cite{AE2021}. 
However, more explicit  criteria for multistability are still needed because the standard tools based on the classical criteria (e.g.,  Routh-Hurwitz criterion, or alternatively Li\'enard-Chipart criterion  \cite{datta1978}) are computationally difficult. Since there has been a list of nice criteria (e.g., \cite{ ShinarFeinberg2012,CF2012, WiufFeliu_powerlaw, signs, DMST2019}) for determining multistationarity, it is natural to study the relation between the numbers of stable positive steady states and positive steady states. One piece of recent work toward this direction is \cite{TZ2022}, where the one-dimensional case is  explored. But these results can not be easily  extended to the higher dimensional networks. 

Since determining  multistability is not easy, it is a common idea to   study the   small networks since nondegenerate multistationarity and multistability can be lifted from a small ``subnetwork" to the related large networks \cite{JA2013, BP16}.
An ambitious goal is to give explicit descriptions for the multistable networks with small sizes.  
As the first step toward the big goal,  Joshi and Shiu \cite{Joshi:Shiu:Multistationary} provided explicit  criteria  for determining the
multistationarity of the networks with only one species or up to two reactions (possibly reversible).
Later,   these results were extended to nondegenerate multistationarity for two-species networks with two reactions  \cite{shiu-dewolff}.
Especially, Joshi and Shiu \cite{Joshi:Shiu:Multistationary} completely characterized one-species networks by “arrow diagrams”, and these results were also extended to more general one-dimensional networks \cite{PV2022, TL2022}.   
After that, Tang and Xu completely described the smallest multistable bi-reaction networks \cite{TX2021}, where they proved that 
for bi-reaction networks with up to four reactants and up to three species, there are  only two
kinds  of networks that are multistable, and for the networks with
one irreversible and one reversible reaction, if there are at most  three reactants and at most two species, then
only four kinds of networks are multistable \cite{TX2021}. 
Recently, Kaihnsa, Nguyen and Shiu \cite{KN2024} proved that an at-most-bimolecular network admitting both multistationarity and  absolute concentration robustness (ACR, e.g., \cite{PG2024}) has at least three species and three reactions, and  it is at least two-dimensional. \\

\begin{figure}[ht]
    \centering
    \begin{tikzpicture}[scale=2]
        \fill[top color=blue!30, bottom color=blue!40, rounded corners] (-3,-2.04) rectangle (2.8,1.4);
        
        
        \node[fill=white, align=center, text width=0.65cm, minimum height=0.5cm] at (-2.7,1.2) {};
        \node[fill=white, align=center, text width=4.6cm, minimum height=0.5cm] at (-1.4,1.2) {Network};
        \node[fill=white, align=center, text width=1cm, minimum height=0.5cm] at (0,1.2) {};
        \node[fill=white, align=center,  text width=4.55cm, minimum height=0.5cm] at (1.5,1.2) {Reduced Network};
        
        \node[fill=green!30, align=center, minimum width=0.5cm, minimum height=2.695cm] at (-2.75,-1.223){$(b)$};
        \node[fill=white, align=left, minimum width=4.8cm, minimum height=2.2cm] at (-1.4,-1.1) {$S_0+E\rightleftharpoons S_0E\xrightarrow{} S_1+E$\\$S_1+E\rightleftharpoons S_1E\xrightarrow{}S_2+E$\\$
        S_2+F\rightleftharpoons S_2F\xrightarrow{} S_1+F$\\$S_1+F\rightleftharpoons S_1F\xrightarrow{} S_0+F
        $};   
        \node[fill=white, align=center, minimum width=1.5cm, minimum height=2.2cm] at (0.05,-1.1){};
        \draw[line width=1.5pt,->] (-0.1,-1.1) -- (0.1,-1.1);
        \node[fill=white, align=left,  minimum width=4.8cm, minimum height=2.2cm] at (1.5,-1.1) {$S_0+E\xrightarrow{} S_0E\xrightarrow{} S_1+E$\\$S_1+E\xrightarrow{}S_2+E$\\$
        S_2+F\xrightarrow{} S_1+F$\\$S_1+F\xrightarrow{} S_0+F
        $};
        \node[fill=white, align=left,  minimum width=10.6cm, minimum height=0.5cm] at (0.05,-1.75){Mitogenactivated protein kinase (MAPK) cascades \cite{FE2012, AE2021}};

        \node[fill=green!30, align=center, minimum width=0.5cm, minimum height=2.87cm] at (-2.75,0.291){$(a)$};
        \node[fill=white, align=left,  minimum width=4.8cm, minimum height=2.35cm] at (1.5,0.42) {$S_0+E_1\xrightarrow{} S_0E_1\xrightarrow{} S_1+E_1$\\$S_0+E_2\xrightarrow{}S_1+E_2$\\$S_0+E_2\xrightarrow{}S_0E_1
        $\\$
        E_1\xrightarrow{}E_2$\\$S_1\xrightarrow{} S_0
        $
        };
        \node[fill=white, align=center, minimum width=1.5cm, minimum height=2.35cm] at (0.05,0.42){};
        \draw[line width=1.5pt,->] (-0.1,0.42) -- (0.1,0.42);
        \node[fill=white, align=left,  minimum width=4.8cm, minimum height=2.35cm] at (-1.4,0.42) {$S_0+E_1\rightleftharpoons S_0E_1\xrightarrow{} S_1+E_1$\\$S_0+E_2\rightleftharpoons S_0E_2\xrightarrow{}S_1+E_2$\\$S_0E_1\rightleftharpoons S_0E_2
        $\\$
        E_1\rightleftharpoons E_2$\\
        $S_1\xrightarrow{} S_0
        $};
        \node[fill=white, align=left,  minimum width=10.6cm, minimum height=0.5cm] at (0.05,-0.28){ Futile cycle with a two-state kinase \cite{FS2016, AE2021}};
    \end{tikzpicture}
     \vspace{1em}
     \captionsetup{justification=centering} 
    \caption{Three-dimensional zero-one networks admitting multistability.}
    \label{fig:reduced_network}
\end{figure}
In this paper, 
we focus on the reaction networks with
 stoichiometric coefficients zero or one (i.e., zero-one networks).
Our interests in zero-one networks are motivated by the fact that many important biochemical reaction networks in cell signalling are zero-one such as phosphorylation-dephosphorylation cycle \cite{HSP2008, HGB2013, CMS2019}, cell cycle \cite{STJ2007, TJN2022, NTJ2022}, hybrid histidine kinase \cite{JCW2005, KFC2015, SRS2023}, and so on (see more in \cite[Figure 2]{TF2023}, where eleven zero-one networks arising in cell signalling are
listed).
In a related recent work \cite{TW2022},   Tang and Wang  found that the smallest zero-one networks that admit Hopf bifurcations are four-dimensional (they contain four species and five reactions). 

In this work, our main goal is to answer Question \ref{question}  for the zero-one networks, and  
we prove the following main results.
\begin{itemize}
\item[{\bf (1)}]A one-dimensional 
zero-one network admits at most one (stable) positive steady state (this steady state is also called a structural attractor \cite{AA2020}), and 
all the parameters (the total constants) can be completely classified according to if 
the network indeed has a stable positive steady state or not (see Theorem \ref{thm:one-dimensional}). 
\item[{\bf (2)}]A two-dimensional zero-one network with up to  three species either admits no multistationarity, or  admits only degenerate positive steady states (see Theorem \ref{thm:twospecies}). 
\item[{\bf (3)}]For a three-dimensional zero-one network with three species, if it admits nondegenerate multistationarity, then it has at least five reactions, and if it admits multistability, then it has at least six reactions (see Theorem \ref{thm:multistab}). 
\end{itemize}
These results imply that the smallest  zero-one networks that admit nondegenerate multistationarity/multistability contain three species and five/six reactions, and they are three-dimensional (here, by ``smallest", we prioritarily consider the network contains species as few as possible, and we put the other factors such as the number of reactions and the dimension later). 

The main result (1) shows that a one-dimensional zero-one network is structural attractive (e.g., \cite{AA2020}) and it naturally has absolute concentration robustness  if it is full dimensional since it is nondegenerate monostationary.  
From the proof of the first result, we can also see that a one-dimensional zero-one reaction network  is dissipative if it admits a positive steady state, and any positive steady state is nondegenerate.  However, the main result (2) and its proof  show that it is more complicated to understand 
a two-dimensional zero-one network. The main difference is that it is possible for a two dimensional network to only admit degenerate steady states.  We conclude and compare the difference/similarity of the one-dimensional and two-dimensional zero-one networks in Table \ref{tab:my_label}. 

\begin{table}[htbp]
    \centering
    \begin{tabular}{|c|c|c|c|c|}
    \hline
         & \cellcolor{blue!30}Multistability  & \cellcolor{blue!30}\makecell[c]{ACR\\for full-dimension} &  \cellcolor{blue!30}Dissipativity& \cellcolor{blue!30}Degeneracy\\
        \hline
         \cellcolor{green!30}one-dim&  \makecell[c]{No\\ (Theorem \ref{thm:one-dimensional})}&  \makecell[c]{Yes\\ (Remark \ref{rmk:twospecies})}&  \makecell[c]{Yes\\ (Lemma \ref{lm:1 dim J2=emptyset dissipative})}& \makecell[c]{No\\ (Lemma \ref{lm:sum_partical<0})}\\
         \hline
         \cellcolor{green!30}two-dim&  \makecell[c]{No\\ (Theorem \ref{thm:twospecies})}&  \makecell[c]{Uncertain\\ (Remark \ref{rmk:twospecies})}&  \makecell[c]{Uncertain\\ (Remark \ref{re:g1g2 not dissipative})}& \makecell[c]{Yes\\ (Remark \ref{rmk:twospecies}) }\\
         \hline
    \end{tabular}
    \captionsetup{justification=centering} 
    \caption{Comparing  one-dimensional and two-dimensional zero-one networks.}
    \label{tab:my_label}
\end{table}

From the main result (3), it is seen that the smallest multistable zero-one network we have found is three-dimensional, and indeed,  
there  exist biologically meaningful three-dimensional zero-one  networks that admit multistability in applications. For instance, for the futile cycle and the MAPK cascades, one can get three-dimensional ``subnetworks"  by removing some inverse reactions and intermediates,  see Fig \ref{fig:reduced_network}. 
And, it is well-known that the multistability can be lifted from these three-dimensional subnetworks to the original large networks \cite{FS2016, FE2012, AE2021}.  

As what has been said,  the main goal of this work is  to find the smallest zero-one networks that admit multistability. Below, we give a framework of this study and briefly introduce the roadmap of the proofs. 
First, for any one-dimensional zero-one network, 
we prove the following nice properties: 
(a) one can explicitly describle all the total constants such that the corresponding positive stoichiometric compatibility class is non-empty, which form an open connected region (Lemma \ref{lm:Pc+ ne emptyset}); 
(b) if a positive stoichiometric compatibility class is nonempty, then there are no boundary steady states  (Lemma \ref{lm:1 dim no boundary});
(c) a network admitting positive steady states is dissipative (Lemma \ref{lm:1 dim no boundary}); 
(d) a positive steady state is always nondegenerate and stable (Lemma \ref{lm:sum_partical<0}). 
Based on the above properties, we can prove the main result (1) (Theorem \ref{thm:one-dimensional})  by applying Theorem \ref{thm:deter_sign} \cite[Theorem 1]{CFMW}, which is based on the Brouwer degree theory. 

Second, we prove the nondegenerate  monostationarity  for the two-dimensional zero-one networks with up to three species (Theorem \ref{thm:twospecies}). 
For a  two-species network, we prove the nondegenerate   monostationarity   by reducing   the steady-state  system  to a quadratic equation (Lemma \ref{lm:twospecies}).
For the three-species case,
the main idea is to first prove the monostationarity for a class of special networks, say the maximum networks (by ``maximum" we mean the network has a maximum number of reactions when the number of species is fixed as three and the dimension is fixed as two,  see Definition \ref{def:maximum network}), and then extend the result to all networks by the inheritance of nondegenerate  multistationarity \cite[Theorem 1]{BP16}.  
In order to systematically study the maximum networks, we classify them into three classes (later, see \eqref{eq:condition2}--\eqref{eq:condition3}) by unifying  the conservation laws (Lemma \ref{lm:3s get I03}). After that, we find that for two of the three classes, we can not apply Theorem \ref{thm:deter_sign}  since the dissipativity can not be determined by the known criteria  (see Remark \ref{re:g1g2 not dissipative}). 
So, we  apply the theory of real algebraic geometry (Lemma \ref{lm:same number a section}) to complete the proof.  After that,  we
need to check the nondegeneracy of the steady states for all the  three-species networks since we hope to apply the inheritance. Indeed, we figure out by a novel computational method that any two-dimensional three-species zero-one network either only admits degenerate  positive steady states, or if the network admits a positive steady state, then the steady state is nondegenerate (Lemma \ref{lm:jac g23 sub}). 

Third, we pursue the smallest zero-one networks that admit nondegenerate multistationarity/multistability.
From the main results (1) and (2), we know that such networks should be at least three-species, and also, the main result (2) implies that if a  three-species network admits nondegenerate multistationarity, then it   should be at least three-dimensional. Then, by linear algebra and by a known result \cite{BB2023}, it is not difficult to conclude that a nondegenerate multistationary  three-species network has at least five reactions (Theorem \ref{thm:multistab} (I)). Notice that anyone can enumerate 
all three-dimensional three-species  zero-one networks with five reactions that admit nondegerenate positive steady states (there are $65440$ such networks). 
Hence, we get the idea of carrying out a computational procedure to search a multistable network (see Fig \ref{fig:horizontal_flowchart}), where we apply the software  {\tt realrootclassification} \cite{CX2013}. Fortunately, we find all ($429$) three-species five-reaction zero-one networks that  admit nondegenerate  multistationarity (we show one of them in  Example \ref{ex:non_multistability}), which takes about $20$ hours to complete the computation. However, the computation also shows that 
none of these $429$  multistationary  networks admits multistability. 
After that, we  enumerate all ($1367698$)  three-dimensional three-species zero-one networks with  six reactions that admit nondegenerate positive steady states. 
Nevertheless, we can not go over all of these networks by this procedure in a reasonable time.  
However,  we still successfully find a multistable network (see Example \ref{ex:smallest}) during the computation (it takes about $30$ hours), which   confirms that the smallest multistable three-species  zero-one networks have six reactions (see Theorem \ref{thm:multistab} (II)). 

We remark that   in the  proofs of the above main results, we need to face some fundamental/challenging problems in computational real algebraic geometry such as checking  the positivity of a multivariate polynomial and classifying the real solutions of a semi-algebraic system (e.g., \cite{SR06, GP12}). 
More specifically, in Section \ref{subsubsection g1g2g3}, when we prove the monostationarity for the  two-dimensional three-species zero-one networks, a key step is to check if the determinants of the Jacobian matrices of steady-state systems change signs at the positive steady states (i.e., checking the positivity). 
And in Section \ref{sec:results}, when we prove the monostability for the three-dimensional three-species zero-one networks with five reactions,  we need to confirm that each network admits at most one stable positive steady state (i.e., real root classification).  
Since it is difficult to pursue a mathematical proof, 
one can try to apply standard  algebraic tools (for instance, the tools based on the methods of quantifier elimination \cite{AT1998, GE1975, GH1991, CW2001, AT1997})  for tackling these problems. In our case, we need to deal with substantial networks, and these networks may contain lots of parameters (rate constants). So, applying the standard tools might not be realistic due to the huge computational expense \cite{JJ1988}. 
For the first problem here, by studying the special structures of the zero-one networks, we prove a sufficient condition for 
determining  the positivity of the determinant of Jacobian matrix evaluated at a positive steady state  (see Lemma 
\ref{lm:jac B jac}), and accordingly, we develop an efficient algorithm (see Algorithm \ref{alg:checksign}) to finish the computation. 
Based on these computations, we are able to conclude the determinant of Jacobian matrix does not change the sign at a nondegenerate positive steady state for all two-dimensional zero-one networks with three species (see Lemma \ref{lm:jac g23 sub}).  For the second problem, we apply the tool of real root classification, which is originally called    {\tt DISCOVERER} and is developed for a special form of real quantifier elimination \cite{LX2016}.  Later, the tool is renamed as {\tt realrootclassification} and integrated into a package for solving a parametric polynomial system called  {\tt RegularChain} \cite{CX2013}  in the comprehensive algebraic system {\tt Maple} \cite{maple}.

The rest of this paper is organized as follows. In Section \ref{sec:back}, we review the basic notions and the definitions of multistationarity/multistability for the mass-action kinetics systems arising from reaction networks. 
In Section \ref{sec:results}, 
we present the main results (Theorems \ref{thm:one-dimensional}--\ref{thm:multistab}) with several illustrated examples. Since the proof of the last main result (Theorem \ref{thm:multistab}) is based on the first two main results and a computational  procedure for detecting a multistable small network, we also present the proof with the implementation details.      
In Section \ref{sec:methods},  we develop the theories for proving Theorem \ref{thm:one-dimensional} and Theorem \ref{thm:twospecies}. Especially, we present a list of results (including a novel algorithm) on determining the sign of determinant of Jacobian matrix for a zero-one network  in a separate subsection.  
Besides, in Appendix \ref{sec:appendix}, we provide a list of useful supporting materials  including the 
classical criteria  for checking stability (Appendix \ref{S2_Appendix}), the theory of real algebraic geometry (Appendix \ref{subsec:real}), and the enumeration of the maximum zero-one networks (Appendix \ref{S1_Appendix}).
\section{Background}\label{sec:back}

In Section \ref{sec:pre}, we briefly recall the standard notions and definitions for reaction networks, see \cite{CFMW} for more details. In Section \ref{sec:mm}, we present the definitions of multistationarity and multistability.  Also, we review a useful criterion proposed in \cite{CFMW} for multistationarity (see  Theorem \ref{thm:deter_sign}),  which will be used in the proofs of the main results in Section \ref{sec:one dim} and Section \ref{sec:2d 3s}. 
\subsection{Chemical reaction networks}\label{sec:pre}
A \defword{reaction network} $G$  (or \defword{network} for short) consists of a set of $s$ species $\{X_1, X_2, \dots, X_s\}$ and a set of $m$ reactions:
\begin{align}\label{eq:network}
\alpha_{1j}X_1 +
 \dots +
\alpha_{sj}X_s
~ \xrightarrow{\kappa_j} ~
\beta_{1j}X_1 +
 \dots +
\beta_{sj}X_s,
 \;
    {\rm for}~
	j=1,2, \ldots, m,
\end{align}
where all \defword{stoichiometric coefficients}  $\alpha_{ij}$ and $\beta_{ij}$ are non-negative integers, and we  assume that
$(\alpha_{1j},\ldots,\alpha_{sj})\neq (\beta_{1j},\ldots,\beta_{sj})$.
Each $\kappa_j \in \mathbb R_{>0}$ is called a \defword{rate constant} corresponding to the
$j$-th reaction in \eqref{eq:network}.
We say a reaction is a \defword{zero-one reaction}, if the stoichiometric coefficients $\alpha_{ij}$ and $\beta_{ij}$ in \eqref{eq:network}
belong to $\{0,1\}$. 
We say a network \eqref{eq:network} is a \defword{zero-one network} if it only has zero-one reactions. We call the $s\times m$ matrix with
$(i, j)$-entry equal to $\beta_{ij}-\alpha_{ij}$ the
\defword{stoichiometric matrix} of
$G$, denoted by ${\mathcal N}$. We call the $s\times m$ matrix the \defword{reactant matrix} of $G$ with $(i,j)$-entry equal to $\alpha_{ij}$, denoted by ${\mathcal Y}$.
We call the image of ${\mathcal N}$
the \defword{stoichiometric subspace}, denoted by $S$.

We denote by $x_1, \ldots, x_s$ the concentrations of the species $X_1, \ldots, X_s$, respectively.
Under the assumption of mass-action kinetics, we describe how these concentrations change in time by the following system of ODEs:
\begin{align}\label{eq:sys}
\dot{x}\;=\;f(\kappa, x)\;:=\;{\mathcal N}v(\kappa, x)\;=\;{\mathcal N} \begin{pmatrix}
\kappa_1\prod \limits_{i=1}^sx_i^{\alpha_{i1}} \\
\kappa_2\prod \limits_{i=1}^sx_i^{\alpha_{i2}} \\
		\vdots \\
\kappa_m\prod \limits_{i=1}^sx_i^{\alpha_{im}} \\
\end{pmatrix},
\end{align}
where $x=(x_1, x_2, \ldots, x_s)^\top$, $v(\kappa, x)=(v_1(\kappa, x),\dots,v_m(\kappa, x))^\top$ and 
\begin{align}
    v_j(\kappa, x)\;:=\;\kappa_j\prod \limits_{i=1}^sx_i^{\alpha_{ij}}. \nonumber
\end{align}
 By considering the rate constants as a vector of parameters $\kappa:=(\kappa_1, \kappa_2, \dots, \kappa_m)^\top$, we have polynomials $f_{i}(\kappa,x) \in \mathbb Q[\kappa, x]$, for $i\in\{1,\dots, s\}$.

Let $d:=s-{\rm rank}({\mathcal N})$. A \defword{conservation-law matrix} of $G$, denoted by $W$, is any row-reduced $d\times s$ matrix whose rows form a basis of $S^{\perp}$ (note here, ${\rm rank}(W)=d$). Notice that the system~\eqref{eq:sys} satisfies $W \dot x =0$. So, any trajectory $x(t)$ beginning at a non-negative vector $x(0) \in \mathbb{R}^s_{\ge 0}$ remains, for all positive time,
 in the following \defword{stoichiometric compatibility class} w.r.t. (with respect to) the  \defword{total-constant vector} $c:= W x(0) \in {\mathbb R}^d$:
\begin{align}\label{eq:pc}
{\mathcal P}_c\;:=\; \{x\in {\mathbb R}_{\geq 0}^s \mid Wx=c\}.
\end{align}
The \defword{positive stoichiometric compatibility classes} are defined as the relative interior of ${\mathcal P}_c$: 
\begin{align*}
{\mathcal P}_c^+\;:=\; \{x\in {\mathbb R}_{> 0}^s \mid Wx=c\}={\mathcal P}_c\cap \mathbb{R}^s_{> 0}.
\end{align*}
Recall that the conservation-law matrix $W$ is row-reduced. 
Let $I=\{i_1,\dots,i_d\}$ be the indices of the first nonzero coordinates of the rows of $W$, and we assume that $i_1<i_2<\cdots<i_d$. 
Define 
\begin{equation}
h_i\;:=\;
\begin{cases}
f_{i}& \text{ if $ i \notin I $ }, \\
(Wx-c)_{k}& \text{ if $ i=i_k\in I $ },
\end{cases} \nonumber
\end{equation}
where $f_1,\dots,f_s$ are the polynomials defined in \eqref{eq:sys}.
Thus, define the system $h(\kappa,c,x)$ (abbreviated as $h$) as 
\begin{align}\label{eq:h}
    h\;:=\;(h_1,\cdots,h_s),
\end{align}
and we call the system \eqref{eq:h}  \defword{the steady-state system augmented by conservation laws}.
Note that for any $\kappa^*\in \mathbb{R}^m_{>0}$ and for any $c^*\in\mathbb{R}^d$, if $x^*\in \mathbb{R}^s_{\geq 0}$ is a common solution of $h_1=\cdots=h_s=0$, then $x^* $ is a steady state of the network $G$ \eqref{eq:network} in ${\mathcal P}_{c^*}$. Such a steady state
$x^*$ is nondegenerate if the matrix ${\rm Jac}_h(\kappa^*, x^*)$ has full rank, where ${\rm Jac}_h$ denotes the Jacobian matrix of $h$ w.r.t. $x$.\\
We say a reaction network is \defword{dissipative} if for all stoichiometric compatibility classes $\mathcal{P}_c$, there exists a compact set where the trajectories of $\mathcal{P}_c$ eventually enter. 
We say a reaction network is \defword{conservative} if all the stoichiometric compatibility classes are
compact subsets of $\mathbb{R}^n_{\ge0}$.
\begin{lemma}\label{lm:conservative-dissipative}\cite{CFMW}
    A conservative reaction network is dissipative.
\end{lemma}

We define the \defword{infinite norm} of a vector $x\in \mathbb{R}^s$ as
\begin{equation}\label{eq:Vert}
\Vert x\Vert_{\infty}\;:=\;\max\{\vert x_1\vert,\vert x_2\vert,\dots,\vert x_s\vert\}. \nonumber
\end{equation}
\begin{lemma}\label{lm:Dissipative}\cite[Proposition 1]{CFMW}
Consider a network $G$ defined as in \eqref{eq:network}.
Let $f$ be the steady-state system defined as in \eqref{eq:sys}. 
 If for each $c$ with $\mathcal{P}_c^+\neq\emptyset$, there exist a vector $\omega\in\mathbb{R}^s_{>0}$ and a real number $M>0$ such that $\omega\cdot f(x)<0$ (``$\cdot$" means inner product of two vectors) for all $x\in \mathcal{P}_c$ with $\Vert x\Vert_{\infty}>M$, then the network $G$ is dissipative.
\end{lemma}
\subsection{Multistationarity and multistability}\label{sec:mm}
For any given rate-constant vector $\kappa^*\in {\mathbb R}^m_{>0}$,  a \defword{steady state} 
of~\eqref{eq:sys} is a  vector of concentration
$x^* \in \mathbb{R}_{\geq 0}^s$ such that $f(\kappa^*, x^*)=0$, where  $f(\kappa, x)$ is on the
right-hand side of the
ODEs~\eqref{eq:sys}.
If all coordinates of a steady state $x^*$ are strictly positive (i.e., $x^*\in \mathbb{R}_{> 0}^s$), then we call $x^*$ a \defword{positive steady state}.
If a steady state $x^*$ has zero coordinates (i.e., $x^*\in \mathbb{R}_{\geq 0}^s\backslash \mathbb{R}_{> 0}^s$), then we call $x^*$ a \defword{boundary steady state}.
We say a steady state $x^*$ is \defword{nondegenerate} if
${\rm im}\left({\rm Jac}_f (
\kappa^*, x^*)|_{S}\right)=S$,
where ${\rm Jac}_f$ denotes the Jacobian matrix of $f$ w.r.t. $x$.
A steady state $x^*$ is 
\defword{exponentially stable} (or, simply \defword{stable} in this paper)
if the steady state $x^*$ is nondegenerate, and  all non-zero eigenvalues of ${\rm Jac}_f(\kappa^*, x^*)$ have negative real parts.
Note that if a steady state is exponentially stable, then it is locally asymptotically stable \cite{P2001}.



We suppose $N\in {\mathbb Z}_{\geq 0}$. 
We say a  network  \defword{admits $N$ nondegenerate/stable positive steady states}  if there exists a rate-constant
vector $\kappa$ and a total-constant vector $c$ such that  the network has $N$ nondegenerate/stable positive steady states  in ${\mathcal P}_c$.
We say a network  admits \defword{(nondegenerate) multistationarity} if the network  admits at least two (nondegenerate) positive steady states.
We say a network  admits \defword{multistability} if the network admits at least two stable positive steady states. The theorem below is a nice criterion for determining multistationarity, see more details in \cite{CFMW}.
\begin{theorem}\label{thm:deter_sign}\cite[Theorem 1]{CFMW} 
Consider a network $G$ defined as in \eqref{eq:network} with a rank-$r$ stoichiometric matrix. 
Let $h$ be the steady-state system augmented by conservation laws defined as in \eqref{eq:h}. Let $\mathcal{P}_c$ be a stoichiometric compatibility class such that $\mathcal{P}_c^+\neq\emptyset$, where $c\in\mathbb{R}^d$. Suppose a rate-constant vector $\kappa\in {\mathbb R}_{>0}^m$ is given.  Further, assume that
\begin{itemize}
    \item[(i)] the network is dissipative, and 
    \item[(ii)] there are no boundary steady states in $\mathcal{P}_c$. 
\end{itemize}
Then, if 
\begin{align*}
    {\rm sign}({\rm det}({\rm Jac}_h(\kappa, x)))=(-1)^r \quad \text{for all positive steady states}\ x\in \mathcal{P}_c^+,
\end{align*}
then there is exactly one positive steady state in $\mathcal{P}_c$. Moreover, this steady state is nondegenerate.
\end{theorem}

\section{Results}\label{sec:results}
In this section, we present the three main results (Theorems \ref{thm:one-dimensional}--\ref{thm:multistab}) stated in the introduction. For each theorem, we provide some  illustrated examples. Since for proving Theorem \ref{thm:one-dimensional} and Theorem \ref{thm:twospecies}, we need to develop more theories, we will give the details in Section \ref{sec:methods}. However, based on Theorem \ref{thm:one-dimensional} and Theorem \ref{thm:twospecies}, we can prove Theorem \ref{thm:multistab} by a list of known results and tools. Here, we present a computational proof, and we also provide the implementation details of the computational procedure. 

In order to present the first main result, we prepare some notions. Consider a zero-one network $G$ with a rank-one stoichiometric matrix $\mathcal{N}$. 
Without loss of generality, suppose that all row vectors of 
$\mathcal{N}$ are generated by the last row vector
$\mathcal{N}_s$, i.e., 
for any $i \in \{1,\dots,s-1\}$, there exists $a_i\in \mathbb{R}$ such that 
\begin{align}\label{eq:1 dim Ni ai Ns}
    \mathcal{N}_i=a_i\mathcal{N}_s.
\end{align} 
For any $i \in \{1,\dots,s\}$, and for any $j \in \{1,\dots,m\}$, we denote by $\mathcal{N}_{ij}$  the $(i,j)$-th entry of $\mathcal{N}$.
By the definition of zero-one network, we have 
\begin{align}\label{eq:1 dim Nij 01-1}
    \mathcal{N}_{ij}\in\{0,1,-1\}.
\end{align}  
Then, by \eqref{eq:1 dim Ni ai Ns} and \eqref{eq:1 dim Nij 01-1}, for any $i \in \{1,\dots,s-1\}$, 
\begin{align}\label{eq:1 dim ai}
    a_i\in\{0,1,-1\}.
\end{align}
By \eqref{eq:sys} and \eqref{eq:1 dim Ni ai Ns}, for any $i \in \{1,\dots,s-1\}$, we get $f_i=a_if_s$. 
Hence,  the conservation law according to $x_i$ can be written as  $x_i=a_ix_s+c_i$, where $c_i \in {\mathbb R}$. 
Thus, by \eqref{eq:1 dim ai}, 
we can classify the indices of the species by defining three sets as follows: 
\begin{align}
\label{eq:J1}
      \mathcal{J}_1&\;:=\;\{i \mid x_i=x_s+c_i,\;i\in\{1,\dots,s-1\}\}, \\
     \label{eq:J2} \mathcal{J}_2&\;:=\;\{i\mid x_i=-x_s+c_i,\;i\in\{1,\dots,s-1\}\},  \\
      \label{eq:J3}\mathcal{J}_3&\;:=\;\{i\mid x_i=c_i,\;i\in\{1,\dots,s-1\}\}.
\end{align}
\begin{definition}
    A row of a matrix is called a \defword{non-zero row} if there exists a non-zero element in this row. If a row has both positive and negative elements, we say \defword{the row changes signs}.
\end{definition}
\begin{theorem}\label{thm:one-dimensional}
     Consider a zero-one network $G$ with a rank-one stoichiometric matrix $\mathcal{N}$. 
     Let $\mathcal{J}_1$, $\mathcal{J}_2$, and $\mathcal{J}_3$ be defined as in \eqref{eq:J1}--\eqref{eq:J3}. We have the following statements.
    \begin{enumerate}[(I)]
	\item If any non-zero row of the matrix $\mathcal{N}$ does not change the sign, then for any $c\in\mathbb{R}^{s-1}$, and for any $\kappa\in \mathbb{R}^{m}_{>0}$, the network $G$ has no positive steady states  in $\mathcal{P}_{c}$.
	\item If all non-zero rows of the matrix $\mathcal{N}$ change signs, then for any $c\in\mathbb{R}^{s-1} $, we have the following statements.
\begin{itemize}
    \item[(i)]  If there exists $k\in \mathcal{J}_2\cup \mathcal{J}_3$ such that $c_k\leq 0$ or there exists $(i,j)\in\mathcal{J}_1\times\mathcal{J}_2$ such that $c_i+c_j\le 0$, then for any $\kappa\in \mathbb{R}^{m}_{>0}$, the network $G$ has no positive steady states in $\mathcal{P}_{c}$.
    \item[(ii)] If $c_k>0$ for any $k\in \mathcal{J}_2\cup \mathcal{J}_3$, and $c_i+ c_j>0$ for any $(i,j)\in\mathcal{J}_1\times\mathcal{J}_2$, then for any $\kappa\in \mathbb{R}^{m}_{>0}$, the network $G$ has exactly one positive steady state  in $\mathcal{P}_{c}$, and the positive steady state is stable. 
\end{itemize}
\end{enumerate}
\end{theorem}
\begin{remark}
   The hypothesis ``all non-zero rows of $\mathcal{N}$ change signs (or, any non-zero row does not change the sign)" in Theorem \ref{thm:one-dimensional}  is equivalent to 
   ``at least one non-zero row of $\mathcal{N}$  changes signs (or, at least one non-zero row does not change the sign)" since the network $G$ is one-dimensional. 
\end{remark}
\begin{remark} 
 Suppose that a one-dimensional zero-one network $G$ satisfies the hypothesis of Theorem \ref{thm:one-dimensional} (II).
We remark that if $\mathcal{J}_2 \neq \emptyset$ or $\mathcal{J}_3\neq \emptyset$, then both of the two cases (i) and (ii) indeed happen. And if $\mathcal{J}_2 =\mathcal{J}_3=\emptyset$, then the case (i) can not happen, and the case (ii) must happen.
\end{remark}
\begin{remark}
All the total-constant vectors satisfying  the hypothesis of Theorem \ref{thm:one-dimensional} (II) (ii) form  the following 
   region $$\{c\in \mathbb{R}^{s-1}\mid c_k>0, \text{ for any }k\in \mathcal{J}_2\cup \mathcal{J}_3, \text{ and } c_i+ c_j>0, \text{ for any }(i,j)\in\mathcal{J}_1\times\mathcal{J}_2\}.$$
   Notice that this region is connected in $\mathbb{R}^{s-1}$. Later, we will see from  Lemma \ref{lm:Pc+ ne emptyset} that a total-constant vector $c$ belongs to the above region is equivalent to ${\mathcal P}^+_c\neq \emptyset$. 
\end{remark}
\begin{example}
Consider the following network $G$:
\begin{equation*}
    X_1\xrightarrow{\kappa_1} X_2.
\end{equation*}
The stoichiometric matrix $\mathcal{N}$ is 
\begin{equation*}
    \begin{pmatrix}
        -1\\
        1\
    \end{pmatrix}.
\end{equation*}
Notice that any non-zero row of $\mathcal{N}$ does not change the sign since there is a single column. By Theorem \ref{thm:one-dimensional} (I), for any $c\in\mathbb{R}$ and for any $\kappa_1 \in \mathbb{R}_{>0}$, the network $G$ has no positive steady states in $\mathcal{P}_{c}$.
\end{example}

\begin{example}\label{ex:13}
Consider the following network $G$:
  \begin{equation}\label{eq:network13}
    X_1\xrightarrow{\kappa_1} X_2+X_3,\quad X_2+X_3\xrightarrow{\kappa_2} X_1. \nonumber
\end{equation}
The stoichiometric matrix $\mathcal{N}$ is 
\begin{equation*}
    \begin{pmatrix}
        -1&1\\
        1&-1\\
        1&-1
    \end{pmatrix}. 
\end{equation*}
The conservation laws are
\begin{equation*}
    x_1=-x_3+c_1,\text{ and }\; x_2=x_3+c_2.
\end{equation*}
Notice that all the non-zero rows  of $\mathcal{N}$ change signs.
By Theorem \ref{thm:one-dimensional} (II), we know that if $c_1\leq 0$ or $c_1+c_2\le 0$, then for any $\kappa\in\mathbb{R}^2_{>0}$, the network $G$ has no positive steady states in $\mathcal{P}_{c}$. If $c_1>0$ and $c_1+ c_2>0$, then for any $\kappa\in\mathbb{R}^2_{>0}$, the network $G$ has exactly one  positive steady state in $\mathcal{P}_{c}$, and the steady state is stable (see Fig \ref{fig:th2}).
\end{example}

\begin{figure}[htbp]
\centering
\subfigure[The parameter region]
{
    \begin{minipage}[b]{.45\linewidth}
        \centering
        \includegraphics[scale=0.43]{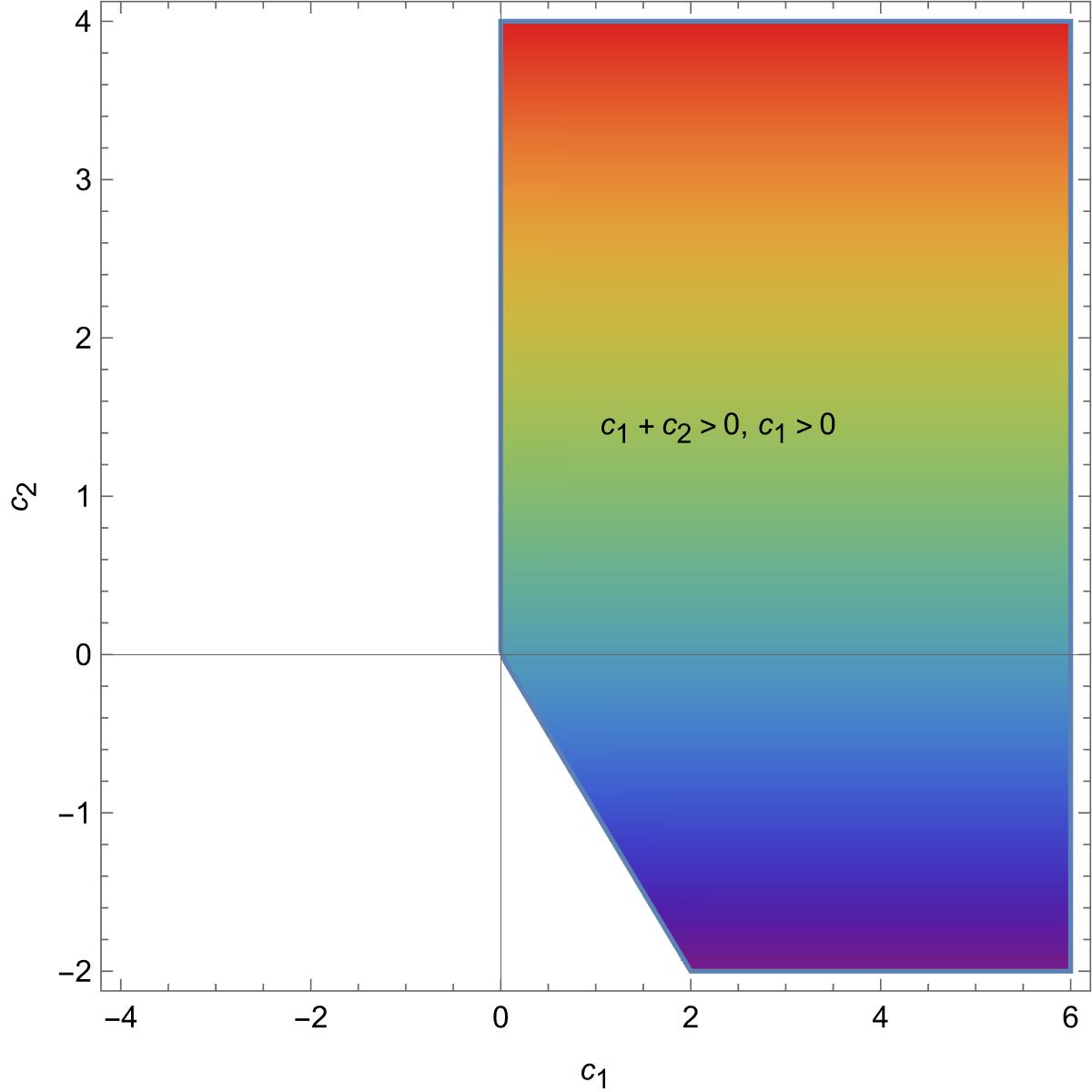}
    \end{minipage}
}
\subfigure[The steady state]
{
 	\begin{minipage}[b]{.45\linewidth}
        \centering
        \includegraphics[scale=0.271]{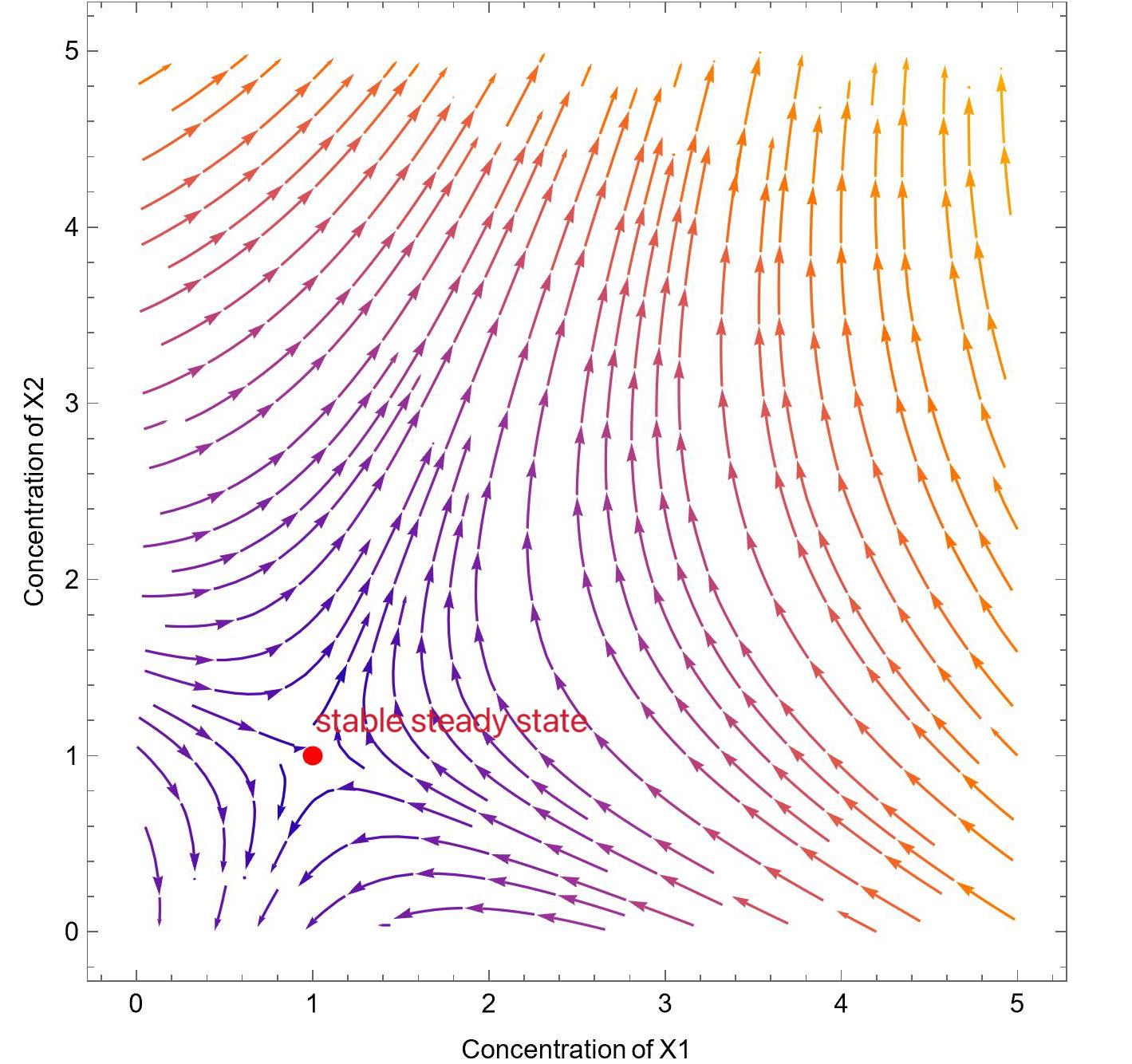}
    \end{minipage}
}
\caption{\textbf{Illustrating Theorem \ref{thm:one-dimensional} by Example \ref{ex:13}.} (a) The parameter region admits a (stable) positive steady state. (b) Let $\kappa_1=\kappa_2=1,\;c_1=2,\;c_2=0$. We get a positive steady state $x = (1,1,1)^\top$, and it is stable.}\label{fig:th2}
\end{figure}
\begin{theorem}\label{thm:twospecies} 
A two-dimensional zero-one network  with up to  three species either admits no multistationarity or  only admits degenerate positive steady states.
\end{theorem}
\begin{corollary}\label{cor:two dim}
    A two-dimensional zero-one network $G$ with up to  three species admits no nondegenerate multistationarity.  And if the network $G$ admits a nondegenerate positive steady state, then the steady state is stable. 
\end{corollary}
\begin{corollary}\label{cor:three dim}
    If a zero-one reaction network $G$ with three species admits nondegenerate multistationarity, then the network $G$ must be three-dimensional.
\end{corollary}
\begin{remark}\label{rmk:twospecies}
Theorem \ref{thm:twospecies} shows that it is possible for a two-dimensional zero-one reaction network to  admit only degenerate positive steady states (and there are infinitely many degenerate positive steady states), e.g., Example  \ref{ex:d2degenerate}.  This is a main difference from the one-dimensional case (recall here, Theorem \ref{thm:one-dimensional} shows that a one-dimensional zero-one network either admits no positive steady states, or admits exactly one stable positive steady state). 
For this reason, a one-dimensional zero-one network has ACR when it is full-dimensional, while we can not see whether a two-dimensional zero-one network has ACR even if it is full-dimensional. 
\end{remark}

\begin{example}\label{ex:d2degenerate}
Consider the following network:
\begin{align}
    &0 \overset{\kappa_1}{\longrightarrow} X_1,&&0 \overset{\kappa_2}{\longrightarrow} X_2,&&0\overset{\kappa_3}{\longrightarrow} X_1+X_2,\notag\\
    &X_1 +X_2\overset{\kappa_4}{\longrightarrow}0,&&X_1+X_2 \overset{\kappa_5}{\longrightarrow} X_1,&&X_1+X_2 \overset{\kappa_6}{\longrightarrow} X_2.\nonumber
\end{align}
The corresponding steady-state system $f$ defined in \eqref{eq:sys} is 
\begin{align}\label{eq:ex_degenerate_f}
    f_1&=-\kappa_4 x_1 x_2-\kappa_6 x_1 x_2+\kappa_1+\kappa_3,\nonumber \\
    f_2&=-\kappa_4 x_1 x_2-\kappa_5 x_1 x_2+\kappa_2+\kappa_3.
\end{align}
Obviously, the network admits positive steady states. For instance, if $\kappa_i=1$ for all $i\in \{1, 2, 3, 4, 5, 6\}$, then any  $x=(x_1, x_2)\in {\mathbb R}_{>0}^2$ satisfying $x_1x_2=1$ is a positive steady state. 
By \eqref{eq:ex_degenerate_f}, it is straightforward to check that 
\begin{equation}
    det({\rm Jac}_f(\kappa,x))\equiv 0.\nonumber
\end{equation}
So, the network only admits degenerate positive steady states. 
\end{example}
Below, we present the last main result Theorem \ref{thm:multistab}. We also present its proof since it is carried out by  a computational procedure, see  the flow diagram in Fig \ref{fig:horizontal_flowchart}.
We give some details for the computations implemented in the proof of Theorem \ref{thm:multistab}.
And we provide one of the smallest zero-one networks admitting nondegenerate multistationarity/multistability in Example \ref{ex:non_multistability}/Example \ref{ex:smallest}.
\begin{theorem}\label{thm:multistab}
   For any three-dimensional zero-one network $G$ with three species, we have the following statements.
   \begin{enumerate}[(I)]
       \item If $G$ admits nondegenerate multistationarity, then $G$ has at least five reactions. Moreover, there exists a three-dimensional zero-one network with three species and five reactions such that the network admits nondegenerate multistationarity.
       \item If $G$ admits multistability, then $G$ has at least six reactions. Moreover, there exists a three-dimensional zero-one network with three species and six reactions such that the network admits multistability.
   \end{enumerate}
\end{theorem}
\begin{proof}\label{pro:small}
\begin{enumerate}[(I)]
    \item First for any three-dimensional  network with three species and three reactions, the stoichiometric matrix  is a full-rank square matrix. 
Therefore, the network admits no positive steady states. 
By \cite[Lemma 3.1]{BB2023}, we know that a three-dimensional network with three species and four reactions admits no nondegenerate multistationarity.
Thus, if a three-dimensional zero-one network  with three species admits nondegenerate multistationarity, then the network has at least five reactions. And there exists a three-dimensional zero-one network with three species and five reactions such that the network admits nondegenerate multistationarity, see Example \ref{ex:non_multistability} (we will explain how to find this example in the part (II)).
\item By (I), if $G$ admits multistability, then $G$ has at least five reactions. Below, we prove by a computational way that if $G$ has five reactions, then the network $G$ admits no multistability.  
 Notice that we  can enumerate all  three-dimensional zero-one networks with three species and five reactions that admit nondegenerate positive steady states, and there are $65440$ such networks. We apply the following comprehensive procedure for detecting the multistable networks (one can apply the procedure for any set of networks). Also,  see the flow diagram of the procedure in Fig \ref{fig:horizontal_flowchart}.
\begin{enumerate}
\item[{\bf (Step 1).}]First, we  check  the injectivity. 

{\bf Method of Step 1. }Recall that if a reaction network is injective (i.e., the determinant $\det({\rm Jac}_f(\kappa, x))$ does not change the sign for any $\kappa\in {\mathbb R}_{>0}^m$ and for any $x\in {\mathbb R}_{>0}^s$), then the network admits no multistationarity (e.g., \cite{CF2005}).  
   In practice, one can check the injectivity by the following two simple criteria:
  (i) if the polynomial $\det({\rm Jac}_f(\kappa, x))$  contains only positive or only negative terms, then the network is injective; (ii)  if the polynomial $\det({\rm Jac}_{x-f}(\kappa, x))$  contains only positive or only negative terms, then the reaction network is injective  \cite[Theorem 3.1]{CF2005}.

  {\bf Result of Step 1. }We find that the injectivity of  $39233$  networks from the original   $65440$ networks can not be determined by the the above two criteria, which might admit multistationarity.  
 \item[{\bf (Step 2).}] 
 Second, we check the nondegenerate multistationarity.

 {\bf Method of Step 2. } Carrying out  {\tt RealRootClassification} in {\tt Maple}.
 
  {\bf Result of Step 2.} All the above $39233$ reaction networks admit at most two nondegenerate positive steady states, and there exists $429$  networks admitting  exactly two nondegenerate positive steady states. 
 \item[{\bf (Step 3).}]Third, we check the multistability. 
 
{\bf Method of Step 3.} For any network admitting exactly two nondegenerate positive steady states, the command {\tt RealRootClassification} gives at least one witness 
$\kappa^*\in \mathbb{R}^{5}_{>0}$ over each open connected component ${\mathcal O}$ of the complement of the discriminant variety (see Remark \ref{re:discriminant}) such that the network has two nondegenerate positive steady states $x^{(1)}$,  $x^{(2)}\in \mathbb{R}^3_{>0}$. And it is straightforward to check whether  ${\rm det}({\rm Jac}_f(\kappa^*,x^{(1)}))$ and ${\rm det}({\rm Jac}_f(\kappa^*,x^{(2)}))$ have different signs for these two steady states. If so, only one of $x^{(1)}$ and $x^{(2)}$ is stable since 
by Lemma \ref{lm:33stable} (see Appendix \ref{S2_Appendix}), for any $\kappa^*\in \mathbb{R}^{5}_{>0}$ and for any corresponding positive steady state $x^* \in \mathbb{R}^3_{>0}$, if the positive steady state $x^*$ is stable, then ${\rm det}({\rm Jac}_f(\kappa^*,x^*))<0$.  Notice that by the theory \cite[Section 6.1]{LX2016} of real root classification, for any rate-constant vector $\kappa$ located in $\mathcal{O}$, there will be two nondegenerate steady states, and if for one particular $\kappa\in \mathcal{O}$, the determinant of the Jacobian matrix ${\rm det}({\rm Jac}_f)$ will have different signs at the two steady states, then it happens for any $\kappa\in \mathcal{O}$ (that means only one of the two steady states will be stable).  

{\bf Result of Step 3. } We conclude that all the $429$  networks admit at most one stable steady state.
\end{enumerate}
Hence, if $G$ admits multistability, then $G$ has at least six reactions.
Moreover, there exists a three-dimensional zero-one network with three species and six reactions such that the reaction network admits multistability, see Example \ref{ex:smallest}. The supporting codes are available online (\href{https://github.com/YueJ13/network/blob/main/smallest}{https://github.com/YueJ13/network/blob/main/smallest}).
\end{enumerate}
\end{proof}
\begin{figure}[H]
    \centering
    \begin{tikzpicture}[node distance=1.5cm, auto]
        \node [inp, align=center,font=\small] (start) {Enumerate all ``non-trivial" $3$-species zero-one networks with $5$ reactions};
        \node [pro,align=center, below of=start,font=\small] (process2) {Check the   injectivity};
        \node [pro,align=center, below of=process2,font=\small] (process4) {Delete the networks admitting no nondegenerate \\multistationarity by 
        using {\tt RealRootClassification} in {\tt Maple}};
        \node [pro,align=center, below of=process4,font=\small] (process5) {Delete the networks admitting no multistability \\by Lemma \ref{lm:33stable} (deduced by Hurwitz criterion)};
             \node [inp,align=center, below of=process5,font=\small] (process7) {Apply a similar procedure to find a $3$-species \\zero-one network with $6$ reactions admitting  multistability};
        \draw [arrow] (start) --node[right,align=center,font=\small] {65440 networks} (process2);
        \draw [arrow] (process2) -- node[right,align=center,font=\small] {39233 networks}(process4);
         \draw [arrow] (process4) --node[right,align=center,font=\small] {429  networks}(process5);
         \draw [arrow] (process5) --node[right,align=center,font=\small]{0 networks}(process7);
    \end{tikzpicture}
    \vspace{1em}
    \captionsetup{justification=centering} 
    \caption{The flow diagram of proof of Theorem \ref{thm:multistab} (II) for the three-dimensional networks.}
    \label{fig:horizontal_flowchart}
    {\bf Note:} By a ``non-trivial" network, we mean the network admits nondegenerate positive steady states. 
\end{figure}

We have implemented the procedure in Fig \ref{fig:horizontal_flowchart} in {\tt Maple}. 
And we run the procedure by a 3.60 GHz Inter Core i9 processor (64GB total memory) under Windows 10. 
We have checked all the ``non-trivial" three-dimensional zero-one networks with three species and five reactions through the procedure in Fig \ref{fig:horizontal_flowchart}. 
We record the timings for carrying out these steps in Table \ref{tab:time}. 
\begin{table}[H]
    \centering
\caption{The computational time (h: hours) for running the procedure in Fig \ref{fig:horizontal_flowchart}.}
    \begin{tabular}{ccc}
    \hline
    \scriptsize{STEP}&\scriptsize{TIME}&\scriptsize{NUMBER OF NETWORKS}  \\
    \hline
         Check  the injectivity&0.1h &65440  \\
         Check the nondegenerate multistationarity&18h  &39233  \\
         Check the multistability&0.1h &429\\ 
    \hline
    \end{tabular}
    
    {\bf Notes:} (i) The column ``STEP" lists the names of the blue steps in the procedure shown in Fig \ref{fig:horizontal_flowchart}. (ii) The column ``TIME" records the computational time for computing each step. (iii) The last column records the number of networks we dealt with. 
    \label{tab:time}
\end{table}

As what has been shown in Fig \ref{fig:horizontal_flowchart} (or Table \ref{tab:time}), we have found $429$  three-species, five-reaction networks that admit nondegenerate multistationarity, and
we present one of them in Example \ref{ex:non_multistability}. 
All theses $429$  reaction networks are listed in a supporting file (\href{https://github.com/YueJ13/network/blob/main/smallest/networks_two_states.txt}{https://github.com/YueJ13/network/blob/main/smallest/networks\_two\_states.txt}).

\begin{example}\label{ex:non_multistability}
Consider the following network $G$:
    \begin{equation}\label{eq:3s5rnetwork}
\begin{aligned}
    &X_1+X_2 \overset{\kappa_1}\longrightarrow X_3,&&X_3 \overset{\kappa_2}\longrightarrow0,&&0\overset{\kappa_3}\longrightarrow X_3,\\
    &X_1+X_3 \overset{\kappa_4}\longrightarrow X_1+X_2+X_3,&&X_2+X_3 \overset{\kappa_5}\longrightarrow X_1+X_2+X_3.
\end{aligned}
\end{equation}
\noindent The steady-state system $f$ defined as in \eqref{eq:sys} is as follows.
\begin{equation}\label{eq:3s5r_f}
    \begin{aligned}
        f_1&=-\kappa_1x_1x_2+\kappa_5x_2x_3, \\
        f_2&=-\kappa_1x_1x_2+\kappa_4x_1x_3,\\
        f_3&=\kappa_1x_1x_2-\kappa_2x_3+\kappa_3.
    \end{aligned}
\end{equation}
Let $\kappa^*=(1,3,2,1,1)$. Substituting $\kappa^*\in\mathbb{R}^5_{>0}$ into \eqref{eq:3s5r_f}, we get two distinct positive steady states.
\begin{equation}
    \begin{aligned}
        x^{(1)}&=(1,1,1)^\top,\\
        x^{(2)}&=(2,2,2)^\top,
    \end{aligned} \nonumber
\end{equation}
where $x^{(1)}$ is stable,  and $x^{(2)}$ is unstable. 
\end{example}
\begin{figure}[H]
    \centering
    \includegraphics[scale=0.2]{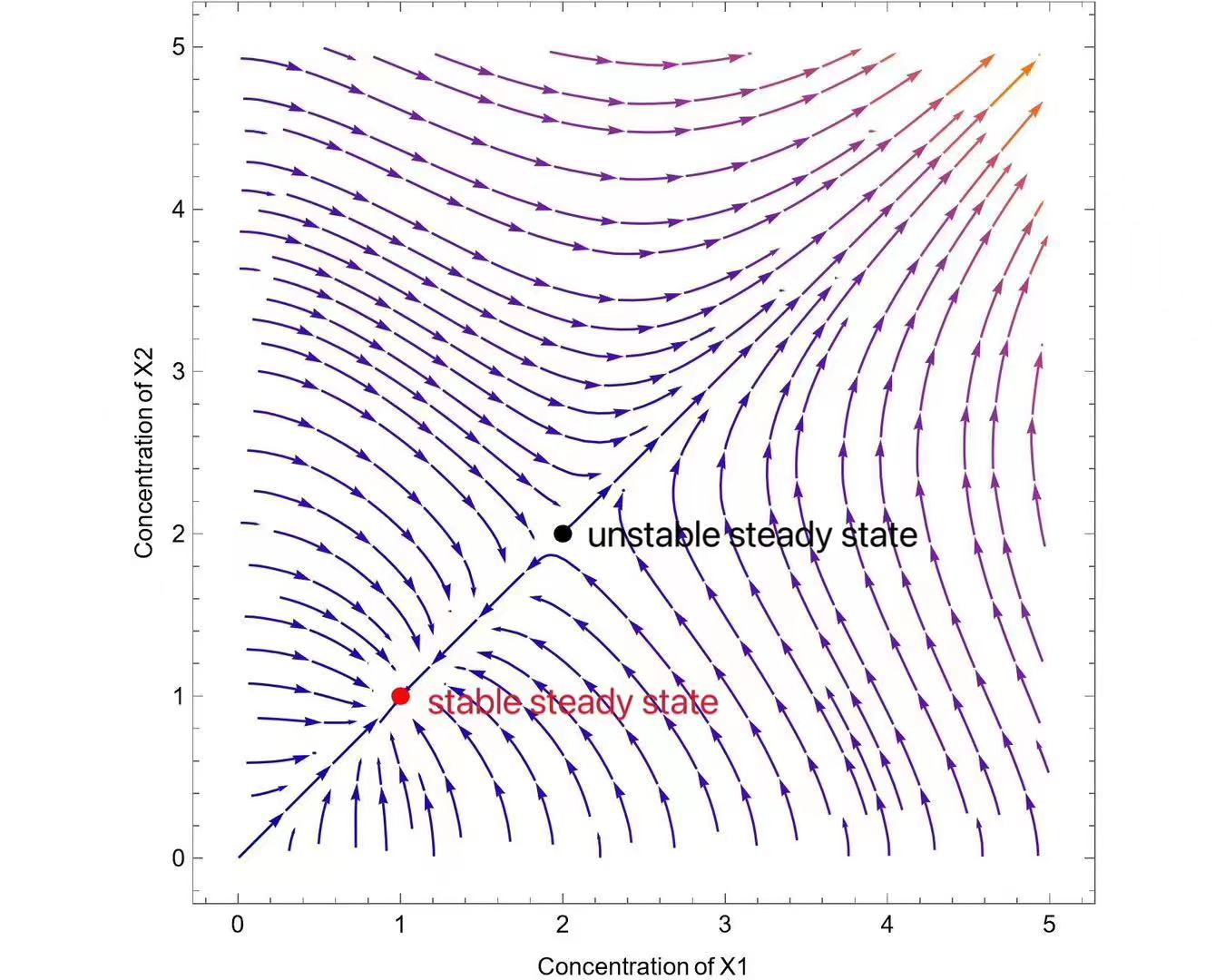}
    \caption{\textbf{The steady states of the network \eqref{eq:3s5rnetwork}.} Let $\kappa^*=(1,3,2,1,1)$. We get two positive nondegenerate steady states $x^{(1)}=(1,1,1)^\top$ and $x^{(2)}=(2,2,2)^\top$.}
    \end{figure}

\begin{remark}   
When we enumerate the non-trivial three-dimensional zero-one networks with three species and six reactions, there are  $1367698$ networks in total. 
We estimate that  it will take about $380$ hours to complete the procedure (the blue steps) in Fig \ref{fig:horizontal_flowchart}.
Hence, it is not very realistic to complete the whole computation. 
However, we  still successfully find a multistable network by applying the procedure to the six-reaction networks, see Example \ref{ex:smallest}.
\end{remark}
\begin{example}\label{ex:smallest}
Consider the following network $G$:
\begin{equation}\label{eq:3-3-6network}
    \begin{aligned}
        &X_1+X_2+X_3\xrightarrow{\kappa_1}0,
&&0\xrightarrow{\kappa_2}X_3,&&X_3\xrightarrow{\kappa_3}0,\\
&X_1\xrightarrow{\kappa_4}X_1+X_2,&&X_2\xrightarrow{\kappa_5}X_1,&&X_1+X_2\xrightarrow{\kappa_6}X_1+X_3.
    \end{aligned}
\end{equation}
\noindent The steady-state system $f$ defined as in \eqref{eq:sys} is as follows.
\begin{equation}\label{eq:3-3-6sys}
\begin{aligned}
    f_1&=-\kappa_1x_1x_2x_3 +\kappa_5x_2,\\
    f_2&=-\kappa_1x_1x_2x_3-\kappa_6x_1x_2+\kappa_4x_1-\kappa_5x_2,\\
    f_3&=-\kappa_1x_1x_2x_3+\kappa_6x_1x_2-\kappa_3x_3+\kappa_2.
\end{aligned}   
\end{equation}
Let $\kappa^*=(\frac{5765}{16},\frac{1655}{65536},\frac{1}{2},\frac{1}{2},\frac{1}{2},\frac{1}{2})$. Substituting  $\kappa^*\in\mathbb{R}^6_{>0}$ into \eqref{eq:3-3-6sys}, we get three distinct positive steady states.
\begin{equation}
    \begin{aligned}
        x^{(1)}&=(0.8340329166, 0.2942918947, 0.001663824382)^\top,\\
        x^{(2)}&=( 0.05999575106, 0.02912421107, 0.02312970964)^\top,\\
        x^{(3)}&=(0.05546474050, 0.02698403889, 0.02501921562)^\top.
    \end{aligned} \nonumber
\end{equation}
It is straightforward to check by Lemma \ref{lm:33stable} that  $x^{(1)}$ and $x^{(3)}$ are stable. 
\end{example}

\begin{figure}[htbp]
    \centering
    \includegraphics[scale=0.2]{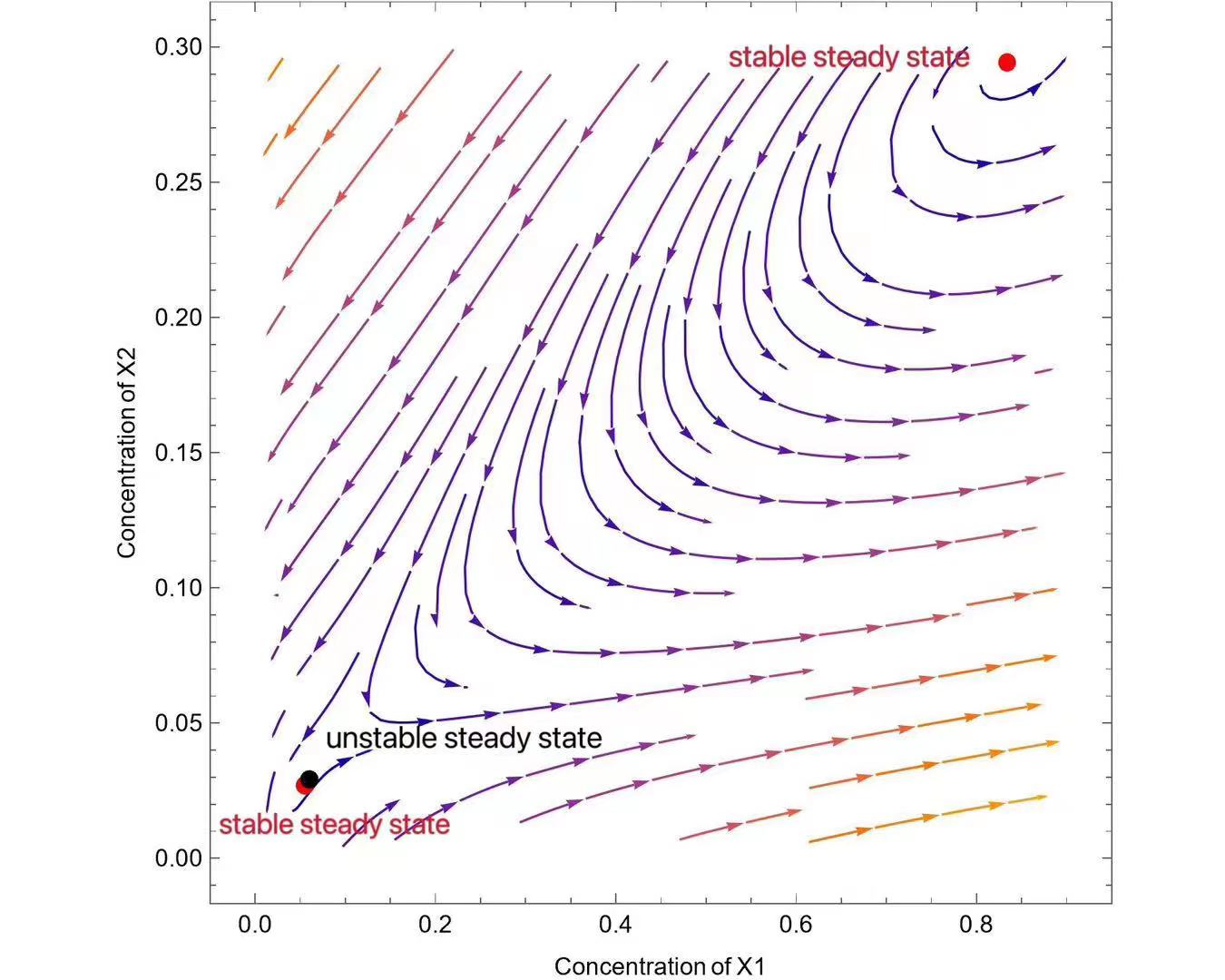}
    \caption{\textbf{The steady states of the network \eqref{eq:3-3-6network}.}  Let $\kappa^*=(\frac{5765}{16},\frac{1655}{65536},\frac{1}{2},\frac{1}{2},\frac{1}{2},\frac{1}{2})$. We get three steady states $ x^{(1)}=(0.8340329166, 0.2942918947,$\\ $0.001663824382)^\top$, $
        x^{(2)}=( 0.05999575106, 0.02912421107, 0.02312970964)^\top$, and $
        x^{(3)}=(0.05546474050, 0.02698403889, 0.02501921562)^\top$, where $x^{(1)}$ and $x^{(3)}$ are stable.}
    \end{figure} 

\section{Methods}\label{sec:methods}
\subsection{Jacobian matrices of zero-one networks}\label{sec:trans}
In this work, a crucial problem is to determine the sign of ${\rm det}({\rm Jac}_h)$ evaluated at a positive steady state.   In this section, we review a standard transformation of the Jacobian matrix using extreme rays (e.g., \cite[Section 4]{TW2022}) since usually it is easier to analyze the sign of the transformed Jacobian matrix. 
We also review two lemmas proved in \cite{TW2022}, which reveal the special properties of the transformed Jacobian matrices for zero-one networks (see Lemmas \ref{lm:kxtohl}--\ref{lm:trace}). 
Then, for the zero-one networks with low dimensions,  we study the sign of ${\rm det}({\rm Jac}_h)$ evaluating at the nondegenerate positive steady states in Lemma \ref{lm:jach}. Based on Lemma \ref{lm:jach} and a known result (Lemma \ref{lm:stable}),  for any zero-one network with dimension one or two, we prove that any nondegenerate positive steady state is stable in Lemma \ref{lm:zero-one stable}. We also develop a criterion (Lemma \ref{lm:jac B jac}) and a corresponding algorithm (Algorithm \ref{alg:checksign}) for checking the sign of ${\rm det}({\rm Jac}_h)$ evaluated at a positive steady state. 
We will use these results to determine  multistationarity  later in Section \ref{sec:one dim} and Section \ref{sec:2d 3s}. 

Consider a network $G$ \eqref{eq:network} with a stoichiometric matrix $\mathcal{N}$ and a reactant matrix $\mathcal{Y}$. Recall the system $\dot{x} = f(\kappa, x) = \mathcal{N}v(\kappa, x)$ defined as in \eqref{eq:sys}. 
For any column vector $y := (y_1, \dots , y_s)^\top \in \mathbb{R}^s$, we denote by ${\rm diag}(y)$ the $s\times s$ diagonal matrix with $y_i$ on the diagonal. Then, the Jacobian matrix  ${\rm Jac}_f(\kappa,x)$  can be written as 
\begin{align}
  {\rm Jac}_f(\kappa,x)=\mathcal{N}{\rm diag}(v(\kappa,x))\mathcal{Y}^\top{\rm diag}(\frac{1}{x}), \nonumber
\end{align}
where $\frac{1}{x}:=(\frac{1}{x_1},\ldots,\frac{1}{x_s})^\top$.\\

Next, we consider a transformation of the Jacobian matrix ${\rm Jac}_f(\kappa,x)$ evaluated at the positive steady states.
For the stoichiometric matrix ${\mathcal N} \in \mathbb{R}^{s \times m}$, the \defword{flux cone} of ${\mathcal N}$ is defined as
\begin{align}\label{eq:fluxcone}
\mathcal{F}({\mathcal N})\; := \;\{\gamma \in {\mathbb R}_{\geq 0}^m:{\mathcal N} \gamma = \textbf{0} \},
\end{align}
where $\bf0$ denotes the vector whose coordinates are all zero.
Let $R^{(1)},\ldots,R^{(t)} \in \mathbb{R}^m_{\geq 0}$ be a fixed choice of representatives for the extreme rays of the flux cone $\mathcal{F}(\mathcal{N})$ (note here, it is known that 
the number of extreme rays is finite \cite{Conradi2019}).
Then, any $\gamma \in \mathcal{F}({\mathcal N})$ can be written as 
\begin{align}\label{eq:jac gamma}
\gamma=\sum_{i=1}^{t} \lambda_{i} R^{(i)}, \;\text{where}\; \lambda_{i} \geq 0\; \text{for any}\; i\in \{1, \ldots, t\}. 
\end{align}
We introduce the new variables $p_1,\ldots,p_s,\lambda_1,\ldots,\lambda_t$.
Let $p:=(p_1,\ldots,p_s)^\top=\frac{1}{x}$, and let $\lambda:=(\lambda_1,\ldots,\lambda_t)^\top$.
We define a new matrix in terms of $p$ and $\lambda$:
\begin{align}\label{eq:defjhl}
   J(p,\lambda)\;:=\; \mathcal{N}{\rm diag}  (\sum_{i=1}^{t} \lambda_{i} R^{(i)})\mathcal{Y}^\top {\rm diag} (p). 
\end{align}

\begin{lemma}\label{lm:kxtohl}\cite[Lemma 4.1]{TW2022}
Consider a network $G$ defined as in \eqref{eq:network}. 
Let $f$ be the steady-state system defined as in \eqref{eq:sys}. 
Let $J(p,\lambda) \in \Q[p,\lambda]^{s \times s}$ be the matrix defined as in \eqref{eq:defjhl}.
For any $\kappa \in \mathbb{R}^m_{>0}$ and for any corresponding positive steady state $x \in \mathbb{R}^s_{>0}$, there exist $p \in \mathbb{R}_{> 0}^s$ and $\lambda \in \mathbb{R}^t_{\geq 0}$ such that $J(p,\lambda)={\rm Jac}_f(\kappa,x)$.
\end{lemma}

Given $\mathcal{M}\in \mathbb{R}^{s\times s}$ and  $I \subseteq\{1,\dots,s\}$, we denote by $\mathcal{M}[I,I]$  the submatrix of $\mathcal{M}$ with entries of $\mathcal{M}$ with indices $(i,j)$ in $I\times I$.
 
\begin{lemma}\label{lm:trace}\cite[Corollary 5.2, Lemma 5.4]{TW2022}
Consider a zero-one network $G$ defined as in \eqref{eq:network}. Let $J(p, \lambda) \in\mathbb{Q}[p, \lambda]^{s\times s}$ be the matrix defined as in \eqref{eq:defjhl}. 
Then, we have the following statements. 
\begin{enumerate}[(I)]
    \item \label{lm:trace 1} For any $I\subseteq\{1,\dots,s\}$ satisfying  $\vert I \vert\;=1$, ${\rm det}(J[I,I])$ is either a zero polynomial or a sum of terms
with negative coefficients.
    \item \label{lm:trace 2} For any $I\subseteq\{1,\dots,s\}$ satisfying  $\vert I \vert\;=2$, ${\rm det}(J[I,I])$ is either a zero polynomial or a sum of terms
with positive coefficients.
\end{enumerate}
\end{lemma}
\begin{lemma}\label{lm:jach}
Consider an $r$-dimensional ($r\in\{1,2\}$) zero-one network $G$ \eqref{eq:network}.
Let $h$ be the steady-state system augmented by conservation laws defined as in \eqref{eq:h}.
For any $\kappa\in \mathbb{R}^{m}_{>0}$, for any $c\in \mathbb{R}^{s-r}$, and for any  corresponding nondegenerate positive steady state $x$ in ${\mathcal P}^{+}_{c}$, we have 
\begin{equation}
    {\rm sign}({\rm det}({\rm Jac}_h(\kappa, x)))=(-1)^r. \nonumber
\end{equation}
\end{lemma}
\begin{proof}
Let $f$ be the steady-state system defined as in \eqref{eq:sys}. Since $G$ is $r$-dimensional, by \cite[Proposition 5.3]{WiufFeliu_powerlaw}, we have
\begin{equation}\label{eq:det_h_f}
\begin{aligned}
 {\rm det}({\rm Jac}_h)&=\sum_{I\subseteq\{1,\dots,s\},\;\vert I\vert=r}{\rm det}({\rm Jac}_f[I,I]).    
\end{aligned}
\end{equation}
Let $J(p, \lambda) \in \mathbb{Q}[p, \lambda]^{s\times s}$ be the matrix
corresponding to $G$ defined in \eqref{eq:defjhl}.
By Lemma \ref{lm:kxtohl}, for any $\kappa \in \mathbb{R}^m_{>0}$ and for any corresponding positive steady state $x \in \mathbb{R}^s_{>0}$, there exist $p \in \mathbb{R}_{> 0}^s$ and $\lambda \in \mathbb{R}^t_{\geq 0}$ such that $J(p,\lambda)={\rm Jac}_f(\kappa,x)$. 
So, we have $$J(p,\lambda)[I,I]={\rm Jac}_f(\kappa,x)[I,I],$$ where $I\subseteq\{1,\dots,s\}$. Note that $G$ is a zero-one network. 
If $\vert I \vert\;=1$, then, by Lemma \ref{lm:trace} (\ref{lm:trace 1}), we get ${\rm det}({\rm Jac}_f[I,I])\leq0$. 
If $\vert I \vert\;=2$, then, by Lemma \ref{lm:trace} (\ref{lm:trace 2}), ${\rm det}({\rm Jac}_f[I,I])\geq0$. Thus,  by \eqref{eq:det_h_f},  for any $\kappa\in \mathbb{R}^{m}_{>0}$, for any $c\in \mathbb{R}^{s-r}$ and for any corresponding steady state $x$ in ${\mathcal P}^{+}_{c}$, we have ${\rm sign}({\rm det}({\rm Jac}_h(\kappa, x)))=(-1)^r$ or ${\rm det}({\rm Jac}_h(\kappa, x))=0$, where $r\in\{1,2\}$.
\end{proof}
The following lemma is elementary (e.g., one can combine two results \cite[Lemma 3.4]{TX2021} and \cite[Proposition 5.3]{WiufFeliu_powerlaw}). 
\begin{lemma}\label{lm:stable}
Consider a one-dimensional  network $G$. 
Let $h$ be the steady-state system augmented by conservation laws defined as in \eqref{eq:h}.  
A nondegenerate steady state $x$ is stable if and only if ${\rm det}({\rm Jac}_h(\kappa,x))<0$.
\end{lemma}
\begin{lemma}\label{lm:zero-one stable}
    Consider an $r$-dimensional ($r\in \{1,2\}$) zero-one network $G$. 
    For any rate-constant vector $\kappa\in \mathbb{R}^{m}_{>0}$, if $x$ is a nondegenerate positive steady state, then $x$ is stable. 
\end{lemma}
\begin{proof}
If $r=1$, then by Lemma \ref{lm:jach} and Lemma \ref{lm:stable}, the conclusion holds. Below, we prove the conclusion for $r=2$. By Lemma \ref{lm:Hurwitz} (see Appendix \ref{S2_Appendix}), we only need to prove that 
for any $\kappa^*\in\mathbb{R}^{m}_{>0}$ and for any corresponding nondegenerate positive steady state $x^* \in \mathbb{R}_{>0}^s$, ${\rm det}({\rm Jac}_h(\kappa^*,x^*))>0$ and $\sum^s_{i=1}\frac{\partial f_i}{\partial x_i}(\kappa^*,x^*)<0$.
Since $x^*$ is nondegenerate, by Lemma \ref{lm:jach}, we have ${\rm det}({\rm Jac}_h(\kappa^*,x^*))>0$. 
Below, we prove $\sum^s_{i=1}\frac{\partial f_i}{\partial x_i}(\kappa^*,x^*)<0$.
Since $x^*$ is a positive steady state, for any $i \in \{1,\dots,s\}$,  $x^*$ is a positive solution to $f_i(\kappa^*,x)=0$. 
Hence,  $f_i(\kappa^*,x)$ is a zero polynomial or $f_i(\kappa^*,x)$ has terms with both negative and positive coefficients. 
Since $G$ has dimension two, there exists at least one polynomial $f_j(\kappa^*,x)$ that has terms with both negative and positive coefficients. 
Hence, by the definition of zero-one network, for any term containing $x_j$ in $f_j$, its coefficient is $-1$.  
So, $\frac{\partial f_j}{\partial x_j}(\kappa^*,x^*)<0$.
By Lemma \ref{lm:kxtohl} and Lemma \ref{lm:trace} (\ref{lm:trace 1}), for any $i \in \{1,\dots,s\} \setminus\{j\}$, we have $\frac{\partial f_i}{\partial x_i}(\kappa^*,x^*) \le 0$. 
Therefore, we have $\sum^s_{i=1}\frac{\partial f_i}{\partial x_i}(\kappa^*,x^*)<0$. 
\end{proof}
In fact, for one-dimensional zero-one networks, we can even prove that ${\rm det}({\rm Jac}_h)<0$ at any positive steady state (notice here, Lemma \ref{lm:jach} only guarantees ${\rm det}({\rm Jac}_h)\leq 0$), see Lemma \ref{lm:sum_partical<0} in Section \ref{sec:one dim}. And there is a ``generically" similar result for the  zero-one reaction networks with dimension two, see more details in Section \ref{subsubsec:sign jach}. 
Below, we prove a criterion for determining if ${\rm det}({\rm Jac}_h)<0$ and develop Algorithm \ref{alg:checksign} for the two-dimensional case (while there exists a more straightforward mathematical proof for the one-dimensional case),  this algorithm will play a key role later in the proof presented in Section \ref{subsubsec:sign jach}.  
Recall the formula \eqref{eq:jac gamma} presented before. 
We denote by  $\gamma_r$ and $R_r^{(k)}$ the $r$-th coordinates of $\gamma$ and $R^{(k)}$, respectively. 
Then, by \eqref{eq:jac gamma}, we have 
\begin{align}\label{eq:jac gamma r}
    \gamma_r=\sum\limits_{k=1}^t \lambda_k R_r^{(k)}.
\end{align}
Let $J(p,\lambda)$ be the transformed Jacobian matrix defined in \eqref{eq:defjhl}.
We define 
\begin{align}\label{eq:jac B}
B(p,\lambda)\;:=\;\sum\limits_{I \subset \{1,\ldots,s\},\;\vert I \vert=2} {\rm det}(J(p,\lambda)[I,I]).
\end{align}
By \eqref{eq:det_h_f}, if a network $G$ is two-dimensional, then the polynomial $B(p,\lambda)$ can be view as the transformed version of ${\rm det}({\rm Jac}_h(\kappa, x))$. 
Note that by \eqref{eq:defjhl}, the degree of each entry in $J(p,\lambda)$ w.r.t. $\lambda$ is at most $1$. Thus, by \eqref{eq:jac B}, the degree of each term in $B(p,\lambda)$ w.r.t. $\lambda$ is at most $2$.
We define
\begin{equation}\label{eq:detjac_h}
    \begin{aligned}
         \Theta\;:=\;\{k \mid & \exists \text{ a term in } B(p,\lambda) \;\text{s.t. the degree of $\lambda_k$ in the term is}\; 2,\;
    k\in\{1,\dots,t\}\},
    \end{aligned}
\end{equation}
and we define 
\begin{align}\label{eq:jac tilde B}
    \tilde{B}(p,\lambda)\;:=\;B(p,\lambda)\vert_{\lambda_k=0 , \;\text{for any }k\in \Theta}.
\end{align}
For any $i\in\{1,\dots,m\}$, we define 
\begin{equation}
\begin{aligned}\label{eq:qi}
    q_i\;:=\;\{k\mid R_i^{(k)}\ne 0,\; k\in\{1,\dots,t\}\}, 
\end{aligned}
\end{equation}
and for any $i,j \in \{1,\dots,m\}$, we define
\begin{align}\label{eq:jac qi.qj}
    \Lambda_{ij}\;:=\;\{\lambda_k \lambda_{\ell} \mid k\in {q}_i,\;{\ell}\in {q}_j\}.
\end{align}
Finally, we define 
\begin{align}\label{eq:tilde_q}
    \tilde{q}_i\;:=\;q_i\setminus\Theta,
\end{align}
and for any $i,j \in \{1,\dots,m\}$, we define
\begin{align}\label{eq:jac tilde qi.qj}
    \tilde{\Lambda}_{ij}\;:=\;\{\lambda_k \lambda_{\ell} \mid k\in \tilde{q}_i,\;{\ell}\in \tilde{q}_j\}.
\end{align}
\begin{lemma}\label{lm:jac B jac}
Consider a two-dimensional zero-one  network. Let $h$ be the steady-state system augmented by conservation laws defined as in \eqref{eq:h}. Let $\tilde{B}(p,\lambda)$ be defined as in \eqref{eq:jac tilde B}. For any integers $i, j\in\{1,\dots,m\}$, let $\tilde{\Lambda}_{ij}$ be defined as in \eqref{eq:jac tilde qi.qj}. If there exist $i, j\in\{1,\dots,m\}$ such that for any element $\lambda_k \lambda_{\ell}$ in $\tilde{\Lambda}_{ij}$, there exists a term of $\tilde{B}(p,\lambda)$ such that this term can be divided by $\lambda_k \lambda_{\ell}$, then for any rate-constant vector $\kappa \in \mathbb{R}_{>0}^m$ and for any corresponding positive steady state $x\in\mathbb{R}^s_{>0}$, we have ${\rm det}({\rm Jac}_h(\kappa,x))>0$.
\end{lemma}
\begin{proof}
 By Lemma \ref{lm:jach}, we have ${\rm det}({\rm Jac}_h(\kappa,x))\ge0$, for any $\kappa \in \mathbb{R}_{>0}^m$ and for any corresponding positive steady state $x\in\mathbb{R}^s_{>0}$. We only need to show that the equality can not be reached. 
 We  prove the conclusion by deducing a  contradiction. We assume that there exist $\kappa^*\in\mathbb{R}_{>0}^m$ and a corresponding positive steady state ${x^*\in \mathbb{R}^s_{>0}}$ such that ${\rm det}({\rm Jac}_h(\kappa^*,x^*))=0$.  
 Note that the network is two-dimensional. Hence, by Lemma \ref{lm:kxtohl}, 
by \cite[Proposition 5.3]{WiufFeliu_powerlaw} and by \eqref{eq:jac B}, there exist $p^*\in \mathbb{R}_{>0}^3$ and $\lambda^*\in\mathbb{R}_{\ge0}^t$ such that $$B(p^*,\lambda^*)={\rm det}({\rm Jac}_h(\kappa^*,x^*))=0.$$ 
Then, by Lemma \ref{lm:trace} (\ref{lm:trace 2}), every term of $B(p,\lambda)$ evaluated at $(p^*,\lambda^*)$ is $0$. 
Note that by \eqref{eq:jac tilde B}, any term of $\tilde{B}(p,\lambda)$ is also a term of $B(p,\lambda)$. So, every term of $\tilde{B}(p,\lambda)$ evaluated at $(p^*,\lambda^*)$ is $0$.
Hence, $\tilde{B}(p^*,\lambda^*)=0$.\\
With the aim of showing the contradiction, below we prove that if there exist indices $i, j\in\{1,\dots,m\}$ such that for any element $\lambda_k\lambda_{\ell}$  in $\tilde{\Lambda}_{ij}$, there exists a term of $\tilde{B}(p,\lambda)$ such that the term can be divided by $\lambda_k\lambda_{\ell}$, then $\tilde{B}(p^*,\lambda^*)>0$. 
Notice that by \eqref{eq:sys}, $\mathcal{N}v(\kappa^*,x^*)=0$, where $v(\kappa^*,x^*)\in \mathbb{R}_{>0}^m$. 
Let $\gamma^*=v(\kappa^*,x^*)$. 
Then, by \eqref{eq:fluxcone}, $\gamma^*\in \mathcal{F}(\mathcal{N})$. 
Note that by \eqref{eq:jac gamma}, we can write $\gamma^*=\sum\limits_{k=1}^t\lambda_k^* R^{(k)}$ for some $\lambda_k^*$'s. 
Hence, by \eqref{eq:jac gamma r}, for any $ r\in \{1,\dots,m\}$, there exists index $k\in \{1,\dots,t\}$ such that $\lambda_k^*R_r^{(k)}>0$ (i.e., $\lambda_k^*>0$ and $R_r^{(k)}>0$). So, by \eqref{eq:qi}, $k \in q_r$. 
Below, we first prove that $k\in\tilde{q_r}$ (i.e., $k \notin \Theta$). 
We prove it by deducing a contradiction. Assume that $k \in \Theta$, see \eqref{eq:detjac_h}. 
By \eqref{eq:defjhl}, the degree of each entry of $J(p,\lambda)$ w.r.t. $\lambda$ is at most one. 
Hence, by \eqref{eq:jac B} and \eqref{eq:detjac_h}, there exists a term in $B(p,\lambda)$ such that the term has the form $u(p)\lambda_k^2$, where $u(p) \in \mathbb{Q}[p]$.
Recall that every term of $B(p,\lambda)$ evaluated at $(p^*,\lambda^*)$ is $0$. 
So, $u(p^*){\lambda_{k}^*}^2=0$. 
Note that $p^* \in \mathbb{R}_{>0}^s$.
Note also that, by Lemma \ref{lm:trace} (\ref{lm:trace 2}), the coefficient of every term of $B(p,\lambda)$ is positive. 
Hence, $\lambda_{k}^*=0$. It is contrary to $\lambda_{k}^*>0$. 
So, $k\in \tilde{q}_r$.
Above all, for any $r \in \{1,\dots,m\}$, there exists $k \in \tilde{q}_r$ such that $\lambda_k^*>0$. 
Then, there exist ${k_1} \in \tilde{q}_i$ and ${k_2} \in \tilde{q}_j$ such that $\lambda_{k_1}^*\lambda_{k_2}^*>0$. 
Note that  for any element in $\tilde{\Lambda}_{ij}$, there exists a term in $\tilde{B}(p,\lambda)$ such that this term can be divided by this element. 
We recall that every term of $\tilde{B}(p,\lambda)$ is also a term of $B(p,\lambda)$. 
Hence, the degree of each entry of $\tilde{B}(p,\lambda)$ w.r.t. $\lambda$ is at most two. 
So, we have 
\begin{align}\label{eq:jac B uv}
    \tilde{B}(p,\lambda)=u_1(p)\lambda_{k_1}\lambda_{k_2}+u_2(p,\lambda),
\end{align} 
where $u_1(p)\in \mathbb{Q}[p]$ and $u_2(p,\lambda)\in \mathbb{Q}[p,\lambda]$. 
Note that $p^* \in \mathbb{R}_{>0}^s$ and $\lambda^* \in \mathbb{R}_{\ge0}^t$. 
Hence, by Lemma \ref{lm:trace} (\ref{lm:trace 2}), we have $u_1(p^*)>0$ and $u_2(p^*,\lambda^*)\ge0$.  
Since $\lambda_{k_1}^*\lambda_{k_2}^*>0$,
we can get $u_1(p^*)\lambda_{k_1}^*\lambda_{k_2}^*>0$.  
So, by \eqref{eq:jac B uv}, we have $\tilde{B}(p^*,\lambda^*)>0$. 
\end{proof}
\begin{remark}\label{lm:jac B_remark}
Note that if we replace $\tilde{B}(p,\lambda)$ \eqref{eq:jac tilde B} and $\tilde{\Lambda}_{ij}$ \eqref{eq:jac tilde qi.qj} respectively with ${B}(p,\lambda)$ \eqref{eq:jac B} and $\Lambda_{ij}$ \eqref{eq:jac qi.qj} in Lemma \ref{lm:jac B jac}, then the conclusion also holds. In practice, we find  that $\tilde{B}(p,\lambda)$ usually has fewer terms than ${B}(p,\lambda)$. For instance, for the network \eqref{eq:g_35} listed in  Appendix \ref{S1_Appendix}, ${B}(p,\lambda)$  has $972$ terms, and $\tilde{B}(p,\lambda)$   has $114$ terms (see \eqref{eq:jh_p} and \eqref{eq:jh_p1} later). 
Therefore, in the  proofs of Lemmas \ref{lm:jac_g23}--\ref{lm:jac g23 sub} presented in the future sections, we will apply a computational method to check whether the hypothesis of Lemma \ref{lm:jac B jac} holds for $\tilde{B}(p,\lambda)$. We remark that if one hopes to complete those proofs by checking 
$B(p,\lambda)$, it might be 
 computationally infeasible.
\end{remark}
 Lemma \ref{lm:jac B jac} gives a theoretical criterion to decide upon the sign of ${\rm det}({\rm Jac}_h(\kappa,x))$, from which one can establish an algorithm with the following steps.
\begin{algorithm}\label{alg:checksign}
    Algorithm for checking the sign of ${\rm det}({\rm Jac}_h(\kappa,x))$ at any positive steady state.
    \begin{description}  
        \item [(Step 1).] 
        \parbox[t]{\dimexpr\linewidth-1cm} {For a given two-dimensional zero-one network $G$, calculate the extreme rays of ${\mathcal F}({\mathcal N})$.}
        \item [(Step 2).] Calculate $\tilde{B}(p,\lambda)$ and $\tilde{\Lambda}_{ij}$ according to \eqref{eq:jac tilde B} and \eqref{eq:jac tilde qi.qj}.
        \item [(Step 3).] \parbox[t]{\dimexpr\linewidth-1cm}  {Check that if there exist $i, j\in\{1,\dots,m\}$ such that for any element $\lambda_k \lambda_{\ell}$ in $\tilde{\Lambda}_{ij}$, there exists a term of $\tilde{B}(p,\lambda)$ such that this term can be divided by $\lambda_k \lambda_{\ell}$. If so, conclude that ${\rm det}({\rm Jac}_h(\kappa,x))>0$ for any $\kappa\in {\mathbb R}^m_{>0}$ and for any corresponding  positive steady state $x\in {\mathbb R}^s_{>0}$.}  
        \end{description}
\end{algorithm}
\begin{example}\label{ex:jac_g23}
This example illustrate how Algorithm \ref{alg:checksign} works. All the computational steps presented can be checked in a supporting file (\href{https://github.com/YueJ13/network/blob/main/sign/detailed_check_jach.nb}{https://github.com/YueJ13/network/\\ \noindent blob/main/sign/detailed\_check\_jach.nb}).
Consider the following network \eqref{eq:g_35} listed in Appendix \ref{S1_Appendix}:\\
   \begin{table}[H]
    \centering
    \begin{tabular}{lll}
        $X_3\xrightleftharpoons[\kappa_2]{\kappa_1}X_1+X_2+X_3$, &  $X_1+X_2\xrightleftharpoons[\kappa_6]{\kappa_5} X_3+X_2$,
         & $X_1\xrightleftharpoons[\kappa_{10}]{\kappa_{9}} X_2+X_3+X_1$,\\
         $0\xrightleftharpoons[\kappa_4]{\kappa_3}X_1+X_2$,&
         $X_1\xrightleftharpoons[\kappa_8]{\kappa_7} X_3$,&
         $0\xrightleftharpoons[\kappa_{12}]{\kappa_{11}} X_2+X_3$.       
    \end{tabular}
\end{table}
First, we calculate the extreme rays $R^{(1)}, \dots, R^{(t)}$ of the flux cone $\mathcal{F}(\mathcal{N})$. For this example, $t=29$.
Then, by \eqref{eq:jac B}, we have 
\begin{equation}\label{eq:jh_p}
    \begin{aligned}
        B(p,\lambda)= &\frac{1}{9}\lambda_2^2p_1p_3+\frac{1}{9}\lambda_{3}^2p_1p_2+\frac{1}{9} \lambda_{4}^2p_1p_3+\frac{1}{9} \lambda_{5}^2p_1p_3+\frac{1}{9} \lambda_{10}^2p_1p_2+\frac{1}{9} \lambda_{11}^2p_1p_2\\
        &+\frac{1}{9} \lambda_{13}^2p_1p_2+\frac{1}{9} \lambda_{16}^2p_1p_2+\frac{1}{9} \lambda_{17}^2p_1p_2+\frac{1}{4} \lambda_{18}^2p_1p_3+\frac{1}{4} \lambda_{19}^2p_1p_3+\frac{1}{9} \lambda_{20}^2p_1p_2\\
        &+\frac{1}{4} \lambda_{24}^2p_1p_2+\frac{1}{4} \lambda_{25}^2p_1p_2+\lambda_{28}^2 p_1p_3+\lambda_{29}^2 p_1p_2+\lambda_{14}\lambda_{23}p_1p_2+\lambda_{21}\lambda_{23}p_1p_2\\       &+\lambda_{14}\lambda_{22}p_1p_2+\lambda_{21}\lambda_{22}p_1p_2+\cdots,
    \end{aligned}
\end{equation}
where we omit $952$ terms in the above polynomial.
Recall that $\Theta$ is defined in \eqref{eq:detjac_h}. Thus, by \eqref{eq:jh_p}, we have
\begin{equation}\label{eq:lambda_2}
    \begin{aligned}
    \Theta=\{2,3, 4, 5,{10},{11},{13}, {16}, {17}, {18}, {19}, {20}, {24}, {25}, {28}, {29} \}.
    \end{aligned}
\end{equation}
\noindent By \eqref{eq:jh_p} and \eqref{eq:lambda_2}, the polynomial  $\tilde{B}(p,\lambda)$ defined in \eqref{eq:jac tilde B} is 
\begin{equation}\label{eq:jh_p1}
    \begin{aligned}
        \tilde{B}(p,\lambda)=\lambda_{14}\lambda_{23}p_1p_2+\lambda_{21}\lambda_{23}p_1p_2+\lambda_{14}\lambda_{22}p_1p_2+\lambda_{21}\lambda_{22}p_1p_2+\cdots,
 \end{aligned}
\end{equation}
where we omit $110$ terms in the above polynomial. 
By \eqref{eq:qi}, we have  
\begin{equation}\label{eq:q1_12}
    \begin{aligned}
        q_1&=\{14, {21}, {24}, {25}, {28}, {29}\},\;\;\\
        q_3&=\{{12}, {15}, {16}, {17}, {18}, {19}\},\;\;\\
        q_5&=\{6, {18}, {19}, {23}, {28}, {29}\},\;\;\\
        q_7&=\{7, {16}, {17}, {22}, {24}, {25}\},\;\;\\
        q_{9}&=\{2, 3, {13}, {20}, {26}, {27}\},\;\;\\
        q_{11}&=\{ 4, 5, 8, 9, {10}, {11}\},\;\; 
    \end{aligned}
    \begin{aligned}
        q_2&=\{2,4, {11}, {12}, {13}, 14\},\;\;\\
        q_4&=\{3,5, {10}, {15}, {20}, {21}\}, \;\;\\
        q_6&=\{{10}, {11}, {13}, {20}, {22}, {23}\},\;\; \\ 
        q_8&=\{2, 3, 4, 5, 6, 7\},\;\; \\ 
q_{10}&=\{9, {16}, {18}, {24}, {26}, {28}\}, \;\;\\
q_{12}&=\{8, {17}, {19}, {25}, {27}, {29}\}.\;\;
    \end{aligned}
\end{equation}
By \eqref{eq:lambda_2} and \eqref{eq:q1_12}, the set of $\tilde{q}_i$'s defined in  \eqref{eq:tilde_q} are 
\begin{equation}\label{eq:tilde q ex}
    \begin{aligned}
         \tilde{q}_1&=\{{{14}}, {21}\},\;\;\\ 
         \tilde{q}_5&=\{6,  {23}\},\;\;\\
         \tilde{q}_9&=\{{26}, {27}\},\;\;
    \end{aligned}
    \begin{aligned}
        \tilde{q}_2&=\{ {12}, {14}\},\;\;\\
        \tilde{q}_6&=\{{22}, {23}\},\;\;\\
        \tilde{q}_{10}&=\{ 9, {26}\},\;\;
    \end{aligned}
    \begin{aligned}
        \tilde{q}_3&=\{{12}, {15}\},\;\;\\
        \tilde{q}_7&=\{7, {22}\},\;\;\\
        \tilde{q}_{11}&=\{ 8, 9\},\;\;
    \end{aligned}
    \begin{aligned}
        \tilde{q}_4&=\{{15}, {21}\},\;\;\\
        \tilde{q}_8&=\{ 6, 7\},\;\;\\
        \tilde{q}_{12}&=\{ 8, {27}\}.\;\;
    \end{aligned}
\end{equation}
By \eqref{eq:jh_p1}, note that $\lambda_{14}\lambda_{23}$, $\lambda_{21}\lambda_{23}$, $\lambda_{14}\lambda_{22}$, and $\lambda_{21}\lambda_{22}$ appear in the terms of $\tilde{B}(p,\lambda)$. 
Also, by \eqref{eq:jac tilde qi.qj} and \eqref{eq:tilde q ex}, we have $$\tilde{\Lambda}_{16}=\{\lambda_{14}\lambda_{23},\lambda_{21}\lambda_{23},\lambda_{14}\lambda_{22},\lambda_{21}\lambda_{22}\}.$$ 
Then, by Lemma \ref{lm:jac B jac}, for any $\kappa \in \mathbb{R}_{>0}^m$ and for any corresponding positive steady state $x\in \mathbb{R}^3_{>0}$, we have ${\rm det}({\rm Jac}_h(\kappa,x))>0$. 
\end{example}
\subsection{One-dimensional zero-one networks}\label{sec:one dim}

In this section, the goal is to prove Theorem \ref{thm:one-dimensional}.
Since the proof for the part (I) is straightforward, the main task is to prove the part (II) where  all non-zero rows of the stoichiometric matrix ${\mathcal N}$ change signs, and 
the main idea is to apply Theorem \ref{thm:deter_sign}. 
In order  to do that, we need to  prepare some ingredients. First, we give a comprehensive form for a one-dimensional zero-one network, see Lemma \ref{lm:change sign fsx}, and we describe all the total constants such that ${\mathcal P}_{c}^+\ne \emptyset$, which is a connected region, see Lemma \ref{lm:Pc+ ne emptyset}. 
Second, 
we show that if $\mathcal{P}_c^+ \ne \emptyset$, then 
a one-dimensional  zero-one network admits  
 no boundary steady states in ${\mathcal P}_{c}$, see Lemma \ref{lm:1 dim no boundary}, and we also show that the network is dissipative, see Lemma \ref{lm:1 dim J2=emptyset dissipative}.
 Third, we prove that ${\rm det}({\rm Jac}_h(\kappa,x))$ evaluated at any positive steady state does not change the sign, 
 see Lemma \ref{lm:sum_partical<0}. 
Finally, based on these results, we complete the proof by applying Theorem \ref{thm:deter_sign}. 
\begin{lemma}\label{lm:change sign fsx}
    Consider a zero-one network $G$ with a rank-one  stoichiometric matrix $\mathcal{N}$. 
    Let $f_1,\dots, f_s$ be the polynomials defined as in \eqref{eq:sys}. Let ${\mathcal J}_i$ $(i=1, 2, 3)$ be defined as in \eqref{eq:J1}--\eqref{eq:J3}. 
    If all non-zero rows of $\mathcal{N}$ change signs, then there exist two non-empty subsets $\tau_1$ and $\tau_2$ of $2^{\mathcal{J}_3}$ (here, $2^{\mathcal{J}_3}$ denotes the set consisting of all subsets of $\mathcal{J}_3$) such that the network has the following form 
    \begin{align}\label{eq:form2}
 \sum_{t\in\mathcal{J}_2}  X_{t}+\sum_{t\in\sigma}  X_{t}&\xrightarrow{\kappa_{\sigma}} \sum_{t\in\{s\}\cup\mathcal{J}_1}X_{t}+\sum_{t\in\sigma}  X_{t},\; \text{for each}\; \sigma\in\tau_1, \\
\label{eq:form1}
 \sum_{t\in\{s\}\cup\mathcal{J}_1}X_{t}+\sum_{t\in\Lambda}  X_{t}&\xrightarrow{\kappa_{\Lambda}} \sum_{t\in\mathcal{J}_2}  X_{t}+\sum_{t\in\Lambda} X_{t}, \; \text{for each}\; \Lambda\in \tau_2,
\end{align}
 where $\kappa_{\sigma}$ and $\kappa_{\Lambda}$ denote the rate constants, 
 and we have 
    \begin{align}\label{eq:f1}
       f_s=-\sum_{\Lambda\in \tau_2}\kappa_{\Lambda}\prod\limits_{t\in \Lambda\cup \{s\}\cup\mathcal{J}_1}x_t+\sum_{\sigma\in \tau_1}\kappa_{\sigma}\prod\limits _{t\in \sigma\cup\mathcal{J}_2}x_t.
    \end{align}
\end{lemma}
\begin{proof}
    First, we can write down all reactions in $G$.  
    For any index $j \in\{1,\dots,m\}$, by \eqref{eq:1 dim Nij 01-1}, $\mathcal{N}_{sj}\in \{0,1,-1\}$. 
    Recall that by \eqref{eq:1 dim Ni ai Ns}, all row vectors of ${\mathcal N}$ can be generated by $\mathcal{N}_s$. 
    For any $j \in \{1,\dots,m\}$, we notice that 
    $\mathcal{N}_{sj}$ can not be $0$ 
    (if $\mathcal{N}_{sj}=0$, then by \eqref{eq:1 dim Ni ai Ns}, for any $i \in \{1,\dots,s-1\}$, $\mathcal{N}_{ij}=0$, which is contrary to the definition of reaction network). 
    Since all non-zero rows of $\mathcal{N}$ change signs, we know that $\mathcal{N}_s$ changes signs.
    Hence, $1$ and $-1$ indeed appear in the coordinates of $\mathcal{N}_s$. 
    For any $j \in \{1,\dots,m\}$, if  $\mathcal{N}_{sj}=1$, 
    then for any $i \in 
    \mathcal{J}_1$, by \eqref{eq:1 dim Ni ai Ns} and \eqref{eq:J1}, we get $\mathcal{N}_{ij}=1$, and 
   for any $i \in 
    \mathcal{J}_2$, by \eqref{eq:1 dim Ni ai Ns} and \eqref{eq:J2}, $\mathcal{N}_{ij}=-1$. 
   Notice that for any $i \in 
    \mathcal{J}_3$, by \eqref{eq:J3}, $\mathcal{N}_{ij}=0$. 
    Hence, if  $\mathcal{N}_{sj}=1$, then the $j$-th reaction has the form \eqref{eq:form2}.
Similarly, if  $\mathcal{N}_{sj}=-1$,  
    then for any $i \in 
    \mathcal{J}_1$, by \eqref{eq:1 dim Ni ai Ns} and \eqref{eq:J1}, $\mathcal{N}_{ij}=-1$, 
     and for any $i \in 
    \mathcal{J}_2$, by \eqref{eq:1 dim Ni ai Ns} and \eqref{eq:J2}, $\mathcal{N}_{ij}=1$. 
    Note again that for any $i \in 
    \mathcal{J}_3$, by \eqref{eq:J3}, $\mathcal{N}_{ij}=0$. 
    Hence,  if  $\mathcal{N}_{sj}=-1$, then the $j$-th reaction has the form \eqref{eq:form1}. 
Thus, by \eqref{eq:form2} and \eqref{eq:form1},  the polynomial $f_s$ defined in \eqref{eq:sys} has the  form \eqref{eq:f1}.
\end{proof}
\begin{lemma}\label{lm:Pc+ ne emptyset}
 Consider a zero-one network $G$ with a rank-one stoichiometric matrix $\mathcal{N}$. Let ${\mathcal J}_i$ $(i=1, 2, 3)$ be defined as in \eqref{eq:J1}--\eqref{eq:J3}.
 For any $c\in\mathbb{R}^{s-1}$, 
 $\mathcal{P}_c^+\ne \emptyset$ if and only if for any $k\in \mathcal{J}_2\cup \mathcal{J}_3$, $c_k>0$,  and for any $(i,j)\in\mathcal{J}_1\times\mathcal{J}_2$, $c_i+ c_j>0$.  
\end{lemma}
\begin{proof}
   `` $\Rightarrow$":
For any $k \in \mathcal{J}_2\cup \mathcal{J}_3$, by $\mathcal{P}_c^+ \ne \emptyset$ and by \eqref{eq:J2}--\eqref{eq:J3}, we have $c_k >0$.
For any $i \in \mathcal{J}_1$ and  for any $j \in \mathcal{J}_2$, 
by \eqref{eq:J1} and \eqref{eq:J2}, we have 
 \begin{equation}\label{eq:con2-con1}
     x_i+x_j=c_i+c_j. 
 \end{equation}
So,  by $\mathcal{P}_c^+ \ne \emptyset$  and  by \eqref{eq:con2-con1}, we have $ c_i + c_j>0$. \\
``$\Leftarrow$":
Below, we construct a point $x^*$ such that $x^*\in \mathcal{P}_c^+$. 
Let 
\begin{numcases}{x_s^*\;:=\;}
 \frac{1}{2}(\min\limits_{k\in \mathcal{J}_2}c_k-\min\limits_{k\in \mathcal{J}_1}c_k),& if $\min\limits_{k\in \mathcal{J}_1}c_k<0$, \label{eq:1 dim x_1^* 1} \\
  \frac{1}{2}\min\limits_{k\in \mathcal{J}_2}c_k,& if  $\min\limits_{k\in \mathcal{J}_1}c_k\ge 0$, \label{eq:1 dim x_1^* 2}
\end{numcases}
\begin{align}
\label{eq:1 dim x_i^*}x_i^*&\;:=\;x_s^*+c_i, \text{ for any }i \in \mathcal{J}_1,\\
\label{eq:1 dim x_j^*}x_j^*&\;:=\;-x_s^*+c_j, \text{ for any }j \in \mathcal{J}_2,\\   
\label{eq:1 dim x_k^*}x_k^*&\;:=\;c_k, \text{ for any }k \in \mathcal{J}_3.
\end{align}
Below, we prove that 
$x^*\in {\mathbb R}_{>0}^s$. 
 Notice that for any $k \in \mathcal{J}_3$, since $c_k>0$, by \eqref{eq:1 dim x_k^*}, we get $x_k^*>0$.
So, we only need to prove that $x_s^*>0$ and for any $i \in \mathcal{J}_1\cup \mathcal{J}_2$,  $x_i^*>0$.  
\begin{enumerate}[(I)]
    \item Assume that $\min\limits_{k\in \mathcal{J}_1}c_k<0$. 
Note that  $c_j>0$ for any $j \in \mathcal{J}_2$. Then, by \eqref{eq:1 dim x_1^* 1}, we have $x_s^*>0$. 
For any $i \in \mathcal{J}_1$, substitute \eqref{eq:1 dim x_1^* 1} into \eqref{eq:1 dim x_i^*}, we have
    \begin{align}
      x_i^*&=\frac{1}{2}(\min\limits_{k\in \mathcal{J}_2}c_k-\min\limits_{k\in \mathcal{J}_1}c_k)+c_i\nonumber \\
      \label{eq:1 dim x_i^* x_1^*}&=\frac{1}{2}(\min\limits_{k\in \mathcal{J}_2}c_k+c_i)+\frac{1}{2}(c_i-\min\limits_{k\in \mathcal{J}_1}c_k).
    \end{align}
Since for any $(i,j)\in\mathcal{J}_1 \times\mathcal{J}_2$, $c_i+ c_j>0$, we have $\min\limits_{k\in \mathcal{J}_2}c_k+c_i>0$. 
Note that for any $i \in \mathcal{J}_1$, we have $c_i-\min\limits_{k\in \mathcal{J}_1}c_k\geq0$. 
Hence, by \eqref{eq:1 dim x_i^* x_1^*}, for any $i \in \mathcal{J}_1$, $x_i^*>0$.
For any $j \in \mathcal{J}_2$, by \eqref{eq:1 dim x_1^* 1} and \eqref{eq:1 dim x_j^*},
\begin{align}\label{eq:1 dim x_j^* x_1^*}
    x_j^*&=c_j-\frac{1}{2}(\min\limits_{k\in \mathcal{J}_2}c_k-\min\limits_{k\in \mathcal{J}_1}c_k)\nonumber\\
    &=\frac{1}{2}(c_j-\min\limits_{k\in \mathcal{J}_2}c_k)+\frac{1}{2}(c_j+\min\limits_{k\in \mathcal{J}_1}c_k). 
\end{align}
Note that for any $j\in \mathcal{J}_2$, we have $c_j-\min\limits_{k\in \mathcal{J}_2}c_k\ge0$.
Since for any $(i,j)\in\mathcal{J}_1 \times\mathcal{J}_2$, $c_i+c_j>0$, we have $c_j+\min\limits_{k\in \mathcal{J}_1}c_k>0$. 
Then, by \eqref{eq:1 dim x_j^* x_1^*}, for any $j \in \mathcal{J}_2$, $x_j^*>0$.
\item Assume that $\min\limits_{k\in \mathcal{J}_1}c_k\ge0$. 
Since for any $j \in \mathcal{J}_2$, $c_j>0$, by \eqref{eq:1 dim x_1^* 2}, we have $x_s^*>0$. 
By $\min\limits_{k\in \mathcal{J}_1}c_k\ge0$, for any $ i \in \mathcal{J}_1$, $c_i\ge0$. Hence, for any $ i \in \mathcal{J}_1$, by \eqref{eq:1 dim x_i^*}, we have $x_i^*>0$. 
By \eqref{eq:1 dim x_1^* 2} and \eqref{eq:1 dim x_j^*}, we have $x_j^*=c_j-\frac{1}{2}\min\limits_{k \in \mathcal{J}_2}c_k>0$ for any $j \in \mathcal{J}_2$. 
\end{enumerate} 
\end{proof}

\begin{lemma}\label{lm:1 dim no boundary}
    Consider a zero-one network $G$ with a rank-one  stoichiometric matrix $\mathcal{N}$. Suppose that all non-zero rows of $\mathcal{N}$ change signs. For any $c\in\mathbb{R}^{s-1}$, if $\mathcal{P}_c^+ \ne \emptyset$, then for any $\kappa\in \mathbb{R}^{m}_{>0}$, the network $G$ has no boundary steady states in $\mathcal{P}_c$. 
\end{lemma}
\begin{proof}
Let ${\mathcal J}_i$ $(i=1, 2, 3)$ be defined as in \eqref{eq:J1}--\eqref{eq:J3}. 
Since $\mathcal{P}_c^+ \ne \emptyset$, by Lemma \ref{lm:Pc+ ne emptyset}, 
    \begin{align}
    \label{eq:ck>0}c_k>0,\text{ for any }k\in \mathcal{J}_2\cup \mathcal{J}_3,\text{ and }\\
    \label{eq:ci<cj}c_i+ c_j>0,\text{ for any } (i,j)\in\mathcal{J}_1 \times\mathcal{J}_2.
\end{align}
\noindent Below we prove the conclusion by deducing a contradiction.
Assume that there exists $\kappa \in \mathbb{R}_{>0}^m$ such that $G$ has a boundary steady state $x$ in $\mathcal{P}_{c}$.
By \eqref{eq:J3} and \eqref{eq:ck>0}, for any $k \in \mathcal{J}_3$, $x_k>0$.
\begin{enumerate}[(I)]
    \item Assume that there exists $i \in \{s\}\cup\mathcal{J}_1$ such that $x_i=0$. 
    By \eqref{eq:J2} and \eqref{eq:ck>0}, for any $j \in \mathcal{J}_2$, we get $x_s$ and $x_j$ can not be $0$ simultaneously. 
By \eqref{eq:J1}, \eqref{eq:J2} and \eqref{eq:ci<cj}, for any $i \in \mathcal{J}_1$ and for any $j \in \mathcal{J}_2$, $x_i$ and $x_j$ can not be $0$ simultaneously.
   So, for any $j \in \mathcal{J}_2$, $x_j>0$. 
    Note that since all non-zero rows of $\mathcal{N}$ change signs, by Lemma \ref{lm:change sign fsx}, $f_s$ has the form \eqref{eq:f1}.
    Hence, we have 
    \begin{align}
       f_s(\kappa, x)=\sum\limits_{\sigma\in \tau_1}\kappa_{\sigma}\prod\limits_{t \in \sigma\cup \mathcal{J}_2}x_t>0, \nonumber
    \end{align}
    where $\tau_1 \subset 2^{\mathcal{J}_3}$ and $\kappa_{\sigma}$ are the rate constants (recall Lemma \ref{lm:change sign fsx}). 
    \item Assume that there exists $j \in \mathcal{J}_2$ such that $x_j=0$. By a similar argument with the proof of case (I), we have $f_s(\kappa, x)<0$.
\end{enumerate}
Overall, we have $f_s(\kappa, x) \ne 0$, which is contrary to the fact that $x$ is a boundary steady state in $\mathcal{P}_c$. 
\end{proof}
\begin{lemma}\label{lm: 1 dim J2 ne emptyset Pc+ compact}
     Consider a zero-one network $G$ with a rank-one  stoichiometric matrix $\mathcal{N}$. 
     Let $\mathcal{J}_i \;(i=1,2,3)$ be defined as in \eqref{eq:J1}--\eqref{eq:J3}. For any $c\in {\mathbb R}^{s-1}$, if $\mathcal{J}_2 \ne \emptyset$, then $\mathcal{P}_c$ is compact.
\end{lemma}
\begin{proof}
Suppose $x\in \mathcal{P}_c$. Notice that by \eqref{eq:pc}, 
$x\in {\mathbb R}_{\geq 0}^s$.
By \eqref{eq:J2}, since $\mathcal{J}_2 \ne \emptyset$, there exists at least one index $j \in \mathcal{J}_2$ such that  $x_s+x_{j}=c_{j}$. So, by \eqref{eq:pc}, we have $c_{j}\ge0$ and hence, we have $x_s\le c_{j}$.
By \eqref{eq:J1}, for any $i \in \mathcal{J}_1$, $x_i=x_s+c_i$. 
Hence,  we have $x_i\le c_{j}+c_i$.
By \eqref{eq:J2}, for any $i\in \mathcal{J}_2$, we have $x_s+x_i=c_i$. Then, we have $x_i\le c_i$. 
By \eqref{eq:J3}, for any $i \in \mathcal{J}_3$, $x_i=c_i$.
Therefore, $\mathcal{P}_c$ is bounded. Clearly, $\mathcal{P}_c$ is closed, and so, it is compact. 
\end{proof}
\begin{lemma}\label{lm:1 dim J2=emptyset dissipative}
 Consider a zero-one network $G$ with a rank-one stoichiometric matrix $\mathcal{N}$. 
 If all non-zero rows of $\mathcal{N}$ change signs, 
then the network $G$ is dissipative.
\end{lemma}
\begin{proof}
 Let $\mathcal{J}_i$ $(i=1, 2, 3)$ be defined as in \eqref{eq:J1}--\eqref{eq:J3}. 
By Lemma \ref{lm: 1 dim J2 ne emptyset Pc+ compact}, 
if ${\mathcal J}_2\neq \emptyset$, 
then ${\mathcal P}_c$ is compact, and hence the network $G$ is conservative. 
Hence, by Lemma \ref{lm:conservative-dissipative},  it is dissipative. Below, we prove the conclusion by assuming  that ${\mathcal J}_2=\emptyset$.
Since all non-zero rows of $\mathcal{N}$ change signs, by Lemma \ref{lm:change sign fsx}, we have 
    \begin{align}\label{eq:f1J2}
        \begin{split}
    f_s=-\sum_{\Lambda\in \tau_2}\kappa_{\Lambda}\prod_{t\in \Lambda\cup \mathcal{J}_1 \cup \{s\}}x_t+\sum_{\sigma\in \tau_1}\kappa_{\sigma}\prod_{i\in \sigma}x_i,
\end{split}
    \end{align}
    where $\tau_1$ and $\tau_2$ are non-empty subsets of $2^{\mathcal{J}_3}$, and $\kappa_{\Lambda}$ and $\kappa_{\sigma}$ are the rate constants of the reactions \eqref{eq:form2} and \eqref{eq:form1}. 

First, we will prove that for each $c\in {\mathbb R}^{s-1}$ such that $\mathcal{P}_c^+ \ne \emptyset$, there exists $M>0$ such that for any $x^* \in \mathcal{P}_{c}$ satisfying $\Vert x^* \Vert_{\infty}>M$, we have $f_s(x^*)<0$. Since $\mathcal{P}_c^+ \ne \emptyset$, by Lemma \ref{lm:Pc+ ne emptyset}, for any $k \in \mathcal{J}_3$, $c_k>0$. 
Hence, for any $x^* \in \mathcal{P}_{c}$ and for any $k \in \mathcal{J}_3$, by \eqref{eq:J3}, we have $x_k^*=c_k>0$.
Note that $\mathcal{J}_2 = \emptyset$. 
So,
if there exists a large enough real number  $M$ such that $\Vert x^* \Vert_{\infty} >M$, then there exists $i \in \mathcal{J}_1\cup\{s\}$ such that $x_i^*=\Vert x^* \Vert_{\infty}>M$. 
If $i=s$, then for any $j \in \mathcal{J}_1$, by \eqref{eq:J1}, 
$x_j^*>M+c_j$. 
If $i\ne s$, then by \eqref{eq:J1}, 
$x_s^*>M-c_i$. 
So, for any $j \in \mathcal{J}_1\setminus\{i\}$, by \eqref{eq:J1}, 
$x_j^*>M-c_i+c_j$. 
Hence,  by \eqref{eq:f1J2}, if $M$ is large enough, for any $x^* \in \mathcal{P}_{c}$ satisfying $\Vert x^* \Vert_{\infty}>M$, we have  $f_s(x^*)<0$.

Note that for any $i \in \mathcal{J}_1$, by 
\eqref{eq:J1}, we have
\begin{align}\label{eq:fi=f1}
    f_i=f_s,
\end{align}
and 
for any $k \in \mathcal{J}_3$, by \eqref{eq:J3}, we have
\begin{align}\label{eq:fk=0}
    f_k=0.
\end{align}
Note again that $\mathcal{J}_2 = \emptyset$.
Let $\omega:=(1,\ldots, 1)^\top\in\mathbb{R}^s_{>0}$ (here, all coordinates of $\omega$ are $1$). 
  Then, by \eqref{eq:fi=f1} and \eqref{eq:fk=0}, we have $\omega \cdot f(x)=(\vert \mathcal{J}_1 \vert +1)f_s(x)$. 
Hence,  there exists a vector $\omega\in\mathbb{R}^s_{>0}$ and a number $M>0$ such that $\omega\cdot f(x^*)<0$ for all $x^*\in \mathcal{P}_c$ with $\Vert x^*\Vert_{\infty}>M$.
So, by Lemma \ref{lm:Dissipative}, the network is dissipative.
\end{proof}
\begin{lemma}\label{lm:sum_partical<0}
     Consider a zero-one network $G$ with a rank-one stoichiometric matrix $\mathcal{N}$. 
     Let $h$ be the steady-state system augmented by conservation laws defined as in \eqref{eq:h}. 
   If all non-zero rows of $\mathcal{N}$ change signs, then for any $\kappa\in \mathbb{R}^{m}_{>0}$, for any $c\in \mathbb{R}^{s-1}$, and for any corresponding positive steady state $x$ in ${\mathcal P}^{+}_{c}$, we have ${\rm det}({\rm Jac}_h(\kappa,x))<0$. 
\end{lemma}
\begin{proof}
 Let $f_1,\dots, f_s$ be the polynomials defined as in \eqref{eq:sys}. 
    Let ${\mathcal J}_i$ $(i=1, 2, 3)$ be the sets defined as in \eqref{eq:J1}--\eqref{eq:J3}. By Lemma \ref{lm:change sign fsx}, we have 
    \begin{align}\label{eq:fs}
       f_s=-\sum_{\Lambda\in \tau_2}\kappa_{\Lambda}\prod\limits_{t\in \Lambda\cup \{s\}\cup\mathcal{J}_1}x_t+\sum_{\sigma\in \tau_1}\kappa_{\sigma}\prod\limits _{t\in \sigma\cup\mathcal{J}_2}x_t,
    \end{align}
    where $\tau_1$ and $\tau_2$ are non-empty subsets of $2^{\mathcal{J}_3}$ and $\kappa_{\Lambda}$ and $\kappa_{\sigma}$ are the rate constants of the reactions \eqref{eq:form2} and \eqref{eq:form1}. By \eqref{eq:fs}, we have \begin{align}\label{eq:fs_partial}
        \frac{\partial f_s}{\partial x_s}(\kappa,x)=-\sum_{\Lambda\in \tau_2}\kappa_{\Lambda}\prod\limits_{t\in \Lambda\cup\mathcal{J}_1}x_t.
    \end{align} Note that for any $i\in\mathcal{J}_1$, we have $f_i=f_s$. Then, we have 
    \begin{align}\label{eq:J1_partial}
        \frac{\partial f_i}{\partial x_i}(\kappa,x)=-\sum_{\Lambda\in \tau_2}\kappa_{\Lambda}\prod\limits_{t\in \Lambda\cup\{s\}\cup\mathcal{J}_1\setminus\{i\}}x_t,\; \text{for any}\; i\in\mathcal{J}_1.
    \end{align}
    Note that for any $i\in\mathcal{J}_2$, we have $f_i=-f_s$. So, we have 
    \begin{align}\label{eq:J2_partial}
        \frac{\partial f_i}{\partial x_i}(\kappa,x)=-\sum_{\sigma\in \tau_1}\kappa_{\sigma}\prod\limits _{t\in \sigma\cup\mathcal{J}_2\setminus\{i\}}x_t,\; \text{for any}\; i\in\mathcal{J}_2.
    \end{align}
    Note that for any $i\in\mathcal{J}_3$, we have $f_i=0$. Hence, we have 
    \begin{align}\label{eq:J3_partial}
        \frac{\partial f_i}{\partial x_i}(\kappa,x)=0,\; \text{for any}\; i\in\mathcal{J}_3.
    \end{align}
    By \eqref{eq:fs_partial}--\eqref{eq:J3_partial}, we have 
    \begin{equation}\label{eq:sum_partial f}
    \begin{aligned}
        \sum^s_{i=1}\frac{\partial f_i}{\partial x_i}(\kappa,x)=&-\sum_{\Lambda\in \tau_2}\kappa_{\Lambda}\prod\limits_{t\in \Lambda\cup\mathcal{J}_1}x_t-\sum_{i\in\mathcal{J}_1}\sum_{\Lambda\in \tau_2}\kappa_{\Lambda}\prod\limits_{t\in \Lambda\cup\{s\}\cup\mathcal{J}_1\setminus\{i\}}x_t\\
        &-\sum_{i\in\mathcal{J}_2}\sum_{\sigma\in \tau_1}\kappa_{\sigma}\prod\limits _{t\in \sigma\cup\mathcal{J}_2\setminus\{i\}}x_t.
    \end{aligned}
    \end{equation}
    By \eqref{eq:sum_partial f}, for any $\kappa \in \mathbb{R}^m_{>0}$ and for any   corresponding positive steady state $x \in \mathbb{R}^s_{>0}$, we have $\sum^s_{i=1}\frac{\partial f_i}{\partial x_i}(\kappa,x)<0$. 
    Since the network $G$ is one-dimensional, by \cite[Proposition 5.3]{WiufFeliu_powerlaw}, 
 ${\rm det}({\rm Jac}_h)=\sum\limits^s_{i=1}\frac{\partial f_i}{\partial x_i}$. 
 Therefore, for any $\kappa\in \mathbb{R}^{m}_{>0}$, for any $c\in \mathbb{R}^{s-1}$, and for any corresponding positive steady state $x$ in ${\mathcal P}^{+}_{c}$, we have ${\rm det}({\rm Jac}_h(\kappa,x))<0$. 
\end{proof}

\noindent\textbf{proof of Theorem \ref{thm:one-dimensional}}
\begin{enumerate}[(I)]
    \item Let $f_1,\dots, f_s$ be the polynomials defined as in \eqref{eq:sys}. We assume that there exists $i\in\{1,\dots,s\}$ such that $\mathcal{N}_i$ does not change the sign. 
So, by \eqref{eq:sys}, the coefficients of terms in $f_i(x)$ are all negative or all positive. 
    Thus, for any $\kappa\in \mathbb{R}^{m}_{>0}$, $f_i(x)=0$ has no positive solutions. Therefore, $G$ admits no positive steady states.
    \item 
    \begin{enumerate}[(i)]
        \item By Lemma \ref{lm:Pc+ ne emptyset}, $\mathcal{P}_c^+=\emptyset$. 
So, of course, $G$ has no positive steady states in $\mathcal{P}_c^+$.
\item By Lemma \ref{lm:Pc+ ne emptyset}, 
 $\mathcal{P}_c^+\ne \emptyset$.  
By Lemma \ref{lm:1 dim no boundary}, there are no boundary steady states in $\mathcal{P}_c$. 
By  Lemma \ref{lm:1 dim J2=emptyset dissipative}, the network $G$ is dissipative.  
By Lemma \ref{lm:sum_partical<0}, for any $\kappa\in \mathbb{R}^{m}_{>0}$ and for any corresponding positive steady state $x$ in ${\mathcal P}^{+}_{c}$, we have  ${\rm det}({\rm Jac}_h(\kappa, x))<0.$
Hence, by Theorem \ref{thm:deter_sign}, for any $\kappa\in \mathbb{R}^{m}_{>0}$, the network $G$ has exactly one positive steady state in $\mathcal{P}_c$, and this positive steady state is nondegenerate. 
 By Lemma \ref{lm:stable}, the nondegenerate positive steady state is stable.
\end{enumerate}
\end{enumerate}

\subsection{Two-dimensional zero-one networks}\label{sec:2d 3s}

In this section, we prove the main result Theorem \ref{thm:twospecies} as follows. Notice that a two-dimensional network has at least two species, and  Theorem \ref{thm:twospecies} states a conclusion for the two-dimensional networks with up to three species. 
First, we show that the conclusion of Theorem \ref{thm:twospecies} holds for the two-dimensional zero-one networks with two species in Lemma \ref{lm:twospecies}. 
Second, we show that the conclusion of Theorem \ref{thm:twospecies} holds for the two-dimensional zero-one networks with three species in Lemma \ref{lm:3-species_2-dimensional}. 
Later, we can naturally complete the proofs of Theorem \ref{thm:twospecies} and its corollaries.

\begin{lemma}\label{lm:twospecies} 
A two-species zero-one network with a rank-two stoichiometric matrix either admits no multistationarity or  only admits degenerate positive steady states.
\end{lemma}
\begin{proof}
Notice that all possible zero-one reactions with two species  are listed below:
\begin{align}
&0\xrightarrow{\kappa_1}X_1,
&&X_1\xrightarrow{\kappa_4}0,
&&X_2\xrightarrow{\kappa_7}0,&&X_1+X_2\xrightarrow{\kappa_{10}}0,\notag\\
&0\xrightarrow{\kappa_2}X_2,
&&X_1\xrightarrow{\kappa_5}X_2,
&&X_2\xrightarrow{\kappa_8}X_1,
&&X_1+X_2\xrightarrow{\kappa_{11}}X_1,\label{eq:2s2r_network} \\
&0\xrightarrow{\kappa_3}X_1+X_2,
&&X_1\xrightarrow{\kappa_6}X_1+X_2,
&&X_2\xrightarrow{\kappa_9}X_1+X_2,
&&X_1+X_2\xrightarrow{\kappa_{12}}X_2.\notag
\end{align}
Note that any two-dimensional zero-one network $G$ consists of some of the above twelve reactions. 
The steady-state system $f$ corresponding to the network $G$ defined in \eqref{eq:sys} is:
\begin{equation}
\begin{aligned}
f_1&=-\kappa_{10}x_1x_2-\kappa_{12}x_1x_2-\kappa_4x_1-\kappa_5x_1+\kappa_8x_2+\kappa_9x_2+\kappa_1+\kappa_3,\\
f_2&=-\kappa_{10}x_1x_2-\kappa_{11}x_1x_2+\kappa_5x_1+\kappa_6x_1-\kappa_7x_2-\kappa_8x_2+\kappa_2+\kappa_3.\label{eq:ode 12}
\end{aligned}
\end{equation}
Notice that $\kappa\in \mathbb{R}^{12}_{\ge0}$ in \eqref{eq:ode 12} (if any reaction in \eqref{eq:2s2r_network} does not appear in the network $G$, then the corresponding $\kappa_i$ can be considered as $0$).  
Solving $x_1$ from $f_1(x)=0$, we get
\begin{equation}\label{eq:two two x_1}
x_1=\frac{\kappa_1+\kappa_3+\kappa_8x_2+\kappa_9x_2}{\kappa_4+\kappa_5+\kappa_{10}x_2+\kappa_{12}x_2}.
\end{equation}
First, we clarify that the denominator   of the right-hand side of \eqref{eq:two two x_1} can not be zero for positive steady states  by deducing a contradiction. 
If  $\kappa_4+\kappa_5+\kappa_{10}x_2+\kappa_{12}x_2=0$, then by $\kappa\in \mathbb{R}^{12}_{\ge0}$ and by $x\in \mathbb{R}^{2}_{>0}$, we must have $\kappa_4=\kappa_5=\kappa_{10}=\kappa_{12}=0$. 
So, by $f_1(x)=0$, we have $\kappa_1+\kappa_3+\kappa_8x_2+\kappa_9x_2=0$. 
Hence, similarly, we have  $\kappa_1=\kappa_3=\kappa_8=\kappa_9=0$. 
Thus, the network has four reactions at most as follows.\\ 
\begin{table}[!htbp]
    \centering
    \begin{tabular}{cccc}      $0\xrightarrow{\kappa_2}X_2,$ & $X_1\xrightarrow{\kappa_6}X_1+X_2,$ &$X_2\xrightarrow{\kappa_7}0,$ &$X_1+X_2\xrightarrow{\kappa_{11}}X_1$.
    \end{tabular}
\end{table}
\noindent Notice that the network is one-dimensional. It is contrary to the hypothesis that the network is two-dimensional. Hence, \eqref{eq:two two x_1} is well-defined. 

We substitute \eqref{eq:two two x_1} into $f_2(x)=0$, and we get 
\begin{align}\label{eq:quadratic function of x_2^*}
&((\kappa_{10}+\kappa_{11})(\kappa_8 + \kappa_9)+(\kappa_7+\kappa_8)(\kappa_{10}+\kappa_{12})){x_2}^2
+((\kappa_1+\kappa_3)(\kappa_{10}+\kappa_{11}) \nonumber\\
&-(\kappa_5+\kappa_6)(\kappa_8+\kappa_9)
-(\kappa_2+\kappa_3)(\kappa_{10}+\kappa_{12})+(\kappa_7+\kappa_8)(\kappa_4+\kappa_5))x_2 \nonumber \\
&-(\kappa_5+\kappa_6)(\kappa_1+\kappa_3)-(\kappa_2+\kappa_3)(\kappa_4+\kappa_5)=0. 
\end{align}
We denote by $C_1$, $C_2$ and $C_3$ the coefficients of the quadratic term, linear term and constant term of \eqref{eq:quadratic function of x_2^*} w.r.t. $x_2$, respectively.  
Note that the network $G$ consists of some of  the twelve reactions in \eqref{eq:2s2r_network}. 
So, we can classify $\{\kappa_1,\dots,\kappa_{12}\}$ into the two following sets:
\begin{align}
    K_1&\;:=\;\{i\in\{1,\dots,12\}\mid \text{the } i \text{-th reaction appears in } G\}, \nonumber\\
    K_2&\;:=\;\{i\in\{1,\dots,12\}\mid  \text{the } i \text{-th reaction does not appear in } G\}. \nonumber
\end{align}
 For any $i \in\{1,2,3\}$, we define
\begin{align}
\tilde{C}_i\;:=\;C_i|_{\kappa_i=0,\text{ for any }i \in K_2}. \nonumber
\end{align}
We prove the conclusion by discussing the following two cases. 
\begin{enumerate}[(I)]
    \item Assume that  $\tilde{C}_1$ is not a zero polynomial. 
    Notice that for any $i \in K_1$, $\kappa_i \in\mathbb{R}_{>0}$. Thus,  by \eqref{eq:quadratic function of x_2^*},  we have $\tilde{C}_1>0$ and $\tilde{C}_3\le0$. 
Hence, we have $\tilde{C}_2^2-4\tilde{C}_1\tilde{C}_3\ge0$. 
If $\tilde{C}_2^2-4\tilde{C}_1\tilde{C}_3>0$, then by the Vieta's formulas, the quadratic equation \eqref{eq:quadratic function of x_2^*} has one positive solution. 
Hence, the network $G$ has only one nondegenerate positive steady state.
If $\tilde{C}_2^2-4\tilde{C}_1\tilde{C}_3=0$, then $\tilde{C}_2=\tilde{C}_3=0$. 
Thus, by \eqref{eq:quadratic function of x_2^*}, the network $G$ has no positive steady states. 
So, the network $G$ admits no multistationarity. 
    \item Assume that  $\tilde{C}_1$ is a zero polynomial. 
    \begin{enumerate}[(i)]
    \item Assume that  $\tilde{C}_3$ is not a zero polynomial.
    Note again for any index $i \in K_1$,  $\kappa_i \in\mathbb{R}_{>0}$. 
    Thus, by \eqref{eq:quadratic function of x_2^*}, $\tilde{C}_3<0$. 
     For any rate-constant vector $\kappa$,  if $\tilde{C}_2\le0$, then by \eqref{eq:quadratic function of x_2^*}, the network $G$ has no positive steady states. 
    If $\tilde{C}_2>0$,  
    then by \eqref{eq:quadratic function of x_2^*},  the network $G$ admits one nondegenerate positive steady state. 
    Hence, the network $G$ admits no multistationarity. 
    \item Assume that  $\tilde{C}_3$ is a zero polynomial. 
    For any rate-constant vector $\kappa$, if $\tilde{C}_2=0$, then by \eqref{eq:quadratic function of x_2^*}, the network $G$ only has degenerate positive steady states. 
    If $\tilde{C}_2\ne0$, then by \eqref{eq:quadratic function of x_2^*}, the network $G$ has no positive steady states. 
    Therefore, the network $G$ only admits degenerate positive steady states. 
    \end{enumerate}  
\end{enumerate} 
\end{proof}
\begin{remark}
We remark that from the proof of   Lemma \ref{lm:twospecies}, we can see that it is indeed possible for a two-dimensional two-species zero-one network  to admit only degenerate positive steady states when both $\tilde{C}_1$ and  $\tilde{C}_3$ are zero polynomials. At this time, there will be always infinitely many degenerate positive steady states if the values of  rate constants make $\tilde{C}_2$ vanish. 
  For instance, if the $4,\;5,\;6,\;7,\;8,\;9$-th reactions in \eqref{eq:2s2r_network} do not appear in the network, then by \eqref{eq:ode 12}, it is straightforward to see that the network has infinitely many degenerate positive steady states  if $\kappa_1=\kappa_2$ and $
  \kappa_{11}=\kappa_{12}$.  
\end{remark}
\begin{lemma}\label{lm:3-species_2-dimensional}
A two-dimensional zero-one network  with three species either admits no multistationarity or  only admits degenerate positive steady states.
\end{lemma}
\begin{remark}
Notice that all two-dimensional zero-one networks with three species can be enumerated. 
By checking all these networks by {\tt Maple}, we find  that if  a two-dimensional zero-one network with three species only admits degenerate positive steady states, then the network admits infinitely many degenerate positive steady states. The supporting codes are available online (\href{https://github.com/YueJ13/network/blob/main/degenerate}{https://github.com/YueJ13/network/blob/main/degenerate}).
\end{remark}

 Unfortunately, the proof of   Lemma \ref{lm:3-species_2-dimensional}   is more technical than that of  Lemma \ref{lm:twospecies}. 
So, we will do it  in the following several subsections. 
First, in Section \ref{subsubsection claw}, we define the ``maximum network"  (see Definition \ref{def:maximum network}), and we show that  there are only five types of conservation laws for a maximum three-species network  with a rank-two stoichiometric matrix,  see Lemma \ref{lm:3s aibi}. 
Second, in Section \ref{subsubsec:sign jach}, we show that  a three-species zero-one network with dimension two either only admits degenerate positive steady states, or if the network admits a nondegenerate positive steady state, then the sign of  ${\rm det}({\rm {Jac}}_h)$ at the steady state will be always positive,  see Lemma \ref{lm:jac>0}. 
Third, in Section \ref{subsubsection g1g2g3}, we prove that any maximum  three-species network with a rank-two stoichiometric matrix admits at most one nondegenerate positive steady state, see Lemma \ref{lm:3-species_maximum}. 
Derived from the above results, we complete the proof of Lemma \ref{lm:3-species_2-dimensional} in Section \ref{subsec:2d 3s main pro} by using the inheritance of nondegenerate multistationarity/multistability. 

Notice that once Lemma \ref{lm:3-species_2-dimensional} is proved, Theorem \ref{thm:twospecies} and its corollaries will naturally be proved as follows. 

\noindent
\textbf{proof of Theorem \ref{thm:twospecies}}\;\;
The conclusion follows from  Lemma \ref{lm:twospecies} and Lemma \ref{lm:3-species_2-dimensional}. 

\noindent\textbf{proof of Corollary \ref{cor:two dim}}\;\;
The conclusion follows from Theorem \ref{thm:twospecies} and Lemma \ref{lm:zero-one stable}. 

\noindent\textbf{proof of Corollary \ref{cor:three dim}}\;\;
By Theorem \ref{thm:one-dimensional}, any one-dimensional zero-one network admits no multistationarity. By Corollary \ref{cor:two dim},  any two-dimensional zero-one network with up to three species admits no nondegenerate multistationarity/multistability.

\subsubsection{Compatibility classes of maximum networks}\label{subsubsection claw}
For any three-species zero-one network $G$ with a rank-two stoichiometric matrix $\mathcal{N}$,
assume that the second and the third rows of  $\mathcal{N}$ are linearly independent. Recall that for any $i \in \{1,2,3\}$, we denote by $\mathcal{N}_i$ the $i$-th row of $\mathcal{N}$. Thus, we have 
\begin{equation}\label{eq:3s Ni_Ns_Ns-1}
    \mathcal{N}_{1}=a\mathcal{N}_{2}+b\mathcal{N}_3,
\end{equation}
where $a,\;b\in \mathbb{R}$. Then, the steady-state polynomial $f_1$ defined in \eqref{eq:sys} can be written as 
\begin{equation}\label{eq:3s f1_fs}
    f_1=af_{2}+bf_3. \nonumber
\end{equation}
So, the conservation law of the network can be written as
\begin{align}\label{eq:3s 0110cl}
 x_1&=a x_{2}+b x_{3}+c,
\end{align}
where $c \in \mathbb{R}$. 

\begin{definition}\label{def:maximum network}
    Consider an $r$-dimensional zero-one network $G$ with $s$ species  (denoted by $X_1, \dots, X_s$). If the dimension becomes $r+1$ when we add any zero-one reaction with at most $s$ species (these species belong to $\{X_1, \ldots, X_s\}$) into $G$, then we say the network $G$ is a \defword{maximum $s$-species network} (or, simply a \defword{maximum network} if the number of species is clear from the context).
\end{definition}

The main goal of this section is to prove that  there are five types of conservation laws for a maximum three-species network (Lemma \ref{lm:3s aibi}). 
\begin{lemma}\label{lm:3s aibi}
Consider a maximum three-species reaction network $G$ with a rank-two stoichiometric matrix $\mathcal{N}$. Suppose that the conservation law $x_1=ax_{2}+bx_3+c$ is defined as in \eqref{eq:3s 0110cl}.
We can always get another conservation law by relabeling the species 
 as $\Bar{X}_1,\Bar{X}_2,\Bar{X}_3$ (the corresponding concentration variables are $\Bar{x}_1,\Bar{x}_2,\Bar{x}_3$) such that the conservation law after relabeling is $\Bar{x}_1=\Bar{a}\Bar{x}_{2}+\Bar{b}\Bar{x}_3+\Bar{c}$, where 
 \begin{align}
     (\vert \Bar{a} \vert, \vert \Bar{b} \vert)\in \{(1,0),(0,1),(0,0),(\frac{1}{2},\frac{1}{2}),(1,1)\}. \nonumber
 \end{align}
 \end{lemma}

In the rest of this subsection, we prove Lemma \ref{lm:3s aibi}. For any three-species zero-one network $G$ with a rank-two stoichiometric matrix $\mathcal{N}$, we define a $2 \times m$ submatrix formed by the last two linearly independent rows of $\mathcal{N}$
\begin{equation}\label{eq:3s N^*}
    \mathcal{N}^*\;:=\;
    \begin{pmatrix}
        \mathcal{N}_{2}\\
        \mathcal{N}_3
    \end{pmatrix}.
\end{equation}
Since we have assumed that $\mathcal{N}_{2}$ and $\mathcal{N}_3$ are linearly independent, the above matrix $\mathcal{N}^*$ is  rank-two. 
For any index $i \in \{1,\dots,m\}$, we denote by $col_i(\mathcal{N}^*)$ and  $col_i(\mathcal{N})$ the $i$-th column of $\mathcal{N}^*$ and the  $i$-th column of $\mathcal{N}$, respectively. 
Then, by \eqref{eq:3s Ni_Ns_Ns-1} and \eqref{eq:3s N^*}, for any $i \in \{1,\dots,m\}$, we have 
    \begin{equation}\label{eq:3s F N^* N}
        F col_i(\mathcal{N}^*)=col_i(\mathcal{N}),
    \end{equation}
    where
    \begin{equation}
        F\;:=\;
        \begin{pmatrix}
            a & b\\
            1 & 0\\
            0 & 1
        \end{pmatrix}. \nonumber
    \end{equation}
\begin{lemma}\label{lm:3s mathscrCN}
    Consider a three-species zero-one reaction network $G$ with a rank-two stoichiometric matrix $\mathcal{N}$. Let $\mathcal{N}^*$ be the submatrix formed by the last two  rows of $\mathcal{N}$, see \eqref{eq:3s N^*}. Then, all the columns of $\mathcal{N}^*$ are in the following set 
\begin{equation}\label{eq:3s N^* 8}
    \mathscr{C}(\mathcal{N}^*)
    \;:=\;\{(0,1)^\top, (0,-1)^\top,(1,0)^\top, (-1,0)^\top, (1,1)^\top, (-1,-1)^\top,(1,-1)^\top, (-1,1)^\top\}.
\end{equation}
\end{lemma}
\begin{proof}
Note that $\mathscr{C}(\mathcal{N}^*)=\{0, 1, -1\}^2\backslash\{(0,0)^\top\}$.
Since the network is zero-one, all the elements of $\mathcal{N}^*$ are in $\{0,1,-1\}$. 
Hence, we only need to prove that $(0,0)^\top$ can not be a column vector of $\mathcal{N}^*$. 
In fact, if there exists index $j \in \{1,\dots,m\}$ such that $col_j(\mathcal{N}^*)$ is $(0,0)^\top$, then, by \eqref{eq:3s F N^* N}, we get $col_j(\mathcal{N})$ is $(0,0,0)^\top$. It is contrary to the definition of the reaction network.
\end{proof}
\begin{lemma}\label{lm:3s no column in N*}
Consider a maximum three-species reaction network $G$ with a rank-two stoichiometric matrix $\mathcal{N}$. Let $\mathcal{N}^*$ be the submatrix formed by the last two rows of $\mathcal{N}$, see \eqref{eq:3s N^*}. Suppose that the conservation law $x_1=ax_{2}+bx_3+c$ is defined as in \eqref{eq:3s 0110cl}. Then, we have the following statements. 
    \begin{enumerate}[(I)]
        \item For any $\ell:=(\ell_1, \ell_2)^\top\in \mathscr{C}(\mathcal{N}^*)$ (see \eqref{eq:3s N^* 8}), it is a  column vector of $\mathcal{N}^*$ if and only if         $a\ell_1+b\ell_2 \in \{-1,0,1\}$. \label{lm:3s exist column N^*}
        \item 
        $(\vert a\vert,\vert b\vert)\in\{(1,0),(0,1),(0,0),(2,1),(1,2),(\frac{1}{2},\frac{1}{2}),(1,1)\}$.\label{lm:3s ab9}
    \end{enumerate}
\end{lemma}
\begin{proof}
\begin{enumerate}[(I)]
\item 
For any $\ell=(\ell_1, \ell_2)^\top\in \mathscr{C}(\mathcal{N}^*)$, by \eqref{eq:3s F N^* N}, it is a column vector of $\mathcal{N}^*$ if and only if $F\ell$ is a column vector of $\mathcal{N}$. 
Hence,  we only need to prove that $F\ell$ is a column vector of $\mathcal{N}$ if and only if 
        $a\ell_1+b\ell_2 \in \{-1,0,1\}$. 
Since ${\rm rank}(\mathcal{N}^*)=2$, we assume that $col_1(\mathcal{N}^*)$ and $col_2(\mathcal{N}^*)$ are linearly independent.
Thus, by \eqref{eq:3s F N^* N}, we know that  $col_1(\mathcal{N})$ and $col_2(\mathcal{N})$ are linearly independent. 
Note that the network $G$ is a maximum network. 
Therefore, $F\ell$ is a column vector of $\mathcal{N}$ if and only if it satisfies the following two conditions.
    \begin{enumerate}[(i)]
        \item $F\ell$ is a linear combination of $col_1(\mathcal{N})$ and $col_2(\mathcal{N)}$. \label{3s cona}
        \item All the elements of $F\ell$ are in $\{-1,0,1\}$. \label{3s conb}
    \end{enumerate}
    We point out that $F\ell$ satisfies (i).
In fact, it is straightforward to check by \eqref{eq:3s N^* 8} that any vector in $\mathscr{C}(\mathcal{N}^*)$ is a linear combination of any two linearly independent vectors in $\mathscr{C}(\mathcal{N}^*)$.  
Notice that by Lemma \ref{lm:3s mathscrCN},  $col_1(\mathcal{N}^*), col_2(\mathcal{N}^*)\in \mathscr{C}(\mathcal{N}^*)$. 
So, there exist $r_1,\;r_2 \in \mathbb{R}$ such that $\ell=r_1col_1(\mathcal{N}^*)+r_2col_2(\mathcal{N}^*)$. 
Then, by \eqref{eq:3s F N^* N}, we have 
    \begin{align}
        F\ell=r_1col_1(\mathcal{N})+r_2col_2(\mathcal{N}). \nonumber
    \end{align} 
On the other hand, it is obvious that $F\ell$ satisfies the condition (ii) if and only if  $a\ell_1+b\ell_2 \in \{-1,0,1\}$ since the first coordinate of $F\ell$ is $a\ell_1+b\ell_2$. Thus, $F\ell$ is a column vector of  $\mathcal{N}$ if and only if $a\ell_1+b\ell_2 \in \{-1,0,1\}$.
\item In order to prove the conclusion, we prove by the following steps. 
 \begin{enumerate}[(1)]
\item  \label{lm:3s a0b1} We prove that if $a=0$ and $b \ne 0$, then $\vert b \vert\;= 1$, and if  $a \ne 0$ and $b=0$, then $\vert a \vert\;=1$. 
In fact, if $a=0$, then by \eqref{eq:3s Ni_Ns_Ns-1}, we have $\mathcal{N}_1=b\mathcal{N}_3$.
Note that all the elements in $\mathcal{N}_1$ and $\mathcal{N}_3$ are in $\{0,1,-1\}$ since the network is zero-one. So, if $b \ne 0$, then 
$\vert b \vert\;=1$.
The other half of the statement holds by symmetry. 
    \item \label{lm:3s ab1}
    We prove that  if $\vert a \vert \;\not\in \{0, 1\}$, then  $(1,0)^\top$ or $(-1,0)^\top$ is not a column vector of $\mathcal{N}^*$, 
        and if $\vert b \vert \;\not\in \{0, 1\}$, then  $(0,1)^\top$ or $(0,-1)^\top$ is not a column vector of $\mathcal{N}^*$. 
 In fact,   if $\vert a \vert\; \not\in \{0,1\}$, then
    $a = a\cdot 1 + b \cdot 0\not\in \{0,1,-1\}$.
    Thus, by (\ref{lm:3s exist column N^*}), we know that $(1,0)^\top$ is not a column vector of $\mathcal{N^*}$. 
    Similarly, one can argue that $(-1,0)^\top$ is not a column vector of $\mathcal{N}^*$.
     The other half of the statement holds by symmetry. 
    \item \label{lm:3s ab3} We prove that if $\vert a \vert \;\not\in \{0, 1\}$ and $\vert b \vert \;\not\in \{0, 1\}$, then $a+b\in \{0, 1, -1\}$ and $a-b\in \{0,1,-1\}$.  In fact, if $\vert a \vert \;\not\in \{0, 1\}$ and $\vert b \vert \;\not\in \{0, 1\}$,  then by (\ref{lm:3s ab1}),  no column vectors in the set $\{(1,0)^\top,(-1,0)^\top,(0,1)^\top,(0,-1)^\top\}$  appear in $\mathcal{N}^*$. 
 Then, by Lemma \ref{lm:3s mathscrCN},  the column vectors of $\mathcal{N}^*$ can only be in the following set 
 \begin{align}\label{eq:3s CN-}
     &\mathscr{C}(\mathcal{N}^*) \setminus \{(1,0)^\top,(-1,0)^\top,(0,1)^\top,(0,-1)^\top\}\nonumber\\
     =&\{(1,1)^\top, (-1,-1)^\top,(1,-1)^\top, (-1,1)^\top\}.
 \end{align}
    Since the matrix $\mathcal{N}^*$ is rank-two, there exist two linearly independent column vectors in $\mathcal{N}^*$, which are from $\{(1,1)^\top,(-1,-1)^\top\}$ and $\{(1,-1)^\top,(-1,1)^\top\}$, respectively.  By the fact that the network $G$ is maximum, all the vectors generated by these two column vectors are column vectors of  $\mathcal{N}^*$. Thus, all the vectors in the set \eqref{eq:3s CN-} are the column vectors of $\mathcal{N}^*$.
 Without loss of generality,  assume that $col_1(\mathcal{N}^*)$ is $(1,1)^\top$ and $col_2(\mathcal{N}^*)$ is $(1,-1)^\top$. 
 By \eqref{eq:3s F N^* N}, we know that $col_1(\mathcal{N})$ and $col_2(\mathcal{N})$ are $(a+b, 1, 1)^\top$ and $(a-b,1,-1)^\top$, respectively. 
 So, since $G$ is a  zero-one network, both $a+b$ and $a-b$ are in $\{0, 1, -1\}$. 
 \item \label{lm:3s ab4}
 We prove that   if $\vert a \vert\; >1$ and $ b  \neq 0$, then $\vert b \vert\; = 1$,  and if  $ a  \neq 0$ and $\vert b \vert\; >1$, then $\vert a \vert \;= 1$. 
 We prove the conclusion by deducing a  contradiction. 
 We assume that $\vert a \vert\;>1$ and $\vert b \vert\; \not\in \{0, 1\}$. By  (\ref{lm:3s ab3}), if $\vert a \vert\;\not\in \{0, 1\}$ and $\vert b \vert \;\not\in \{0, 1\}$, then both $a+b$ and $a-b$ are in $\{0,1,-1\}$. 
 Since $\vert a \vert\; > 1$ and $\vert b \vert \;\not\in \{0, 1\}$, by $a+b\in\{0,1,-1\}$,  $a$ and $b$ have different signs, and by $a-b\in \{0,1,-1\}$,  $a$ and $b$ have the same sign. It is contrary.
        So, if $\vert a \vert\; >1$ and $\vert b \vert \;\neq 0$, then the only possibility is that $\vert b \vert = 1$. 
     The other half of the statement holds by symmetry.  
        \item  \label{lm:3s ab6} We prove that if $0<\vert a \vert\;<1$ and $ b  \neq 0$, then $0< \vert b \vert \;< 1$, and if  
       $ a  \neq 0$ and  $0<\vert b \vert\;<1$, then $0< \vert a \vert\; < 1$. 
      In fact,  by (\ref{lm:3s ab4}), we know that if   
        $\vert a \vert\;\neq 0$ and  $\vert b \vert \;> 1$, then $\vert a \vert\; = 1$. 
       So, if $0<\vert a \vert\; < 1$ and $\vert b \vert \;\neq 0$, then  $0<\vert b \vert \;\leq 1$. 
        Below,  we prove that $\vert b\vert \;\neq 1$ by deducing a  contradiction. 
        Assume that $\vert b \vert\; = 1$. 
        Since $0<\vert a \vert\;<1$ and $ \vert b \vert\; = 1$, we get $\vert a+b \vert\not\in \{0, 1, -1\}$ and $\vert a-b\vert \not\in \{0, 1, -1\}$. 
          So, by (\ref{lm:3s exist column N^*}), no column vectors in $\{(1,1)^\top,(-1,-1)^\top, (1,-1)^\top,(-1,1)^\top\}$ appear in $\mathcal{N}^*$.  
          Note that $\vert a \vert\; \not\in  \{0, 1\}$. 
          Hence, by (\ref{lm:3s ab1}),  no column vectors in  the set $\{(1,0)^\top, (-1,0)^\top\}$ appear in $\mathcal{N}^*$. 
          So, by Lemma \ref{lm:3s mathscrCN},  all the column vectors of $\mathcal{N}^*$ can only be in $\{(0,1)^\top, (0,-1)^\top\}$. 
          Thus, the rank of $\mathcal{N}^*$ is one, which is contrary to the fact that $\mathcal{N}^*$ is rank-two.
          Hence, if $0<\vert a \vert\;<1$ and $ \vert b \vert \;\neq 0$, then $0< \vert b \vert\; < 1$. 
         The other half of the proof holds by symmetry.  
        \item \label{lm:3s dis cl para}    We prove that if $0<\vert a \vert\; <1$ and $0<\vert b \vert\;<1$, then $\vert a \vert\;=\vert b \vert\;=\frac{1}{2}$. 
        In fact, if $0<\vert a \vert\; <1$ and $0<\vert b \vert\;<1$, then
by (\ref{lm:3s ab3}), we know that both $a+b$ and  $a-b$  are in  $\{0, 1, -1\}$.
 \begin{itemize}
     \item If $a+b \in \{1,-1\}$ and $a-b=0$, then we get $\vert a \vert\;=\vert b \vert\;=\frac{1}{2}$.
     \item  If  
$a+b=0$ and $a-b \in \{1,-1\}$, then we get $\vert a \vert\;=\vert b \vert\;=\frac{1}{2}$. 
\item  If $a+b$ and $a-b$ are in $\{1,-1\}$, or $a+b=a-b=0$, then we can get $a \in \{0,1,-1\}$ and $b \in \{0,1,-1\}$, which is contrary to the assumption. (For instance, if $a+b=1$ and $a-b =1$, then we get $a=1$ and $b=0$.) 
 \end{itemize}
     \item \label{lm:3s Delta_3 2}
   We prove that  if $\vert a \vert \;>1$ and $ \vert b \vert\;=1$, then  $\vert a \vert\;=2$, 
       and if $ \vert a \vert\; =1$ and  $\vert b \vert \;>1$, then  $\vert b \vert\;=2$. 
   In fact,  if $a\neq 0$ and $b\ne 0$, then by \eqref{eq:3s Ni_Ns_Ns-1}, $\mathcal{N}_1$ and $\mathcal{N}_3$ are linearly independent, and  
\begin{align}
    \mathcal{N}_2=\frac{1}{a}\mathcal{N}_1-\frac{b}{a}\mathcal{N}_3. \nonumber
\end{align}
Therefore, we get a new conservation law, in which $x_2$ can be written as
\begin{align}
    x_{2}=\frac{1}{a}x_1-\frac{b}{a}x_3+c, \nonumber
\end{align}
where $c\in \mathbb{R}$. 
By $\vert a \vert\;>1$ and $\vert b \vert\;=1$, we have $0<\vert \frac{1}{a} \vert\;<1$ and $0<\vert \frac{b}{a} \vert\;<1$. By (\ref{lm:3s dis cl para}), we have $\vert a \vert \;=2$.
The other half of the statement holds by symmetry. 
  \end{enumerate}
  By (\ref{lm:3s a0b1}), if $a=0$ or $b=0$, then we have $(\vert a \vert, \vert b \vert)\in \{(0,1),(1,0)\}$. 
 We assume that $a\neq 0$ and $b\neq 0$. By (\ref{lm:3s ab6})--(\ref{lm:3s dis cl para}), if $0<\vert a\vert\;<1$ or $0<\vert b\vert\;<1$, then we get $(\vert a \vert, \vert b \vert)=(\frac{1}{2},\frac{1}{2})$. If $\vert a\vert\;=1$ and $\vert b\vert\;=1$, then 
 obviously, the conclusion holds. By (\ref{lm:3s ab4}) and (\ref{lm:3s Delta_3 2}), if $\vert a\vert\;>1$ and $b\neq 0$, then we have 
 $(\vert a \vert, \vert b \vert)=(2,1)$, and if $\vert b\vert\;>1$ and $a\neq 0$, then we have 
 $(\vert a \vert, \vert b \vert)=(1,2)$.
 \end{enumerate}
\end{proof}
\begin{lemma}\label{lm:3s get I03}
       Consider a maximum three-species reaction network $G$ with a rank-two stoichiometric matrix $\mathcal{N}$. Suppose that the conservation law $x_1=ax_{2}+bx_3+c$ is defined as in \eqref{eq:3s 0110cl}. If $(\vert a\vert,\vert b\vert)\in\{(2,1),(1,2)\}$, then we can get another  conservation law by relabeling the species as $\Bar{X}_1,\Bar{X}_2,\Bar{X}_3$ (the corresponding concentration variables are $\Bar{x}_1,\Bar{x}_2,\Bar{x}_3$) such that the conservation law after relabeling is $$\Bar{x}_1=\Bar{a}\Bar{x}_{2}+\Bar{b}\Bar{x}_3+\Bar{c},$$ where $(\vert \Bar{a} \vert, \vert \Bar{b} \vert)=(\frac{1}{2},\frac{1}{2})$.
\end{lemma}
\begin{proof}
By the symmetry of $a$ and $b$ (i.e., we can always exchange the two species $X_2$ and $X_3$ in $G$ by relabeling the species), we only need to prove the conclusion when the hypothesis is that  $(\vert a\vert,\vert b\vert)=(2,1)$. 
Note that by \eqref{eq:3s 0110cl}, we have
\begin{align}\label{eq:3s lmdelta31}
    x_{1}=ax_{2}+bx_3+c.
\end{align}
    Note that $\vert a \vert \;=2 \;(\neq 0)$.  
So, we can solve $x_{2}$ from \eqref{eq:3s lmdelta31}, and we get
 \begin{align}
    x_{2}=\frac{1}{a}x_{1}-\frac{b}{a}x_3-\frac{c}{a}, \nonumber
\end{align}
where $\vert \frac{1}{a} \vert\;=\frac{1}{2}$.
Note also that $\vert b \vert\;=1$. Hence, we have $\vert \frac{b}{a} \vert\;=\frac{1}{2}$. Relabel $x_2$, $x_1$ and $x_3$ as  $\Bar{x}_1$, $\Bar{x}_2$ and $\Bar{x}_3$. 
Then, we have the conclusion. 
\end{proof}
\noindent
\textbf{proof of Lemma \ref{lm:3s aibi}}\;\;
By Lemma \ref{lm:3s no column in N*} (\ref{lm:3s ab9}), we have 
 \begin{align}
    (\vert a\vert,\vert b\vert)\in\{(1,0),(0,1),(0,0),(2,1),(1,2),(\frac{1}{2},\frac{1}{2}),(1,1)\}. \nonumber
 \end{align}
 Note that if $(\vert a\vert,\vert b\vert)\in\{(1,0),(0,1),(0,0),(\frac{1}{2},\frac{1}{2}),(1,1)\}$, then the conclusion naturally holds by keeping the original labels.   
 If $(\vert a\vert,\vert b\vert)\in\{(2,1),(1,2)\}$, then the conclusion follows from  Lemma \ref{lm:3s get I03}.
\subsubsection{The sign of ${\rm det}({\rm {Jac}}_h)$} \label{subsubsec:sign jach}
The goal of this section is to prove the following Lemma \ref{lm:jac>0}, which shows that the sign of ${\rm det}({\rm Jac}_h)$ at any positive steady state is generically staying the same for any two-dimensional zero-one network with three species. 
\begin{lemma}\label{lm:jac>0}
Consider a three-species zero-one reaction network $G$ with a rank-two stoichiometric matrix $\mathcal{N}$. Let $h$ be the steady-state system augmented by conservation laws defined as in \eqref{eq:h}. 
Then, either the network $G$ only admits degenerate positive steady states, or for any $\kappa \in \mathbb{R}_{>0}^m$, if there exists a  corresponding positive steady state $x\in \mathbb{R}_{>0}^3$, then we have ${\rm det}({\rm Jac}_h(\kappa,x))>0$. 
\end{lemma}

Consider a three-species zero-one network $G$.
Recall that by \eqref{eq:3s 0110cl}, the conservation law can be written as 
\begin{align}\label{eq:3s claw}
    x_1=ax_2+bx_3+c,
\end{align}
 where $c\in\mathbb{R}$.
Define 
\begin{align}
    \label{eq:maximum networks}\mathcal{G}\;:=\;&\{\text{all the two-dimensional maximum three-species networks}\}.
\end{align}
According to Lemma \ref{lm:3s aibi}, we classify the networks of $\mathcal{G}$ into three classes according to the values of $(a,b)$ as follows.
\begin{align}  \label{eq:condition2}\mathcal{G}_1\;:=\;&\{G \mid (a,b)=(\frac{1}{2},\frac{1}{2}),\;G\in\mathcal{G}\},\\  \label{eq:condition1}\mathcal{G}_2\;:=\;&\{G \mid (a,b)\in\{(1,0),(0,1),(0,0)\},\;G\in\mathcal{G}\},\\  \label{eq:condition3}\mathcal{G}_3\;:=\;&\mathcal{G}\setminus\{\mathcal{G}_1\cup \mathcal{G}_2\}. 
\end{align}
A \defword{subnetwork} of the network $G$ consists of some reactions in $G$ 
\cite{Joshi:Shiu:Multistationary}.
Notice that any two-dimensional three-species zero-one network is a subnetwork of a certain maximum network in 
${\mathcal G}$. If this maximum network is in ${\mathcal G}_1$, then we can in fact directly compute  ${\rm det}({\rm Jac}_h(\kappa,x))$ and check its sign, see Lemma \ref{lm:jac_g1}.  
However, if this maximum network is in the set ${\mathcal G}_2\cup {\mathcal G}_3$, then checking the sign  becomes more challenging. Here, we will apply the criterion (Lemma \ref{lm:jac B jac}) and the corresponding algorithm  (Algorithm \ref{alg:checksign}) developed in Section \ref{sec:trans} for determining the sign of ${\rm det}({\rm Jac}_h(\kappa,x))$ at a positive steady state. Based on this computational method, we can prove the conclusion for all the two-dimensional subnetworks  of any network in $\mathcal{G}_2\cup\mathcal{G}_3$, see Lemma \ref{lm:jac g23 sub}. 
Above all, Lemma \ref{lm:jac>0} follows from Lemma \ref{lm:jac_g1} and Lemma \ref{lm:jac g23 sub}.
\begin{lemma}\label{lm:0 set reaction}
    Consider a maximum network $G$ with a rank-two stoichiometric matrix $\mathcal{N}$. 
    Recall that $\mathcal{N}_{ij}$ denotes the $(i,j)$-th entry of $\mathcal{N}$.
    For any $j \in \{1,\dots,m\}$, we define 
    \begin{equation}\label{eq:zeta}
   \begin{aligned}
     \zeta_{1j}&\;:=\;\{i\mid \mathcal{N}_{ij}=0, \;\;\;\;1\leq i\leq s\},\\
      \zeta_{2j}&\;:=\;\{i\mid \mathcal{N}_{ij}=-1, \;1\leq i\leq s\},\\
     \zeta_{3j}&\;:=\;\{i\mid \mathcal{N}_{ij}=1,\;\;\;\; 1\leq i\leq s\}.
  \end{aligned}  \nonumber
 \end{equation}
    Then,  the network 
    $G$ must contain the following reaction \begin{equation}\label{eq:reaction zeta 123}
        \sum\limits_{i\in \zeta_{1j} \cup \zeta_{2j}}X_i\xrightarrow[]{}\sum\limits_{i\in \zeta_{1j} \cup \zeta_{3j}}X_i.
    \end{equation}
 \end{lemma}
 \begin{proof} 
   Note that the corresponding column vector of the reaction \eqref{eq:reaction zeta 123} is the same with the $j$-th column vector of $\mathcal{N}$. Note also that the network $G$ is maximum. Therefore, the network must contain  the reaction \eqref{eq:reaction zeta 123}.
 \end{proof}
\begin{lemma}\label{lm:special_form}
  Any network $G$ in $\mathcal{G}_1$ has the following form
  \begin{equation}\label{eq:3s6r}
    \begin{aligned}    X_1+X_2+X_3\xrightleftharpoons[\kappa_2]{\kappa_1}0, \quad X_1+X_2\xrightleftharpoons[\kappa_4]{\kappa_3}X_1+X_3,
     \quad X_2\xrightleftharpoons[\kappa_6]{\kappa_{5}}X_3.
    \end{aligned}
\end{equation}
And the corresponding steady-state system $f$ defined as in \eqref{eq:sys} is 
\begin{align}
    f_1&=-\kappa_1x_1x_2x_3+\kappa_2, \label{eq:3s6r_f 1} \\
    f_2&=-\kappa_1x_1x_2x_3-\kappa_3x_1x_2+\kappa_4x_1x_3-\kappa_5x_2+\kappa_6x_3+\kappa_2, \label{eq:3s6r_f 2}\\
    f_3&=-\kappa_1x_1x_2x_3+\kappa_3x_1x_2-\kappa_4x_1x_3+\kappa_5x_2-\kappa_6x_3+\kappa_2. \label{eq:3s6r_f 3}
\end{align}
\end{lemma}
\begin{proof}
 Recall that $\mathcal{N^*}$ \eqref{eq:3s N^*} denotes the
submatrix formed by the last two row vectors (linearly independent) of the stoichiometric matrix $\mathcal{N}$. 
We denote by $col_i(\mathcal{N^*})$ and $col_i(\mathcal{N})$ the $i$-th column vector of $\mathcal{N^*}$ and $\mathcal{N}$, respectively. 
For any $G\in\mathcal{G}_1$, by \eqref{eq:condition2}, we have $(a,b)=(\frac{1}{2},\frac{1}{2})$ (here, $a$ and $b$ are the coefficients in the conservation law). 
Notice that $-a-b=-1$ and $-a+b=0$. 
So, by Lemma \ref{lm:3s no column in N*} (\ref{lm:3s exist column N^*}),
there exist ${j,k \in \{1,\dots,m\}}$
such that $col_j(\mathcal{N}^*)=(-1,-1)^\top$ and $col_k(\mathcal{N}^*)=(-1,1)^\top$.  Recall that by \eqref{eq:3s F N^* N}, for any $i \in \{1,\dots,m\}$, 
\begin{align} 
col_i(\mathcal{N})=
\begin{pmatrix}
    \frac{1}{2}&\frac{1}{2}\\
    1&0\\
    0&1
\end{pmatrix}
col_i(\mathcal{N}^*). \nonumber
\end{align}
So, we have 
$col_j(\mathcal{N})=(-1,-1,-1)^\top$ and $col_k(\mathcal{N})=(0,-1,1)^\top$. 
Hence, by Lemma \ref{lm:0 set reaction},  the network $G$ must contain the following two reactions:
\begin{align}
    &X_1+X_2+X_3\xrightarrow{}0, \nonumber\\
   &X_1+X_2\xrightarrow{}X_1+X_3. \nonumber
\end{align}
Note that $col_j(\mathcal{N})$ and $col_k(\mathcal{N})$ are linearly independent, which generate all columns of $\mathcal{N}$ since $G$ is two-dimensional. 
Hence, by the fact that the network $G$ is maximum, the network $G$ can be written as \eqref{eq:3s6r}, and
by \eqref{eq:3s6r}, the corresponding steady-state system $f$ defined as in \eqref{eq:sys} has the form \eqref{eq:3s6r_f 1}--\eqref{eq:3s6r_f 3}.
\end{proof} 
\begin{lemma}\label{lm:jac_g1}
For any two-dimensional zero-one network $G$, let $h$ be the steady-state system augmented by conservation laws defined as in \eqref{eq:h}.
If $G$ is a subnetwork of certain network in $\mathcal{G}_1$,  then for any $\kappa \in \mathbb{R}_{>0}^m$ and for any corresponding positive steady state $x\in \mathbb{R}_{>0}^3$, we have $${\rm det}({\rm Jac}_h(\kappa,x))>0.$$
\end{lemma}
\begin{proof}
If $G$ is a  subnetwork of a network $\tilde{G}$, where  $\tilde{G}\in \mathcal{G}_1$, then the reactions in $G$ are some of the reactions in $\tilde{G}$.
Note that by Lemma \ref{lm:jac_g1}, $\tilde{G}$ has the form \eqref{eq:3s6r}.
Notice that there are $6$ reactions in $\tilde{G}$. 
We define 
\begin{align}
    M\;:=\;\{i\mid \text{the } i\text{-th reaction in }\tilde{G}\text{ also appears in }G, \;1 \le i \le 6 \}. \nonumber
\end{align}
Let the corresponding steady-state system of $G$ and $\tilde{G}$ are $f$ and $\tilde{f}$, respectively. 
Then, we have 
\begin{align}\label{eq:jach g1 sub f}
    f(\kappa,x)=\tilde{f}(\kappa,x)|_{\kappa_i=0, \text{ for any } i \in \{1,\dots,6\}\setminus M}.
\end{align}
Let the steady-state systems augmented by conservation laws of $G$ and $\tilde{G}$ are $h$ and $\tilde{h}$, respectively. 
Then, we have
\begin{align}\label{eq:g1 sub detjac}
   {\rm det}({\rm Jac}_h(\kappa,x))={\rm det}({\rm Jac}_{\tilde{h}}(\kappa,x))|_{\kappa_i=0, \text{ for any } i \in \{1,\dots,6\}\setminus M}.
\end{align}
 By Lemma \ref{lm:special_form}, the steady-state system $\tilde{f}$ is
   \eqref{eq:3s6r_f 1}--\eqref{eq:3s6r_f 3}.
By \eqref{eq:3s claw} and \eqref{eq:condition2}, we know that the conservation law of $\tilde{G}$ is $x_1=\frac{1}{2}x_2+\frac{1}{2}x_3+c$.
Therefore, the steady-state system augmented by conservation laws $\tilde{h}$ defined in \eqref{eq:h} is
\begin{equation}\label{eq:g1 h}
    \begin{aligned}
        \tilde{h}_1&=x_1-\frac{1}{2}x_2-\frac{1}{2}x_3-c,\\
       \tilde{h}_2&=-\kappa_1x_1x_2x_3-\kappa_3x_1x_2+\kappa_4x_1x_3-\kappa_5x_2+\kappa_6x_3+\kappa_2,\\
        \tilde{h}_3&=-\kappa_1x_1x_2x_3+\kappa_3x_1x_2-\kappa_4x_1x_3+\kappa_5x_2-\kappa_6x_3+\kappa_2. \nonumber
    \end{aligned}
\end{equation}
Hence, it is straightforward to check that  
\begin{equation}\label{eq:3s6rdet_h}
\begin{aligned}
    {\rm det}({\rm Jac}_{\tilde{h}})=&\kappa_1(2\kappa_3x_1^2x_2 + \kappa_3x_1x_2^2 + 2\kappa_4x_1^2x_3 + \kappa_4x_1x_3^2 + 2\kappa_5x_1x_2 + \\
    &\kappa_5x_2x_3 +2\kappa_6x_1x_3 + \kappa_6x_2x_3).
\end{aligned}
\end{equation}
Below, for the network $G$, we prove that for any $\kappa \in \mathbb{R}_{>0}^{m}$ and for any corresponding positive steady state $x \in \mathbb{R}_{>0}^3$, ${\rm det}({\rm Jac}_{h}(\kappa,x))>0$. 
Since $G$ is two-dimensional, by \eqref{eq:3s6r}, we know that at least one of the two reactions corresponding to $\kappa_1$ and $\kappa_2$ appears in $G$, and at least one of the four reactions corresponding to $\kappa_3$, $\kappa_4$, $\kappa_5$ and $\kappa_6$ appears in $G$.
Since $G$ admits  a positive steady state $x$, by \eqref{eq:3s6r_f 1} and \eqref{eq:jach g1 sub f}, we know that both reactions corresponding to $\kappa_1$ and $\kappa_2$ appear in $G$. 
So, by  \eqref{eq:g1 sub detjac} and \eqref{eq:3s6rdet_h}, ${\rm det}({\rm Jac}_{h}(\kappa,x))>0$.
\end{proof}
\begin{lemma}\label{lm:jac_g23}
For any $G\in\mathcal{G}_2\cup \mathcal{G}_3$, let $h$ be the steady-state system augmented by conservation laws defined as in \eqref{eq:h}. 
Then, for any $\kappa \in \mathbb{R}_{>0}^m$ and for any corresponding positive steady state $x\in \mathbb{R}_{>0}^3$, we have ${\rm det}({\rm Jac}_h(\kappa,x))>0$. 
\end{lemma}
\begin{proof}
Notice that there  are nine networks in the set $\mathcal{G}_2\cup \mathcal{G}_3$, see Appendix \ref{S1_Appendix}.  For each network, we apply Algorithm \ref{alg:checksign} to check whether ${\rm det}({\rm Jac}_h(\kappa,x))$ changes sign at the positive steady states. Like what we have done in Example \ref{ex:jac_g23}, we find that for  all the nine networks,  for any $\kappa \in \mathbb{R}_{>0}^m$ and for any corresponding positive steady state $x\in \mathbb{R}_{>0}^3$, ${\rm det}({\rm Jac}_h(\kappa,x))>0$. 
The supporting codes  are available online (\href{https://github.com/YueJ13/network/blob/main/jach_positve}{https://github.com/\\ \noindent YueJ13/network/blob/main/jach\_positve}).
\end{proof}
\begin{lemma}\label{lm:jac g23 sub}
For any two-dimensional zero-one network $G$, let $h$ be the steady-state system augmented by conservation laws defined as in \eqref{eq:h}.
If  $G$ is a  subnetwork  of a certain network in $\mathcal{G}_2\cup \mathcal{G}_3$,  
then either $G$ only admits degenerate positive steady states, or for any $\kappa \in \mathbb{R}_{>0}^m$, if the network $G$ has a positive steady state $x\in \mathbb{R}_{>0}^3$, then  we have ${\rm det}({\rm Jac}_h(\kappa,x))>0$.
\end{lemma}
\begin{proof}
Recall that we have listed all the nine networks in $\mathcal{G}_2\cup \mathcal{G}_3$ in Appendix \ref{S1_Appendix}. For each network in $\mathcal{G}_2\cup \mathcal{G}_3$,  we can enumerate all its two-dimensional subnetworks. For each such subnetwork $G$, we apply the following computational procedure to prove the conclusion.  
First, we compute ${\mathcal F}({\mathcal N})$. If ${\mathcal F}({\mathcal N})=\emptyset$, then  $G$ admits no positive steady states. If ${\mathcal F}({\mathcal N})\neq\emptyset$, then we compute the polynomial $B(\lambda, p)$ defined as in \eqref{eq:jac B}. If $B(\lambda, p)\equiv 0$, then by \cite[Proposition 5.3]{WiufFeliu_powerlaw} and by Lemma \ref{lm:kxtohl}, 
for any $\kappa \in \mathbb{R}_{>0}^m$ and for any positive steady state $x\in \mathbb{R}_{>0}^3$, we get ${\rm det}({\rm Jac}_h(\kappa,x))=0$. That means $G$ only admits degenerate positive steady states.
If $B(\lambda, p)$ is not the zero polynomial, then we apply Algorithm
\ref{alg:checksign} to check if for any $\kappa \in \mathbb{R}_{>0}^m$ and for any corresponding positive steady state $x\in \mathbb{R}_{>0}^3$, ${\rm det}({\rm Jac}_h(\kappa,x))>0$. Our computational results show that the conclusion holds. The supporting codes are available online, see the link provided in the proof of Lemma \ref{lm:jac_g23}.
\end{proof}
\begin{remark}
Comparing the statements of Lemma \ref{lm:jac_g23} and Lemma \ref{lm:jac g23 sub}, we see that if $G$ is a maximum network in ${\mathcal G}_2\cup {\mathcal G}_3$, then once it admits a positive steady state, the steady state must be nondegenerate. However, if  $G$ is not maximum but is a subnetwork of a certain network in ${\mathcal G}_2\cup {\mathcal G}_3$, then it is possible for it to only admits degenerate positive steady states. For instance, consider the following network:
\begin{align}
    &X_1+X_2+X_3\xrightarrow{\kappa_1}X_1+X_3, 
    &&0\xrightarrow{\kappa_2}X_2, \nonumber \\
    &X_1+X_2+X_3\xrightarrow{\kappa_3}X_2, 
    &&0\xrightarrow{\kappa_4}X_1+X_3. \nonumber
\end{align}
This is a subnetwork of the network \eqref{eq:g_21} in $\mathcal{G}_2$, see Appendix \ref{S1_Appendix}. Then, similarly to Example \ref{ex:d2degenerate}, it is straightforward to check that the network only admits degenerate positive steady states.   
\end{remark}
\subsubsection{The monostationarity of the maximum networks }\label{subsubsection g1g2g3}
Recall that $\mathcal{G}$ is the set of all two-dimensional maximum three-species networks, see \eqref{eq:maximum networks}. The main result of this section is the following lemma. 
\begin{lemma}\label{lm:3-species_maximum}
For any network $G\in\mathcal{G}$, for any $c \in \mathbb{R}$, either $\mathcal{P}_c^+= \emptyset$ or for any $\kappa \in \mathbb{R}_{>0}^m$, the network $G$ has exactly one  nondegenerate positive steady state in $\mathcal{P}_c$.
\end{lemma}
Recall that $\mathcal{G}_1$, $\mathcal{G}_2$ and $\mathcal{G}_3$ defined in \eqref{eq:condition2}--\eqref{eq:condition3} give a partition of $\mathcal{G}$. In order to prove Lemma \ref{lm:3-species_maximum}, we will prove the conclusion for the networks in $\mathcal{G}_1$, $\mathcal{G}_2$ and $\mathcal{G}_3$ in Lemma \ref{lm:3s6r}, Lemma  \ref{lm:two dim 0110 at most} and Lemma \ref{lm:3s 2d}, respectively. The proof of Lemma \ref{lm:3-species_maximum} will be naturally completed based on these lemmas,  see the end of this section. 
We apply the theory  of real algebraic geometry (see Appendix \ref{subsec:real}) to prove Lemma \ref{lm:3s6r} and Lemma  \ref{lm:two dim 0110 at most}. And, We apply Theorem \ref{thm:deter_sign} to complete the proof of  Lemma \ref{lm:3s 2d}, where the challenging/
technical  part is to  prove that the networks in $\mathcal{G}_3$ are dissipative.
\begin{remark}\label{re:g1g2 not dissipative}
We remark that 
for some networks in $\mathcal{G}_1\cup \mathcal{G}_2$,  we can not check whether these networks are dissipative by Lemma \ref{lm:conservative-dissipative} or Lemma \ref{lm:Dissipative}. Hence, we can not prove the monostationarity of these networks by applying Theorem \ref{thm:deter_sign}. 
 For instance, 
 by \eqref{eq:condition2}, the conservation law of any network $G$ in   $\mathcal{G}_1$ is $x_1=\frac{1}{2}x_2+\frac{1}{2}x_3$. 
Thus, the network $G$ is not conservative. 
By Lemma \ref{lm:special_form}, the steady-state system $f$ is given in \eqref{eq:3s6r_f 1}--\eqref{eq:3s6r_f 3}. 
Hence, for each $c$ with $\mathcal{P}_c^+ \ne \emptyset$, for any $\omega=(\omega_1,\omega_2,\omega_3)\in \mathbb{R}_{>0}^3$ and for any $M>0$, the sign of 
$$\omega\cdot f=(\omega_1+\omega_2+\omega_3)(-\kappa_1x_1x_2x_3+\kappa_2)+(\omega_2-\omega_3)(-\kappa_3x_1x_2+\kappa_4x_1x_3-\kappa_5x_2+\kappa_6x_3)$$
is uncertain for all $x\in\mathcal{P}_c$ with $\Vert x \Vert_{\infty}>M$.  
So, we can not check whether the network  is dissipative by Lemma \ref{lm:conservative-dissipative} or Lemma \ref{lm:Dissipative}.
 \end{remark}
\begin{lemma}\label{lm:3s6r}
   For any network $G\in\mathcal{G}_1$, for any $c \in\mathbb{R}$ and for any $\kappa\in\mathbb{R}^m_{>0}$, the network $G$ has exactly one nondegenerate positive steady state in $\mathcal{P}_{c}$.
\end{lemma}
\begin{proof}
By Lemma \ref{lm:special_form}, the steady-state system $f$ defined in \eqref{eq:sys} has the form \eqref{eq:3s6r_f 1}--\eqref{eq:3s6r_f 3}.
 For any 
 $x=(x_1, x_2,x_3)\in \mathbb{R}_{>0}^3$, if  any $x_i$ $(i\in \{1, 2,3\})$ is large enough, then by \eqref{eq:3s6r_f 1}, we have $$f_1(\kappa, x)=-\kappa_1x_1x_2x_3+\kappa_2<0.$$
 Hence, the network $G$ admits no positive steady states at infinity.
Since $\kappa \in \mathbb{R}_{>0}^m$, for any $x\in \mathbb{R}_{\ge0}^3$, 
if there exists $i\in \{1, 2,3\}$ such that $x_i=0$,  then by \eqref{eq:3s6r_f 1}, $f_1(\kappa, x)>0$. 
So, the network $G$ admits no boundary steady states.  
By Lemma \ref{lm:jac_g1}, for any $\kappa \in \mathbb{R}_{>0}^m$ and for any corresponding positive steady state $x\in \mathbb{R}_{>0}^3$, ${\rm det}({\rm Jac}_h(\kappa,x))>0$. 
By \eqref{eq:condition2}, the conservation law is 
 $x_1=\frac{1}{2}x_2+\frac{1}{2}x_3+c$ ($c\in\mathbb{R}$). 
 So, the steady-state system augmented by conservation laws $h$ defined in \eqref{eq:h} is 
 \begin{align}
     \label{eq:g1 h1}h_1&=x_1-\frac{1}{2}x_2-\frac{1}{2}x_3-c,\\
     \label{eq:g1 h2}h_2&=f_2=-\kappa_1x_1x_2x_3-\kappa_3x_1x_2+\kappa_4x_1x_3-\kappa_5x_2+\kappa_6x_3+\kappa_2,\\
     \label{eq:g1 h3}h_3&=f_3=-\kappa_1x_1x_2x_3+\kappa_3x_1x_2-\kappa_4x_1x_3+\kappa_5x_2-\kappa_6x_3+\kappa_2.
 \end{align}
Let every coordinate of $\kappa^*$ be $1$. Let $c^*=0$.
Substitute $\kappa^*$ and $c^*$ into \eqref{eq:g1 h1}--\eqref{eq:g1 h3}, we have
        \begin{align}
        h_1&= x_1-\frac{1}{2}x_2-\frac{1}{2}x_3,\nonumber\\
    h_2&=-x_1x_2x_3-x_1x_2+x_1x_3-x_2+x_3+1,\nonumber\\
        h_3&= -x_1x_2x_3+x_1x_2-x_1x_3+x_2-x_3+1. \nonumber
        \end{align}
We solve $x_1$, $x_2$ and  $x_3$ from $h_1=h_2=h_3=0$, 
and it is straightforward to check that the network $G$ has exactly one positive steady state $x^*=(1,\;1,\;1)^\top$ in $\mathcal{P}_{c^*}$. 
Note that the conservation law of the network $G$ is $x_1=\frac{1}{2}x_2+\frac{1}{2}x_3+c$. Hence, for any $c\in \mathbb{R}$, $\mathcal{P}_c^+ \ne \emptyset$. 
By Lemma \ref{lm:same number a section} (see Appendix \ref{subsec:real}), for any $c \in \mathbb{R}$ and for any $\kappa \in \mathbb{R}_{>0}^m$, the network $G$ has exactly one nondegenerate positive steady state in $\mathcal{P}_c$. 
\end{proof}
In the next lemma, we deal with the networks in $\mathcal{G}_2$.
\begin{lemma}\label{lm:two dim 0110 at most}
For any $G\in\mathcal{G}_2$,   for any $c \in \mathbb{R}$, either $\mathcal{P}_c^+= \emptyset$ or for any $\kappa \in \mathbb{R}_{>0}^m$, the network $G$ has exactly one nondegenerate 
 positive steady state in $\mathcal{P}_c$.
\end{lemma}
\begin{proof}
In order to prove the conclusion, first we consider the conservation law. 
By the definition of $\mathcal{G}_2$ in \eqref{eq:condition1}, the conservation law of any network in $\mathcal{G}_2$ is one of the following three expressions:
\begin{align}
       \label{eq:double_con1}x_1&=x_3+c,\\
       \label{eq:double_con2}x_1&=x_2+c,\\
       \label{eq:double_con3}x_1&=c.
\end{align}
\begin{enumerate}[(I)]
    \item Assume that the conservation law of the network $G$ is \eqref{eq:double_con1}. 
First, we claim that the network $G$ has the following form:
    \begin{align}
&X_1+X_2+X_3\xrightleftharpoons[\kappa_2]{\kappa_1}X_1+X_3, &&    X_1+X_2\xrightleftharpoons[\kappa_4]{\kappa_3} X_1,&&   X_2+X_3\xrightleftharpoons[\kappa_{6}]{\kappa_{5}} X_3,\notag\\
&X_2\xrightleftharpoons[\kappa_{8}]{\kappa_{7}} 0, &&    X_1+X_2+X_3\xrightleftharpoons[\kappa_{10}]{\kappa_9}X_2,&&   X_1+X_3\xrightleftharpoons[\kappa_{12}]{\kappa_{11}}0, \label{eq:double_network}\\
&X_2\xrightleftharpoons[\kappa_{14}]{\kappa_{13}} X_1+X_3, &&
     X_1+X_2+X_3\xrightleftharpoons[\kappa_{16}]{\kappa_{15}}0.\notag
    \end{align}
\noindent And, the steady-state system $f$ defined as in \eqref{eq:sys} is
\begin{align}
    f_1=&-\kappa_9x_1x_2x_3-\kappa_{15}x_1x_2x_3-\kappa_{11}x_1x_3-\kappa_{14}x_1x_3+\kappa_{10}x_2+\kappa_{13}x_2\notag\\
&+\kappa_{12}+\kappa_{16}, \label{eq:double_f1}\\
    f_2=&-\kappa_1x_1x_2x_3-\kappa_{15}x_1x_2x_3-\kappa_3x_1x_2+\kappa_2x_1x_3+\kappa_{14}x_1x_3
    -\kappa_5x_2x_3\notag\\
    &+\kappa_4x_1-\kappa_{7}x_2
    -\kappa_{13}x_2 +\kappa_6x_3+\kappa_{8}+\kappa_{16},\label{eq:double_f2}\\
    f_3=&-\kappa_9x_1x_2x_3-\kappa_{15}x_1x_2x_3-\kappa_{11}x_1x_3-\kappa_{14}x_1x_3+\kappa_{10}x_2+\kappa_{13}x_2\notag\\
&+\kappa_{12}+\kappa_{16}\label{eq:double_f3}.
\end{align}
In fact, recall that $\mathcal{N^*}$ \eqref{eq:3s N^*} denotes the
submatrix formed by the last two linearly independent row vectors  of the stoichiometric matrix  $\mathcal{N}$. We denote by $col_i(\mathcal{N^*})$ and $col_i(\mathcal{N})$ the $i$-th column vector of $\mathcal{N^*}$ and $\mathcal{N}$, respectively. 
By \eqref{eq:double_con1}, we can get $(a,b)=(0,1)$ (here, $a$ and $b$ are the two coefficients in the conservation law $x_1=ax_2+bx_3+c$). 
Note that $a\cdot 1 + b\cdot 0=0$ and $a\cdot 0 + b\cdot 1=1$. Hence, by Lemma \ref{lm:3s no column in N*} (\ref{lm:3s exist column N^*}), there exist indices $j,k\in \{1,\dots,m\}$ such that $col_j(\mathcal{N}^*)=(1,0)^\top$ and
$col_k(\mathcal{N}^*)=(0,1)^\top$. 
Recall that by \eqref{eq:3s F N^* N}, 
\begin{align} \label{eq:double N^* N}   col_i(\mathcal{N})=
\begin{pmatrix}
    0&1\\
    1&0\\
    0&1
\end{pmatrix}
col_i(\mathcal{N}^*).
\end{align}
So, by \eqref{eq:double N^* N}, 
$col_j(\mathcal{N})$ is $(0,1,0)^\top$, and $col_k(\mathcal{N})$ is $(1,0,1)^\top$. 
Hence, by Lemma \ref{lm:0 set reaction}, there exist two reactions in the network $G$ as follows.
\begin{align}
&X_1+X_3\xrightarrow{}X_2+X_1+X_3,  \nonumber\\
&X_2\xrightarrow{}X_1+X_3+X_2.\nonumber
\end{align}
Note that $col_j(\mathcal{N})$ and $col_k(\mathcal{N})$ are linearly independent. Hence, all columns of $\mathcal{N}$ can be generated by $col_j(\mathcal{N})$ and $col_k(\mathcal{N})$. 
So, by the fact that $G$ is maximum, the network $G$ can be written as \eqref{eq:double_network}.
By \eqref{eq:double_network}, the steady-state system $f$ defined as in \eqref{eq:sys} has the form \eqref{eq:double_f1}--\eqref{eq:double_f3}.

\hspace{1em} For any $x=(x_1, x_2, x_3)\in \mathbb{R}_{>0}^3$,
if $x_{2}$ is large enough, then by \eqref{eq:double_f2}, we can get $f_{2}(x)<0$, and if $x_{1}$ or $x_{3}$ is large enough, then by \eqref{eq:double_f3}, we have $f_{3}(x)<0$. 
So, the network $G$ admits no positive steady states at infinity.
Since $\kappa\in\mathbb{R}_{>0}^{16}$, for any $x\in\mathbb{R}_{>0}^3$, if $x_1=0$ or $x_3=0$, then by \eqref{eq:double_f1}, we get $f_1(x)>0$, and if $x_2=0$, then by \eqref{eq:double_f2}, $f_2(x)>0$. 
Thus, the network $G$ admits no boundary steady states.  
By Lemma \ref{lm:jac_g23}, for any $\kappa\in \mathbb{R}^{m}_{>0}$ and for any  corresponding positive steady state $x\in\mathbb{R}^3_{>0}$, ${\rm det}({\rm Jac}_h(\kappa, x))>0$.
Recall that the conservation law is \eqref{eq:double_con1}. 
 Hence, the steady-state system augmented by conservation laws $h$ defined in \eqref{eq:h} is 
$$(h_1,h_2,h_3)=(x_1-x_3-c,f_2,f_3).$$ 
Let every coordinate of $\kappa^*$ be $1$. Let $c^*=0$. Substitute $\kappa^*$ and $c^*$ into $h$, we have
 \begin{align}
    h_1=&x_1-x_3, \nonumber\\
    h_2=&(-x_2+1)(2x_1x_3+x_1+x_3+2),\nonumber\\
    h_3=&2(-x_1x_3+1)(x_2+1).\nonumber
    \end{align}
We solve $x_1$, $x_2$ and $x_3$ from $h_1=h_2=h_3=0$, and we get 
 exactly one positive steady state $x^*=(1,\;1,\;1)^\top$ in $\mathcal{P}_{c^*}$.  Note that by \eqref{eq:double_con1}, for any $c \in \mathbb{R}$, $\mathcal{P}_c^+\neq \emptyset$. Therefore, by Lemma \ref{lm:same number a section} (see Appendix \ref{subsec:real}), for any $c \in \mathbb{R}$ and for any $\kappa \in \mathbb{R}_{>0}^m$, the network $G$ has exactly one nondegenerate positive steady state in $\mathcal{P}_c$. 
 \item Assume that the conservation law of the network $G$ is \eqref{eq:double_con2}. 
    Note that we can exchange the two species $X_2$ and $X_3$ in the network $G$ by relabeling the species. Hence, by the proof (I), we get the conclusion.
\item Assume that the conservation law of the network $G$ is 
 \eqref{eq:double_con3}.  If $c\le0$, then by \eqref{eq:double_con3},  $x_1=c\le0$ (i.e., $\mathcal{P}_c^+=\emptyset$).
 If $c>0$, then for any $c \in \mathbb{R}_{>0}$, we have $x_1=c>0$ is a constant. 
 Thus, we can consider the network $G$ as a two-dimensional zero-one networks with two species.
 By the proof of Lemma \ref{lm:twospecies}, the network $G$ has exactly one nondegenerate positive steady state in $\mathcal{P}_c$.
\end{enumerate}
\end{proof}



In the rest of the section, we study the networks in $\mathcal{G}_3$. The last bigger result is the following Lemma \ref{lm:3s 2d}.
\begin{lemma}
    \label{lm:3s 2d}
     For any $G\in\mathcal{G}_3$,  for any $c \in \mathbb{R}$, either $\mathcal{P}_c^+= \emptyset$ or for any $\kappa \in \mathbb{R}_{>0}^m$, the network $G$ has exactly one nondegenerate positive steady state in $\mathcal{P}_{c}$. 
\end{lemma}

In order to prove Lemma \ref{lm:3s 2d}, 
 we first give the possible conservation laws in Lemma \ref{lm:g3 claw}, and we provide two useful lemmas for proving the dissipativity. 
 Second, based on these lemmas, we prove that any network  in $\mathcal{G}_3$ is dissipative  in Lemma \ref{lm:g3 compact}.
 Third, we prove in Lemma \ref{lm:G3 bound} that any network in $\mathcal{G}_3$ admits no boundary steady states. Finally, we prove Lemma \ref{lm:3s 2d} by applying Theorem \ref{thm:deter_sign}.
\begin{lemma}\label{lm:g3 claw}
    For any $G\in \mathcal{G}_3$,  the conservation law  has one of the following forms:
    \begin{align}
     \label{eq:g3 -1-1}x_1&=-x_2-x_3+c,\\
      \label{eq:g3 -1/2-1/2}x_1&=-\frac{1}{2}x_2-\frac{1}{2}x_3+c,\\
        \label{eq:g3 0-1}x_1&=-x_3+c, (\text{or } x_1=-x_2+c),\\       
        \label{eq:g3 1/2-1/2}x_1&=\frac{1}{2}x_2-\frac{1}{2}x_3+c, (\text{or }x_1=-\frac{1}{2}x_2+\frac{1}{2}x_3+c),\\
        \label{eq:g3 1-1}x_1&=x_2-x_3+c, (\text{or}\; x_1=-x_2+x_3+c),\\
        \label{eq:g3 11}x_1&=x_2+x_3+c.  
    \end{align}
\end{lemma}
\begin{proof}
Recall that the conservation law has the form $x_1=ax_2+bx_3+c$, see \eqref{eq:3s claw}. 
By Lemma \ref{lm:3s aibi}, we have $(\vert a \vert, \vert b\vert)\in \{(1,0)$, $(0,1)$, $(0,0)$, $(\frac{1}{2},\frac{1}{2})$, $(1,1)$\}. 
Recall that the set $\mathcal{G}_3=\mathcal{G} \setminus \{\mathcal{G}_2\cup \mathcal{G}_1\}$  (see \eqref{eq:condition3}), where the sets $\mathcal{G}_1=\{G\mid(a,b)=(\frac{1}{2},\frac{1}{2})\}$ (see \eqref{eq:condition2})
and $\mathcal{G}_2=\{G\in \mathcal{G}\mid (a,b)\in \{(1,0),(0,1),(0,0)\}\}$, 
 (see \eqref{eq:condition1}). 
So, we have the conclusion.  
\end{proof}

\begin{lemma}\label{lm:g3 omega}
    For any  $G\in\mathcal{G}_3$, 
    let $f=(f_1, f_2, f_3)$ be the steady-state system defined as in \eqref{eq:sys}.  Assume that the conservation law has the form $x_1=ax_2+bx_3+c$. If $b<0$, then there exists $\omega\in \mathbb{R}^3_{>0}$ such that $\omega \cdot f=f_2$. 
\end{lemma}
\begin{proof}
By \eqref{eq:3s claw}, we have $f_1=af_2+bf_3$. 
So, for any $y=(y_1,y_2,y_3)^\top\in\mathbb{R}^3_{>0}$,
 we have
\begin{equation}\label{eq:y1 y3 fx 2}
    y \cdot f=(y_1a+y_2)f_2+ (y_1b+y_3)f_3.
\end{equation}
Let $y_1=-\frac{1}{b}$. 
Since $b<0$, we have $y_1>0$. 
Let $y_3=1$, and let $y_2=\vert -\frac{a}{b} \vert \;+1$. 
So, by \eqref{eq:y1 y3 fx 2}, we have
\begin{equation}\label{eq:specialyf}
\begin{aligned}
   y \cdot f&=(-\frac{a}{b}+\vert -\frac{a}{b} \vert +1)f_2. 
\end{aligned} \nonumber
\end{equation}
Let $\omega=\frac{y}{-\frac{a}{b}+\vert -\frac{a}{b} \vert +1}$. 
So, there exists $\omega\in\mathbb{R}^3_{>0}$ such that $\omega \cdot f=f_{2}$.
\end{proof}
\begin{lemma}\label{lm:xs-1 large}
Consider a network $G\in\mathcal{G}_3$. 
Assume that the conservation law is \eqref{eq:g3 0-1}, \eqref{eq:g3 1/2-1/2} or \eqref{eq:g3 1-1}. 
For any $c\in \mathbb{R}$, if $\mathcal{P}_c^+\ne \emptyset$, then for any $R>0$, there exists $M>0$ such that for any $x=(x_1, x_2, x_3)\in \mathcal{P}_{c}$ satisfying $\Vert x\Vert_{\infty}>M$, we have $x_{2}>R$.
\end{lemma}
\begin{proof}
  \begin{enumerate}[(I)]
        \item Assume that the conservation law has the form \eqref{eq:g3 0-1}. 
        Note that we can exchange the two species $X_2$ and $X_3$ in the network $G$ by relabeling the species. 
        Hence, we only need to consider the conservation law with the form $x_1=-x_3+c$. 
        Since $ \mathcal{P}_{c}^+\ne \emptyset$, $c> 0$. 
        Then, for any $x\in \mathcal{P}_{c}$, we have $x_1\le c$ and $x_3\le c$. 
    For any $R>0$, we let $M=R+c$. 
    Then for any $x\in \mathcal{P}_{c}$ satisfying $\Vert x\Vert_{\infty}>M$, we have  $x_{2}=\Vert x\Vert_{\infty}>M>R$.
    \item Assume that the conservation law has the form \eqref{eq:g3 1/2-1/2}. 
   We only need to consider the conservation law with the form $x_1=\frac{1}{2}x_2-\frac{1}{2}x_3+c$.
        For any $c\in \mathbb{R}$ and for any $R>0$, we let $M = R + 2\vert c\vert + 2c + 1$. Notice that $M>0$. 
For any $x\in \mathcal{P}_{c}$ satisfying $\Vert x\Vert_{\infty}>M$, we have the following  three cases.
\begin{enumerate}[(i)]
    \item Assume that $x_1=\Vert x\Vert_{\infty}>M$. Note that $x_3\ge0$ since $x\in \mathcal{P}_{c}$. 
    Note that the conservation law is $x_1=\frac{1}{2}x_2-\frac{1}{2}x_3+c$.
    Hence, $x_2+2c=2x_1+x_3>M$. Thus, we have $x_2>R$. 
 \item If $x_2=\Vert x\Vert_{\infty}>M$, then obviously, $x_2>R$.   
    \item If $x_3=\Vert x\Vert_{\infty}>M$, then the proof is similar to the proof of case (i). 
    \end{enumerate}
    \item Assume that the conservation law has the form \eqref{eq:g3 1-1}. The proof is similar to the proof of  case (II).  
    \end{enumerate}
\end{proof}
\begin{lemma}\label{lm:g3 compact}
     Any network in $\mathcal{G}_3$  is dissipative.
\end{lemma}
\begin{proof}
By Lemma \ref{lm:g3 claw}, for any network $G$ in $\mathcal{G}_3$, the conservation law will be one of the $6$ forms \eqref{eq:g3 -1-1}--\eqref{eq:g3 11}. Below, we prove the conclusion for each of them. 
\begin{enumerate}[(I)]
  \item  If the conservation law has the form \eqref{eq:g3 -1-1} or \eqref{eq:g3 -1/2-1/2}, then  $\mathcal{P}_c$ is a compact set. Thus, the network  is conservative. Hence, by Lemma \ref{lm:conservative-dissipative}, the network  is dissipative. 
  \item Assume that the conservation law has the form \eqref{eq:g3 0-1}.
We recall that $\mathcal{N^*}$ \eqref{eq:3s N^*} is the
submatrix formed by the last two linearly independent row vectors of the stoichiometric matrix  $\mathcal{N}$. 
We denote by $col_i(\mathcal{N^*})$ and $col_i(\mathcal{N})$ the $i$-th column vector of matrices $\mathcal{N^*}$ and $\mathcal{N}$, respectively. 
In \eqref{eq:g3 0-1}, we only need to consider the conservation law $$x_1=-x_3+c.$$ 
Note here, $(a,b)=(0,-1)$ ($a$ and $b$ are the two coefficients  in the conservation law $x_1=ax_2+bx_3+c$). 
Note that $a\cdot1+b\cdot0=0$ and $a\cdot0+b\cdot1=0$. 
Hence, by Lemma \ref{lm:3s no column in N*} (\ref{lm:3s exist column N^*}), there exist $j,k \in\{1,\dots,m\}$ such that $col_j(\mathcal{N}^*)=(1,0)^\top$ and
$col_k(\mathcal{N}^*)=(0,1)^\top$. 
Recall that by \eqref{eq:3s F N^* N}, 
\begin{align} \label{eq:3s 0-1 N^* N}   col_i(\mathcal{N})=
\begin{pmatrix}
    0&-1\\
    1&0\\
    0&1
\end{pmatrix}
col_i(\mathcal{N}^*).
\end{align}
So, by \eqref{eq:3s 0-1 N^* N}, we have  
$col_j(\mathcal{N})=$$(0,1,0)^\top$ and $col_k(\mathcal{N})=$$(-1,0,1)^\top$.
Hence, by Lemma \ref{lm:0 set reaction},  the network $G$ contains the  following two reactions:
\begin{align*}  X_1+X_3\xrightarrow{}&X_2+X_1+X_3, \\
X_1+X_2\xrightarrow{}&X_3+X_2.
\end{align*}
Note that $col_j(\mathcal{N})$ and $col_k(\mathcal{N})$ are linearly independent. Hence, all columns of $\mathcal{N}$ can be generated by $col_j(\mathcal{N})$ and $col_k(\mathcal{N})$.
Thus, by the fact that the network $G$ is maximum, the network $G$ can be written as follows:
\begin{table}[!htbp]
    \centering
    \label{eq:3s lambda1 network}
    \begin{tabular}{lll} $X_1\xrightleftharpoons[\kappa_2]{\kappa_1}X_2+X_3$, & $X_1+X_2\xrightleftharpoons[\kappa_4]{\kappa_3}X_3$, & $X_1\xrightleftharpoons[\kappa_6]{\kappa_5}X_3$, \\ $X_1+X_2\xrightleftharpoons[\kappa_8]{\kappa_7}X_3+X_2$,&
$0\xrightleftharpoons[\kappa_{10}]{\kappa_9}X_2$,&
$X_1\xrightleftharpoons[\kappa_{12}]{\kappa_{11}}X_2+X_1$,\\
$X_3\xrightleftharpoons[\kappa_{14}]{\kappa_{13}}X_2+X_3$,&
$X_1+X_3\xrightleftharpoons[\kappa_{16}]{\kappa_{15}}X_2+X_1+X_3$.
\end{tabular}  
\end{table}\\
\noindent
Hence, for the steady-state system $f=(f_1, f_2, f_3)$ defined as in \eqref{eq:sys}, we have 
\begin{align}
\label{eq:3s lambda1 f2}
    f_2=&-\kappa_{16}x_1x_2x_3-\kappa_3x_1x_2-\kappa_{12}x_1x_2+\kappa_{15}x_1x_3-\kappa_2x_2x_3\notag\\
    &-\kappa_{14}x_2x_3+\kappa_1x_1+\kappa_{11}x_1-\kappa_{10}x_2+\kappa_4x_3+\kappa_{13}x_3+\kappa_{9}.
\end{align}
Recall that the conservation law is $x_1=-x_3+c$. 
So, for any $c\in \mathbb{R}$, if $\mathcal{P}_c^+\ne \emptyset$, we have $c>0$.
For any $x^* \in \mathcal{P}_c$, since $c>0$, we have $x_1^*\le c$ and $x_3^*\le c$. 
So, by \eqref{eq:3s lambda1 f2},
there exists $R>0$ such that
$f_2(x^*)<0$ when $x_{2}^*>R$. 
By Lemma \ref{lm:xs-1 large}, for any $c\in \mathbb{R}$ and for any $R>0$, there exists a number $M>0$ such that for any $x^*\in \mathcal{P}_{c}$ satisfying $\Vert x^*\Vert_{\infty}>M$, we have $x_{2}^*>R$. 
So, for any $x^*\in \mathcal{P}_{c}$ satisfying $\Vert x^*\Vert_{\infty}>M$, we have $f_2(x^*)<0$. 
Note that $b=-1<0$. 
Thus, by Lemma \ref{lm:g3 omega} and by Lemma \ref{lm:Dissipative}, $G$ is dissipative. 
\item Assume that the conservation law has the form \eqref{eq:g3 1/2-1/2}.  In \eqref{eq:g3 1/2-1/2}, we only need to consider the conservation law with the form $x_1=\frac{1}{2}x_2-\frac{1}{2}x_3+c$. 
Note here, $(a,b)=(\frac{1}{2},-\frac{1}{2})$. 
Note that $a\cdot1+b\cdot1=0$ and $a\cdot1+b\cdot(-1)=1$. 
Hence, by Lemma \ref{lm:3s no column in N*} (\ref{lm:3s exist column N^*}), there exist $j,k\in \{1, \dots,m\}$ such that  $col_j(\mathcal{N}^*)=(1,1)^\top$ and  
$col_k(\mathcal{N}^*)=(1,-1)^\top$.
Recall that by \eqref{eq:3s F N^* N}, 
\begin{align} \label{eq:3s 1/2-1/2 N^* N}   col_i(\mathcal{N})=
\begin{pmatrix}
    \frac{1}{2}&-\frac{1}{2}\\
    1&0\\
    0&1
\end{pmatrix}
col_i(\mathcal{N}^*).
\end{align}
So, by \eqref{eq:3s 1/2-1/2 N^* N}, 
$col_j(\mathcal{N})=$$(0,1,1)^\top$ 
and $col_k(\mathcal{N})=$$(1,1,-1)^\top$.
Hence, by Lemma \ref{lm:0 set reaction}, the network $G$ contains the  follow two reactions.
\begin{align*}
X_1&\xrightarrow{}X_2+X_3+X_1, \\
   X_3&\xrightarrow{}X_1+X_2.
\end{align*}
Note that $col_j(\mathcal{N})$ and $col_k(\mathcal{N})$ are linearly independent. Hence, all columns of $\mathcal{N}$ can be generated by $col_j(\mathcal{N})$ and $col_k(\mathcal{N})$.  So, by the fact that the network is maximum, the network can be written as follows:
\begin{table}[!htbp]
    \centering
    \begin{tabular}{ccc}
        $X_3\xrightleftharpoons[\kappa_2]{\kappa_1}X_1+X_2$, &  $X_1\xrightleftharpoons[\kappa_4]{\kappa_3} X_2+X_3+X_1$,
         & $0\xrightleftharpoons[\kappa_6]{\kappa_5} X_2+X_3$.
    \end{tabular}
\end{table}\\
\noindent
Hence, for the steady-state system $f=(f_1, f_2, f_3)$ defined as in \eqref{eq:sys}, we have 
\begin{align}\label{eq:3s lambda7 f2}
    f_2=-\kappa_4x_1x_2x_3-\kappa_2x_1x_2-\kappa_6x_2x_3+\kappa_3x_1+\kappa_1x_3+\kappa_5.
\end{align}
Below, for any $c\in {\mathbb R}$, we prove that there exists a number $R>0$ such that for any $x^*=(x_1^*, x_2^*, x_3^*)\in \mathcal{P}_{c}$ satisfying $x_2^*>R$, we have $f_2(x^*)<0$. 
Note that the conservation law is $x_1=\frac{1}{2}x_2-\frac{1}{2}x_3+c$. So, 
there exists $R>0$ such that for any $x^*\in \mathcal{P}_{c}$ satisfying $x_2^*>R$, we have 
$x_1^*+\frac{1}{2}x_3^*=\frac{1}{2}x_2^*+c>0$. Hence, $x_1^*$ and $x_3^*$ can not be $0$ simultaneously. Thus, by \eqref{eq:3s lambda7 f2},  if $R$ is large enough, then  we have $f_2(x^*)<0$. 
By Lemma \ref{lm:xs-1 large},  for any $R>0$, there exists $M>0$ such that for any $x^*\in \mathcal{P}_{c}$ satisfying $\Vert x^*\Vert_{\infty}>M$, we have $x_{2}^*>R$. 
Hence, there exists $M>0$ such that for any 
     $x^*\in \mathcal{P}_{c}$ satisfying $\Vert x^*\Vert_{\infty}>M$, we can get $f_2(x^*)<0$. 
Notice that $b=-\frac{1}{2}<0$. 
Thus, by Lemma \ref{lm:g3 omega} and by Lemma \ref{lm:Dissipative}, $G$ is dissipative. 
\item  Assume that the conservation law has the form \eqref{eq:g3 1-1}. We only need to consider the conservation law ${x_1=x_2-x_3+c}$. Here, we have $(a,b)=(1,-1)$. 
We notice that $a\cdot1+b\cdot0=1$ and 
$a\cdot0+b\cdot1=-1$. 
Hence, by Lemma \ref{lm:3s  no column in N*} (\ref{lm:3s exist column N^*}), there exist $j,k \in\{1,\dots,m\}$  such that $col_j(\mathcal{N}^*)=(1,0)^\top$ and $col_k(\mathcal{N}^*)=(0,1)^\top$.
Recall that by \eqref{eq:3s F N^* N}, 
\begin{align} \label{eq:3s 1-1 N^* N}   col_i(\mathcal{N})=
\begin{pmatrix}
    1&-1\\
    1&0\\
    0&1
\end{pmatrix}
col_i(\mathcal{N}^*).
\end{align}
So, by \eqref{eq:3s 1-1 N^* N}, 
$col_j(\mathcal{N})=$$(1,1,0)^\top$ and $col_k(\mathcal{N})=$$(-1,0,1)^\top$. 
Hence, by Lemma \ref{lm:0 set reaction}, the network $G$ contains the following two reactions:
\begin{align*}
X_3&\xrightarrow{}X_1+X_2+X_3,\\
X_1+X_2&\xrightarrow{}X_3+X_2.
\end{align*}
Note that $col_j(\mathcal{N})$ and $col_k(\mathcal{N})$ are linearly independent. Hence, all columns of $\mathcal{N}$ can be generated by $col_j(\mathcal{N})$ and $col_k(\mathcal{N})$.  So, by the fact that the network is maximum, the network can be written as follows:
\begin{table}[H]
    \centering
        \label{eq:3s lambda7 network}
    \begin{tabular}{lll}
        $X_3\xrightleftharpoons[\kappa_2]{\kappa_1}X_1+X_2+X_3$, &  $X_1+X_2\xrightleftharpoons[\kappa_6]{\kappa_5} X_3+X_2$,
         & $X_1\xrightleftharpoons[\kappa_{10}]{\kappa_{9}} X_2+X_3+X_1$,\\
         $0\xrightleftharpoons[\kappa_4]{\kappa_3}X_1+X_2$,&
         $X_1\xrightleftharpoons[\kappa_8]{\kappa_7} X_3$,&
         $0\xrightleftharpoons[\kappa_{12}]{\kappa_{11}} X_2+X_3$.       
    \end{tabular}
\end{table}
\noindent
Hence, for the steady-state system $f$ defined as in \eqref{eq:sys}, we have 
\begin{align}
    f_2=-\kappa_2x_1x_2x_3-\kappa_{10}x_1x_2x_3-\kappa_4x_1x_2-\kappa_{12}x_2x_3+\kappa_9x_1+\kappa_1x_3+\kappa_3+\kappa_{11}. \nonumber
\end{align}
Below, for any $c\in {\mathbb R}$, we prove that there exists a number $R>0$ such that for any $x^*=(x_1^*, x_2^*, x_3^*)\in \mathcal{P}_{c}$ satisfying $x_2^*>R$, we have $f_2(x^*)<0$. 
Note that the conservation law is $x_1=x_2-x_3+c$. 
Thus, there exists $R>0$ such that for any $x^*\in \mathcal{P}_{c}$ satisfying $x_2^*>R$, we have 
$x_1^*+x_3^*=x_2^*+c>0$. Hence, $x_1^*$ and $x_3^*$ can not be $0$ simultaneously. 
Thus, by \eqref{eq:3s lambda7 f2},  if the number $R$ is large enough, then  we have $f_2(x^*)<0$. 
By Lemma \ref{lm:xs-1 large},  for any $R>0$, there exists $M>0$ such that for any $x^*\in \mathcal{P}_{c}$ satisfying $\Vert x^*\Vert_{\infty}>M$, we have $x_{2}^*>R$. 
 So, there exists $M>0$ such that for any 
     $x^*\in \mathcal{P}_{c}$ satisfying $\Vert x^*\Vert_{\infty}>M$, we have $f_2(x^*)<0$. 
Note that $b=-1<0$. 
Thus, by Lemma \ref{lm:g3 omega} and by Lemma \ref{lm:Dissipative}, $G$ is dissipative.

\item Assume the conservation law has the form \eqref{eq:g3 11}. So, by \eqref{eq:3s Ni_Ns_Ns-1}, $\mathcal{N}_1=\mathcal{N}_2+\mathcal{N}_3$. 
Recall that we have assumed that $\mathcal{N}_2$ and $\mathcal{N}_3$ are linearly independent. 
Hence, we get $\mathcal{N}_1$ and $\mathcal{N}_3$ are linearly independent. 
Thus, we can exchange the two species $X_1$ and $X_2$ in the network $G$ by relabeling the species, and we call the new concentration variables as $\tilde{x}_1, \tilde{x}_2, \tilde{x}_3$. 
Hence, by \eqref{eq:g3 11}, the new conservation law is  $\tilde{x}_1=\tilde{x}_2-\tilde{x}_3-c$. 
Therefore, the conclusion follows from  the  case (IV).
    \end{enumerate}
\end{proof}



\begin{lemma}\label{lm:G3 bound}
   For any $G \in \mathcal{G}_3$,  for any $c\in\mathbb{R}$, if $\mathcal{P}_c^+ \neq \emptyset$, then for any $\kappa \in \mathbb{R}_{>0}^m$, the network $G$ has no boundary steady states in $\mathcal{P}_c$.
\end{lemma}
\begin{proof} 
We can list all networks in ${\mathcal G}_3$, see \eqref{eq:g_31}--\eqref{eq:g_36} in Appendix \ref{S1_Appendix}. 
Below, we prove the conclusion for the network \eqref{eq:g_31}. The proof for any other network  is similar. 
Note that the corresponding conservation law is  $x_1=-x_2-x_3+c$.
Since $\mathcal{P}_c^+\neq \emptyset$, we have $c>0$. Hence, $(0,0,0)$ can not be a boundary steady state.
Notice that for the network in \eqref{eq:g_31}, the steady-state system $f$ defined  in \eqref{eq:sys} is given by the following polynomials: 
\begin{align}
f_1&=-\kappa_5x_1x_2-\kappa_3x_1x_3+\kappa_4x_2x_3+\kappa_6x_2x_3-\kappa_9x_1-\kappa_{11}x_1+\kappa_{10}x_2+\kappa_{12}x_3,\label{eq:g3 1 f1}\\
f_2&=-\kappa_1x_1x_2+\kappa_2x_1x_3+\kappa_3x_1x_3-\kappa_4x_2x_3 +\kappa_9x_1-\kappa_7x_2-\kappa_{10}x_2+\kappa_8x_3,\label{eq:g3 1 f2}\\
f_3&=\kappa_1x_1x_2+\kappa_5x_1x_2-\kappa_2x_1x_3-\kappa_6x_2x_3+\kappa_{11}x_1+\kappa_7x_2-\kappa_8x_3-\kappa_{12}x_3. \label{eq:g3 1 f3}
\end{align}
Assume that $x=(x_1,x_2,x_3)\in\mathbb{R}_{\ge0}^3$. 
If $x_1=0$, $x_2>0$ and $x_3>0$, then by \eqref{eq:g3 1 f1}, we get $f_1(x)>0$. 
Similarly, if $x_1>0$, $x_2=0$ and $x_3>0$, then by \eqref{eq:g3 1 f2}, $f_2(x)>0$,
and if $x_1>0$, $x_2>0$ and $x_3=0$, then by \eqref{eq:g3 1 f3}, $f_3(x)>0$. 
If $x_1=x_2=0$ and $x_3>0$, then by \eqref{eq:g3 1 f2}, $f_2(x)>0$. 
Similarly, if $x_1=x_3=0$ and $x_2>0$, then by \eqref{eq:g3 1 f2}, $f_2(x)<0$, 
and if $x_2=x_3=0$ and $x_1>0$, then, by \eqref{eq:g3 1 f2}, $f_2(x)>0$. 
Hence, for any $c\in\mathbb{R}$, if $\mathcal{P}_c^+ \neq \emptyset$, then the network $G$ admits no boundary steady states in $\mathcal{P}_c$. 
\end{proof}
\noindent
\textbf{proof of Lemma \ref{lm:3s 2d}}\;\; For any  $G \in \mathcal{G}_3$, by Lemma \ref{lm:g3 compact}, $G$ is dissipative. 
For any $c\in \mathbb{R}$, if $\mathcal{P}_c=\emptyset$, then the conclusion holds, and if $\mathcal{P}_c\neq\emptyset$, then we prove as follows. 
Since $\mathcal{P}_c\neq\emptyset$, by Lemma \ref{lm:G3 bound}, $G$ admits no boundary steady states. Let $h$ be the steady-state system augmented by conservation laws defined in \eqref{eq:h}. 
By Lemma \ref{lm:jac_g23}, for any $\kappa\in \mathbb{R}^{m}_{>0}$, if $G$ has a  corresponding positive steady state $x$ in ${\mathcal P}^{+}_{c}$, then ${\rm sign}({\rm det}({\rm Jac}_h(\kappa, x)))=1$.
Thus, by Theorem \ref{thm:deter_sign}, the network has exactly one nondegenerate positive steady state in $\mathcal{P}_c$. 

\noindent
\textbf{proof of Lemma \ref{lm:3-species_maximum}}\;\; Recall that $\mathcal{G}$ is the set of all the two-dimensional maximum three-species networks, see \eqref{eq:maximum networks}, and  $\mathcal{G}_1$, $\mathcal{G}_2$ and $\mathcal{G}_3$ defined in \eqref{eq:condition2}--\eqref{eq:condition3} give a partition of $\mathcal{G}$. So, the conclusion follows from Lemma \ref{lm:3s6r}, Lemma \ref{lm:two dim 0110 at most} and Lemma \ref{lm:3s 2d}. 
\subsubsection{Proof of Lemma \ref{lm:3-species_2-dimensional} by inheritance}\label{subsec:2d 3s main pro}
\noindent 
\textbf{proof of Lemma \ref{lm:3-species_2-dimensional}}\;\;
By Lemma \ref{lm:3-species_maximum}, all the two-dimensional maximum three-species networks admit no nondegenerate multistationarity/multistability. 
For any three-species two-dimensional network $G$, by Lemma \ref{lm:jac>0}, either the network $G$ only admits degenerate positive steady states, or for any $\kappa \in \mathbb{R}_{>0}^m$, if the network $G$ has a corresponding positive steady state $x\in \mathbb{R}_{>0}^3$, then we get ${\rm det}({\rm Jac}_h(\kappa,x))>0$. Hence,
if the network $G$ admits a nondegenerate positive steady state, then by Lemma \ref{lm:3-species_maximum} and by the inheritance of nondegenerate multistationarity \cite[Theorem 1]{BP16}, we get $G$ admits no nondegenerate multistationarity/multistability.  In conclusion, either $G$ only admits degenerate positive steady states, or $G$ admits at most one nondegenerate positive steady state.


\section{Supporting information}\label{sec:appendix}

\subsection{S1 Appendix.}\label{S2_Appendix}
In this section,  we first present a necessary and sufficient condition for  a nondegenerate steady state of a two-dimensional network to be stable, see Lemma \ref{lm:Hurwitz}. 
After that,  based on Hurwitz criterion (see Lemma \ref{lm:criterion_Hurwitz}), we prove a necessary condition for  a steady state of a three-dimensional network with three species to be stable, see Lemma \ref{lm:33stable}. 
Remark that Lemma \ref{lm:Hurwitz} and Lemma \ref{lm:33stable}  are used in the proofs of Lemma \ref{lm:zero-one stable} in Section \ref{sec:trans} and Theorem \ref{thm:multistab} in Section \ref{sec:results}.

\begin{lemma}\label{lm:Hurwitz}
   Consider a two-dimensional network $G$. Let $f_1,\dots,f_s$ be the polynomials defined in \eqref{eq:sys}. Let $h$ be the steady-state system augmented by conservation laws defined in \eqref{eq:h}. For any $\kappa\in \mathbb{R}^{m}_{>0}$, and for any corresponding nondegenerate  steady state $x \in \mathbb{R}^s$, $x$ is stable if and only if 
   \begin{align} \sum^s_{i=1}\frac{\partial f_i}{\partial x_i}(\kappa,x)<0, \text{ and }  {\rm det}({\rm Jac}_h(\kappa,x))>0. \nonumber
   \end{align}
\end{lemma}
\begin{proof}
We denote by $E$ an $s\times s$ identity matrix.
The characteristic polynomial of ${\rm Jac}_f$ can be written as $
        {\rm det}(\lambda E-{\rm Jac}_f)=\lambda^s-a_1\lambda^{s-1}+a_2\lambda^{s-2}+\cdots+(-1)^{s}a_{s}$,
where $a_i$ $(i\in\{1,\dots,s\})$ is the sum of all $i$-th order principal subformulas of ${\rm Jac}_f$. Note that the rank of ${\rm Jac}_f$ is two. Thus, for any $i\in\{3,\dots,s\}$, we have $a_i=0$. Then, we have
\begin{equation}
    \begin{aligned}
        {\rm det}(\lambda E-{\rm Jac}_f)&=\lambda^s-\lambda^{s-1}\sum^s_{i=1}\frac{\partial f_i}{\partial x_i}+\lambda^{s-2}\sum_{I\subseteq\{1,\dots,s\},\vert I\vert=2}{\rm det}({\rm Jac}_f[I,I]).
    \end{aligned} \nonumber
\end{equation}
By \cite[Proposition 5.3]{WiufFeliu_powerlaw},  we have 
\begin{equation}\label{eq:det_char_jacf}
     {\rm det}(\lambda E-{\rm Jac}_f)=\lambda^{s-2}(\lambda^2-\lambda \sum^s_{i=1}\frac{\partial f_i}{\partial x_i}+ {\rm det}({\rm Jac}_h)). \nonumber
\end{equation}
So, for any $\kappa^*\in \mathbb{R}^{m}_{>0}$, and for any corresponding nondegenerate   steady state $x^* \in \mathbb{R}^s$,
all
non-zero eigenvalues of ${\rm Jac}_f(\kappa^*, x^*)$  have negative real parts (i.e., $x^*$ is stable) if and only if 
$\sum^s_{i=1}\frac{\partial f_i}{\partial x_i}(\kappa^*,x^*)<0$ and ${\rm det}({\rm Jac}_h(\kappa^*,x^*))>0$.
\end{proof}
\begin{lemma} \label{lm:criterion_Hurwitz}\cite[Criterion 1]{AE2021}
Let $q(z)=b_sz
^s + b_{s-1}z
^{s-1}\cdots + b_0$ be a real polynomial with $b_s>0$ and
$b_0\neq 0$. The Hurwitz matrix ${\mathcal H} = (H_{ij})$ associated with $q(z)$ has entries $H_{ij} = b_{s - 2i+j }$ for
$i, j \in\{ 1, \dots , s\}$ by letting $b_k = 0$ if $k \notin  \{0,\dots, s\}$:
\begin{equation}
\begin{aligned}
{\mathcal H}=\begin{pmatrix}
b_{s-1}&b_s&0&0&\cdots&0\\
b_{s-3}&b_{s-2}&b_{s-1}&b_s&\cdots&0\\
\vdots&\ddots&\vdots&\vdots&\ddots&\vdots\\
0&0&0&b_{6-s}&\cdots&b_2\\
0&0&0&0&\cdots&b_0\
\end{pmatrix}\in\mathbb{R}^{s\times s}.
\end{aligned} \nonumber
\end{equation}
The $i$-th Hurwitz determinant is defined to be ${\mathcal H}_i = {\rm det}({\mathcal H}[I,I])$, where $I = \{ 1, \dots, i\}$. Then, all
roots of $q(z)$ have negative real parts if and only if ${\mathcal H}_i > 0$ for all $i \in\{ 1,\dots , s\}$. 
\end{lemma}
\begin{lemma}\label{lm:33stable}
Consider a three-dimensional network $G$ 
 with three species. Let $f$ be the steady-state system defined as in \eqref{eq:sys}. For any $\kappa\in \mathbb{R}^{m}_{>0}$ and for any corresponding  steady state $x \in \mathbb{R}^3$, if $x$ is stable, then ${\rm det}({\rm Jac}_f(\kappa,x))<0$.
\end{lemma}
    \begin{proof} 
Assume that the characteristic polynomial of ${\rm Jac}_f(\kappa,x)$ is $$
\lambda^3+b_1\lambda^{2}+b_2\lambda+b_{3}.$$
By Lemma \ref{lm:criterion_Hurwitz}, the Hurwitz matrix ${\mathcal H}$ is equal to
\begin{equation}
 \begin{aligned}
     \begin{pmatrix}
       b_1&1&0\\
       b_3&b_2&b_1\\
         0&0&b_3
     \end{pmatrix}.
 \end{aligned}  \nonumber
\end{equation}
By Lemma \ref{lm:criterion_Hurwitz}, all non-zero eigenvalues of ${\rm Jac}_f(\kappa,x)$ have negative real parts if and only if $b_1>0$, $b_1b_2-b_3>0$, and $b_3(b_1b_2-b_3)>0$. 
Note that $b_{3}=-{\rm det}({\rm Jac}_f(\kappa,x))$. 
Thus, if $x$ is stable, then $b_3>0$ (i.e., ${\rm det}({\rm Jac}_f(\kappa,x))<0$).
\end{proof}
\subsection{S2 Appendix} \label{subsec:real}
In this section, we introduce the theoretical foundation (see Lemma \ref{lm:same number a section}) for the method of real root classification in computational real algebraic geometry.  
We remark that Lemma \ref{lm:same number a section} is a corollary of Ehresmann’s theorem \cite{CS1992} for which there exist semi-algebraic statements since 1992. We remark that   Lemma 
\ref{lm:same number a section} is used to prove the monostationarity of two classes of zero-one networks in Section \ref{subsubsection g1g2g3}.

First, we review some useful notions in computational algebraic geometry. 
Let $u:=(u_1,\ldots,u_n)$ and $x:=(x_1,\ldots,x_s)$. 
We define $\pi$ as the \defword{canonical projection}:
 ${\mathbb C}^{n}\times {\mathbb
C}^{s}\rightarrow{\mathbb C}^{n}$ such that
for every $(u, x)\in {\mathbb C}^{n}\times {\mathbb
C}^{s}$, 
$\pi(u, x)=u$.
For any finite polynomial set  $g:=\{g_1, \ldots, g_s\}\subseteq\Q[u,x]$, we define
\begin{align}
    {\mathcal V}(g)\;:=\;\{(u, x)\in {\mathbb C}^{n}\times {\mathbb C}^{s}\mid g_1(u, x)=\cdots=g_s(u,x)=0\}. \nonumber
\end{align}  
A finite polynomial set $g\subseteq\Q[u,x]$
is called a \defword{general zero-dimensional system} if there exists an affine variety $V\subsetneq {\mathbb C}^n$ such that for any
$u\in {\mathbb C}^n\backslash V$, the cardinality  of the set $\pi^{(-1)}(u)\cap \mathcal{V}(g)$ is a finite number (i.e., 
the equations $g_1(u, x)=\cdots=g_s(u,x)=0$ have finitely many common solutions for $x$ in ${\mathbb C}^s$). 

The \defword{set of nonproperness} of $g$, denoted by ${\mathcal V}_{\infty}(g)$, is defined  as the set of the $u\in \overline{\pi({\mathcal V}(g))}$ such that there does not exist a compact neighborhood $U$ of $u$ where
$\pi^{-1}(U)\cap {\mathcal V}(g)$ is  compact \cite{DV2005}.
 Geometrically,   ${\mathcal V}_{\infty}(g)$ is the set of parameters $u$ such that the  equations $g_1=\dots=g_s=0$ have some solution $x$ with coordinates tending to infinity.
 By \cite[Lemma 2 and Theorem 2]{DV2005}, ${\mathcal V}_{\infty}(g)$ is an algebraically closed set and can be computed by Gr\"obner bases. 
 We define
$\mathcal{V}_{J}(g)\;:=\;
\overline{\pi(\mathcal{V}(g)\cap \mathcal{V}(Jac_g))}$.
Geometrically, ${\mathcal V}_{J}(g)$ is the closure of the union of the projection of the singular
locus of ${\mathcal V}(g)$ and the set of
critical values of the restriction of $\pi$ to the regular locus of ${\mathcal V}(g)$ \cite[Definition 2]{DV2005}. Also, we define
$\mathcal{V}_{0}(g)\;:=\;
\overline{\pi(\mathcal{V}(g)\cap \mathcal{V}(\Pi_{i=1}^sx_i))}.$
 Geometrically, $\mathcal{V}_{0}(g)$ is  the algebraic closure of the  parameters such that the equations $g_1=\cdots=g_s=0$ have the solutions with  zero coordinates.  Based on the above notions, we are prepared to present the following lemma. 
\begin{lemma}\cite[Theorem 2]{XT2017}\label{lm:same number a section}
For any general zero-dimensional system  $g\subseteq\Q[u,x]$, if $\mathcal{C}$ is an open connected component of 
$\mathbb{R}^{n}\setminus(\mathcal{V}_{\infty}(g)\cup \mathcal{V}_{J}(g))$,
then over $u\in\mathcal{C} $, the cardinality of the  set $\pi^{(-1)}(u)\cap \mathcal{V}(g)\cap \mathbb{R}^{n+s}$ is constant. Moreover, if $\mathcal{C}$ is an open connected component of
$\mathbb{R}^{n}\setminus(\mathcal{V}_{\infty}(g)\cup \mathcal{V}_{J}(g)\cup \mathcal{V}_{0}(g)),$
then over $u\in\mathcal{C} $, the cardinality of the  set $\pi^{(-1)}(u)\cap \mathcal{V}(g)\cap (\mathbb{R}^{n}\times \mathbb{R}^{s}_{>0}) $ is constant.
\end{lemma}
\begin{remark}\label{re:discriminant}
    \cite[Definition 1]{DV2005} The variety $\mathcal{V}_{\infty}(g)\cup \mathcal{V}_{J}(g)\cup \mathcal{V}_{0}(g)$ stated in Lemma \ref{lm:same number a section} is called a \defword{discriminant variety} when one wants to classify positive solutions for the general zero-dimensional system $g$.
\end{remark}
\subsection{S3 Appendix.}\label{S1_Appendix}
In this section, we list all the networks in $\mathcal{G}_2\cup \mathcal{G}_3$ (recall that $\mathcal{G}_2$ and $\mathcal{G}_3$ are two classes of maximum three-species networks defined in  \eqref{eq:condition1}--\eqref{eq:condition3}). 
We remark that in the proof of Lemma \ref{lm:jac_g23} in Section \ref{subsubsec:sign jach}, we have applied a computational method to go over all networks in  $\mathcal{G}_2\cup \mathcal{G}_3$. 
 \\
\subparagraph*{The set $\mathcal{G}_2$ consists of the following  networks \eqref{eq:g_21}--\eqref{eq:g_23}.}
\begin{align} 
&X_1+X_2+X_3\xrightleftharpoons[\kappa_2]{\kappa_1}X_1+X_3  
&&X_1+X_2\xrightleftharpoons[\kappa_4]{\kappa_3} X_1 
&&X_2+X_3\xrightleftharpoons[\kappa_{6}]{\kappa_{5}} X_3\label{eq:g_21}\\    &X_2\xrightleftharpoons[\kappa_{8}]{\kappa_{7}} 0 
&&X_1+X_2+X_3\xrightleftharpoons[\kappa_{10}]{\kappa_9}X_2  
&&X_1+X_3\xrightleftharpoons[\kappa_{12}]{\kappa_{11}}0 \notag\\
&X_2\xrightleftharpoons[\kappa_{14}]{\kappa_{13}} X_1+X_3 
&&X_1+X_2+X_3\xrightleftharpoons[\kappa_{16}]{\kappa_{15}}0\notag\\
\notag\\[-0.5\baselineskip]
&X_1+X_2+X_3\xrightleftharpoons[\kappa_2]{\kappa_1}X_1+X_2  
&&X_1+X_3\xrightleftharpoons[\kappa_4]{\kappa_3} X_1 
&&X_2+X_3\xrightleftharpoons[\kappa_{6}]{\kappa_{5}} X_2\label{eq:g_22}\\   &X_3\xrightleftharpoons[\kappa_{8}]{\kappa_{7}} 0    &&X_1+X_2+X_3\xrightleftharpoons[\kappa_{10}]{\kappa_9}X_3 
&&X_1+X_2\xrightleftharpoons[\kappa_{12}]{\kappa_{11}}0  \notag\\ &X_3\xrightleftharpoons[\kappa_{14}]{\kappa_{13}} X_1+X_2 
&&X_1+X_2+X_3\xrightleftharpoons[\kappa_{16}]{\kappa_{15}}0\notag\\
\notag\\[-0.5\baselineskip]
&X_1+X_2+X_3\xrightleftharpoons[\kappa_{2}]{\kappa_{1}}X_1+X_3
&&X_1+X_2\xrightleftharpoons[\kappa_4]{\kappa_3}X_1
&&X_1+X_3\xrightleftharpoons[\kappa_{6}]{\kappa_{5}}X_1 \label{eq:g_23}\\       &X_1+X_2+X_3\xrightleftharpoons[\kappa_{8}]{\kappa_{7}}X_1+X_2
&&X_2+X_3\xrightleftharpoons[\kappa_{10}]{\kappa_9}X_2       
&&X_3\xrightleftharpoons[\kappa_{12}]{\kappa_{11}}0 \notag \\
&X_1+X_2\xrightleftharpoons[\kappa_{14}]{\kappa_{13}}X_1+X_3
&&X_2+X_3\xrightleftharpoons[\kappa_{16}]{\kappa_{15}}X_3
&&X_2\xrightleftharpoons[\kappa_{18}]{\kappa_{17}}X_3\notag\\   
&X_2\xrightleftharpoons[\kappa_{20}]{\kappa_{19}}0\notag
&&X_2+X_3\xrightleftharpoons[\kappa_{22}]{\kappa_{21}}0\notag
&&X_1+X_2+X_3\xrightleftharpoons[\kappa_{24}]{\kappa_{23}}X_1\notag
\end{align}
\subparagraph*{The set $\mathcal{G}_3$ consists of the following networks \eqref{eq:g_31}--\eqref{eq:g_36}.}
\begin{align}
&X_1+X_2\xrightleftharpoons[\kappa_2]{\kappa_1}X_1+X_3
&&X_1+X_3\xrightleftharpoons[\kappa_{4}]{\kappa_3}X_2+X_3
&&X_1+X_2\xrightleftharpoons[\kappa_{6}]{\kappa_{5}}X_2+X_3\label{eq:g_31}\\
&X_2\xrightleftharpoons[\kappa_8]{\kappa_7}X_3
&&X_1\xrightleftharpoons[\kappa_{10}]{\kappa_9}X_2 
&&X_1\xrightleftharpoons[\kappa_{12}]{\kappa_{11}}X_3\notag
\end{align}
\begin{align}
&X_1+X_2\xrightleftharpoons[\kappa_2]{\kappa_1} X_1+X_3        &&X_1\xrightleftharpoons[\kappa_4]{\kappa_3} X_2+X_3
&&X_2\xrightleftharpoons[\kappa_6]{\kappa_5}X_3 \label{eq:g_32} \\
\notag \\[-0.5\baselineskip]
&X_1\xrightleftharpoons[\kappa_2]{\kappa_1}X_2+X_3
&&X_1+X_2\xrightleftharpoons[\kappa_4]{\kappa_3}X_3
&&X_1\xrightleftharpoons[\kappa_{6}]{\kappa_{5}}X_3
\label{eq:g_33}\\ 
&X_1+X_2\xrightleftharpoons[\kappa_8]{\kappa_7}X_3+X_2 
&&0\xrightleftharpoons[\kappa_{10}]{\kappa_{9}}X_2
&&X_1\xrightleftharpoons[\kappa_{12}]{\kappa_{11}}X_2+X_1\notag \\
&X_3\xrightleftharpoons[\kappa_{14}]{\kappa_{13}}X_2+X_3
&&X_1+X_3\xrightleftharpoons[\kappa_{16}]{\kappa_{15}}X_2+X_1+X_3\notag\\
\notag\\[-0.5\baselineskip] 
&X_3\xrightleftharpoons[\kappa_2]{\kappa_1}X_1+X_2 
&&X_1\xrightleftharpoons[\kappa_4]{\kappa_3} X_2+X_3+X_1
&&0\xrightleftharpoons[\kappa_6]{\kappa_5} X_2+X_3\label{eq:g_34}\\
\notag\\[-0.5\baselineskip] 
&X_3\xrightleftharpoons[\kappa_2]{\kappa_1}X_1+X_2+X_3 
&&X_1+X_2\xrightleftharpoons[\kappa_6]{\kappa_5} X_3+X_2
&&X_1\xrightleftharpoons[\kappa_{10}]{\kappa_{9}} X_2+X_3+X_1\label{eq:g_35}\\  
&0\xrightleftharpoons[\kappa_4]{\kappa_3}X_1+X_2
&&X_1\xrightleftharpoons[\kappa_8]{\kappa_7} X_3
&&0\xrightleftharpoons[\kappa_{12}]{\kappa_{11}} X_2+X_3\notag\\
\notag\\[-0.5\baselineskip]        &X_1+X_2\xrightleftharpoons[\kappa_{2}]{\kappa_{1}} X_1+X_3 
&&X_1+X_2+X_3\xrightleftharpoons[\kappa_4]{\kappa_3} X_2
&&X_1+X_2+X_3\xrightleftharpoons[\kappa_{6}]{\kappa_{5}} X_3\label{eq:g_36}\\   
&X_2\xrightleftharpoons[\kappa_8]{\kappa_7}X_3
&&X_1+ X_3\xrightleftharpoons[\kappa_{10}]{\kappa_9}0         &&X_1+X_2\xrightleftharpoons[\kappa_{12}]{\kappa_{11}}0\notag 
\end{align}
\section{Acknowledgments}
We thank Professor Frank Sottile for his nice advice in the BIRS-IMAG workshop ``Positive Solutions of Polynomial Systems Arising From Real-life Applications".
We also thank Professor Balázs Boros for his suggestion on how to efficiently generate small networks in the Joint Annual Meeting KSMB-SMB 2024. 
Finally, we thank Professor Murad Banaji  for generating and providing all three-species three-dimensional five-reaction and six-reaction zero-one networks with considering 
isomorphism and some other equivalences by his own efficient procedure.

\nolinenumbers

%
%
%

\end{document}